\documentclass[10pt]{amsart}
\usepackage{xypic,amscd,amssymb,latexsym}
\usepackage{amsmath}	
\usepackage{graphicx}
\usepackage{pdfsync}
\usepackage{cancel}

\usepackage{pstricks}
\usepackage{pst-poly}

\usepackage[breaklinks=true]{hyperref}

\usepackage{url}

\usepackage{setspace}

\usepackage{changepage}

\begin{document}
\pagenumbering{arabic}

\newtheorem{theorem}{Theorem}[section]
\newtheorem{proposition}[theorem]{Proposition}
\newtheorem{lemma}[theorem]{Lemma}
\newtheorem{corollary}[theorem]{Corollary}
\newtheorem{remark}[theorem]{Remark}
\newtheorem{definition}[theorem]{Definition}
\newtheorem{conjecture}[theorem]{Conjecture}

\def\qed{{\quad \vrule height 8pt width 8pt depth 0pt}}

\newcommand{\cplx}[0]{\mathbb{C}}

\newcommand{\vs}[0]{\vspace{2mm}}

\newcommand{\til}[1]{\widetilde{#1}}

\newcommand{\mcal}[1]{\mathcal{#1}}

\newcommand{\ul}[1]{\underline{#1}}

\newcommand{\ol}[1]{\overline{#1}}

\newcommand{\wh}[1]{\widehat{#1}}


\author{So Young Cho}

\address{Department of Mathematics, Ewha Womans University, 52 Ewhayeodae-gil, Seodaemun-gu, Seoul 03760, Republic of Korea}

\email[So Young Cho]{csy961015@gmail.com}

\author{Hyuna Kim}

\address{Department of Mathematics, Ewha Womans University, 52 Ewhayeodae-gil, Seodaemun-gu, Seoul 03760, Republic of Korea}

\email[Hyuna Kim]{lunamartina@naver.com}

\author[Hyun Kyu Kim]{Hyun Kyu Kim}

\address{Department of Mathematics, Ewha Womans University, 52 Ewhayeodae-gil, Seodaemun-gu, Seoul 03760, Republic of Korea}

\email[Hyun Kyu Kim]{hyunkyukim@ewha.ac.kr, hyunkyu87@gmail.com}

\author{Doeun Oh}

\address{Department of Mathematics, Ewha Womans University, 52 Ewhayeodae-gil, Seodaemun-gu, Seoul 03760, Republic of Korea}

\email[Doeun Oh]{doeun007@hanmail.net}

\numberwithin{equation}{section}

\title[Laurent positivity of quantized canonical bases for quantum cluster varieties]{Laurent positivity of quantized canonical bases for quantum cluster varieties from surfaces}

\begin{abstract}
In 2006, Fock and Goncharov constructed a nice basis of the ring of regular functions on the moduli space of framed ${\rm PGL}_2$-local systems on a punctured surface $S$. The moduli space is birational to a cluster $\mathcal{X}$-variety, whose positive real points recover the enhanced Teichm\"uller space of $S$. Their basis is enumerated by integral laminations on $S$, which are collections of closed curves in $S$ with integer weights. Around ten years later, a quantized version of this basis, still enumerated by integral laminations, was constructed by Allegretti and Kim. For each choice of an ideal triangulation of $S$, each quantum basis element is a Laurent polynomial in the exponential of quantum shear coordinates for edges of the triangulation, with coefficients being Laurent polynomials in $q$ with integer coefficients. We show that these coefficients are Laurent polynomials in $q$ with positive integer coefficients. Our result was expected in a positivity conjecture for framed protected spin characters in physics and provides a rigorous proof of it, and may also lead to other positivity results, as well as categorification. A key step in our proof is to solve a purely topological and combinatorial ordering problem about an ideal triangulation and a closed curve on $S$. For this problem we introduce a certain graph on $S$, which is interesting in its own right.
\end{abstract}

\maketitle

\vspace{-7mm}

\tableofcontents

\section{Introduction}

\subsection{Background: quantum Teichm\"uller theory}

Fock and Goncharov \cite{FG06} \cite{FG09} defined three kinds of {\em cluster varieties}, denoted by the letters $\mathcal{A}, \mathcal{D}, \mathcal{X}$, associated to any given {\em exchange matrix $(\varepsilon_{ij})_{i,j}$}, which is a skew-symmetrizable square matrix with integer entries. Among them is an important class of special cases which are birational to some moduli spaces $\mathcal{A}_{G,S}$, $\mathcal{X}_{G,S}$, $\mathcal{D}_{G,S}$ associated to punctured surfaces $S$ and reductive algebraic groups $G$. Roughly speaking, $\mathcal{A}_{G,S}$ and $\mathcal{X}_{G,S}$ are some versions of moduli spaces of $G$-local systems on $S$. So, a point on one of these moduli spaces consists of a monodromy representation $\pi_1(S)\to G$ of the fundamental group of $S$ into $G$ defined up to conjugation in $G$, satisfying some conditions, together with certain data at the punctures of $S$. Meanwhile, each cluster variety is obtained by gluing affine varieties along birational maps given by explicit formulas that follow a certain pattern, which involves only mutliplication, division, and addition, but not subtraction. Hence for each semifield $K$, e.g. $K = \mathbb{R}_{>0}=$the positive reals, one can ask for the set of $K$-points of a cluster variety. Let $S$ be an oriented punctured surface, say a compact genus $g$ surface minus $s$ points, with $s\ge 1$ and $2-2g-s<0$. In case $G= {\rm SL}_2$ or ${\rm PGL}_2$, the sets of $\mathbb{R}_{>0}$-points of the corresponding cluster $\mathcal{A}$- and $\mathcal{X}$-varieties recover some versions of the classical Teichm\"uller space of $S$
\begin{align*}
\mathcal{A}_{{\rm SL}_2,S}(\mathbb{R}_{>0}) & \cong \mbox{Penner's decorated Teichm\"uller space of $S$}, \\
\mathcal{X}_{{\rm PGL}_2,S}(\mathbb{R}_{>0}) & \cong \mbox{the enhanced Teichm\"uller space of $S$},
\end{align*}
where by Teichm\"uller space we mean the set of all faithful group homomorphisms $\pi_1(S) \to {\rm PSL}_2(\mathbb{R})$ with discrete image, defined modulo conjugation in ${\rm PSL}_2(\mathbb{R})$. The above two versions have certain restrictions on the monodromy around punctures, and some extra data at punctures.

\vs

One of the major achievements of the paper \cite{FG06} is a certain `duality' map
\begin{align}
\label{eq:I}
\mathbb{I} : \mathcal{A}_{{\rm SL}_2,S}(\mathbb{Z}^t) \hookrightarrow \mathcal{O}(\mathcal{X}_{{\rm PGL}_2,S}),
\end{align}
where $\mathbb{Z}^t$ denotes the semifield of tropical integers. The left hand side $\mathcal{A}_{{\rm SL}_2,S}(\mathbb{Z}^t)$ is in bijection with $\mathbb{Z}^n$ as a set for some positive integer $n$, and has a natural geometric realization as the set of all {\em even integral laminations} on $S$. An integral lamination is a collection of nontrivial homotopy classes of non-intersecting closed curves on $S$ with integer weights on curves, satisfying some conditions, and `even' refers to a certain parity condition on weights; see Def.\ref{def:integral_laminations} for a precise definition. Fock and Goncharov naturally assigned to each even integral lamination $\ell$ an element $\mathbb{I}(\ell)$ of $\mathcal{O}(\mathcal{X}_{{\rm PGL}_2,S})$, i.e. a regular function on the moduli space $\mathcal{X}_{{\rm PGL}_2,S}$. For example, if $\ell$ consists of a single loop $\gamma$ not homotopic to a puncture, with positive integer weight $k$, the regular function $\mathbb{I}(\ell)$ is given by the trace of the monodromy along $\gamma^k = \gamma.\gamma.\cdots.\gamma$ ($k$ times winding around $\gamma$), i.e. $\mathbb{I}(\ell)(\rho)={\rm Trace}(\til{\rho}([\gamma^k]))$, $\forall \rho \in \mathcal{X}_{{\rm PGL}_2,S}$, where $\til{\rho} : \pi_1(S) \to {\rm SL}_2$ is a certain lift of the monodromy representation $\rho : \pi_1(S) \to {\rm PGL}_2$. When $\ell$ consists of several non-intersecting non-homotopic loops, then $\mathbb{I}(\ell)$ is defined as the product of these functions for each constituent loop. They showed that these $\mathbb{I}(\ell)$'s are indeed regular, form a $\mathbb{Q}$-basis of the ring $\mathcal{O}(\mathcal{X}_{{\rm PGL}_2,S})$ of all regular functions, and that they also satisfy a number of favorable properties. Recently, using ideas from mirror symmetry, Gross, Hacking, Keel, and Kontsevich \cite{GHKK} constructed a duality map $\mathcal{A}(\mathbb{Z}^t) \hookrightarrow \mathcal{O}(\mathcal{X})$ for more general cluster varieties, which is expected to specialize to the above map $\mathbb{I}$ in eq.\eqref{eq:I}.

\vs

Let us give a little more detail on what is a `regular' function on the moduli space $\mathcal{X}_{{\rm PGL}_2,S}$. By construction, the charts of an atlas of this moduli space are enumerated by {\em ideal triangulations} of $S$. An {\em ideal arc} is a homotopy class of unoriented non-nullhomotopic paths running between punctures, where the two endpoint punctures need not be distinct. An ideal triangulation is a maximal collection of ideal arcs that mutually do not intersect in their interior parts, i.e. may intersect only at punctures. An ideal triangulation divides $S$ into regions called {\em ideal triangles}, each of which is bounded by three not-necessarily distinct ideal arcs. For a chosen ideal triangulation $T$ of $S$, for each ideal arc $e$ constituting it, one assigns a coordinate function $X_e$. The affine variety associated to $T$ is the {\em split algebraic torus} given by the ${\rm Spec}$ of the ring of all Laurent polynomials in the variables $X_e$'s ($e\in T$) with coefficients in $\mathbb{Q}$, and the moduli space $\mathcal{X}_{{\rm PGL}_2,S}$ is birational to the cluster $\mathcal{X}$-variety which is obtained by gluing these tori by some birational maps. It is shown \cite{FG06} that a {\em regular} function on the moduli space $\mathcal{X}_{{\rm PGL}_2,S}$ can be written, for {\em each} ideal triangulation $T$, as a Laurent polynomial in $X_e$'s ($e\in T$) with coefficients in $\mathbb{Q}$; hence a regular function is said to be {\em universally Laurent}. In Teichm\"uller theory, $X_e$ corresponds to the exponential of the {\em shear coordinate function} along the ideal arc $e$, studied by Penner and Thurston in 1980's \cite{P87} \cite{Thurs} \cite{P}. Fock and Goncharov showed that the above mentioned $\mathbb{I}(\ell)$ which can be viewed as a function on the enhanced Teichm\"uller space and which is essentially given by the trace of monodromy, can be written as a Laurent polynomial over $\mathbb{Z}$ in these exponential shear coordinates, for each ideal triangulation $T$. Moreover, in their proof, they explicitly write down the monodromy $\rho(\gamma)$ of each loop in terms of $X_e$'s, and it is manifest that $\mathbb{I}(\ell)$ is a Laurent polynomial in $X_e$'s with {\em positive} integer coefficients.

\vs

Moving forward, as a step toward a deformation quantization of the moduli space $\mathcal{X}_{{\rm PGL}_2,S}$ with respect to a canonical Poisson structure, Fock and Goncharov \cite{FG09} obtained a quantum version of $\mathcal{X}_{{\rm PGL}_2,S}$. For each triangulation $T$, they first deformed the classical ring of regular functions on the torus associated to $T$, i.e. a commutative Laurent polynomial ring, to a family of non-commutative rings, given by the ring of Laurent polynomials in non-commuting variables $\wh{X}_e$'s ($e\in T$) with coefficients being in $\mathbb{Z}[q,q^{-1}]$. Here $q$ is a quantum parameter, which can be thought of as being a formal symbol, where $q=1$ or $q\to 1$ represents the `classical limit'. The new variables satisfy the relations $\wh{X}_e \wh{X}_f = q^{2\varepsilon_{ef}} \wh{X}_f \wh{X}_e$, $\forall e,f\in T$, where $\varepsilon_{ef}$ is an integer encoding the combinatorics of the triangulation $T$. Then they also deformed the classical gluing birational maps to some non-commutative birational maps between the above `non-commutative tori', in a consistent manner. So the result may be thought of as having a `non-commutative' variety, say $\mathcal{X}_{{\rm PGL}_2,S}^q$, where this symbol actually denotes the ring of `quantum' regular functions on this quantized variety, that is, the ring of all elements that can be written as Laurent polynomials in $\wh{X}_e$'s ($e\in T$) with coefficients in $\mathbb{Z}[q,q^{-1}]$, for each $T$. The word `quantum Teichm\"uller space' may vaguely refer to this ring $\mathcal{X}^q_{{\rm PGL}_2,S}$.

\subsection{The main result}

Fock and Goncharov \cite{FG06} conjectured the existence of a quantum version of the classical duality map $\mathbb{I}$ in eq.\eqref{eq:I}, 
$$
\wh{\mathbb{I}}^q : \mathcal{A}_{{\rm SL}_2,S}(\mathbb{Z}^t) \hookrightarrow \mathcal{X}^q_{{\rm PGL}_2,S},
$$
with favorable properties analogous to those satisfied by $\mathbb{I}$, plus the condition that it recovers $\mathbb{I}$ in the classical limit $q\to 1$. One major consequence of this is that it yields a deformation quantization map
\begin{align}
\label{eq:deformation_quantization}
\mathcal{O}(\mathcal{X}_{{\rm PGL}_2,S}) \to \mathcal{X}^q_{{\rm PGL}_2,S}
\end{align}
for the Poisson moduli space $\mathcal{X}_{{\rm PSL}_2,S}$, defined as sending each basis element $\mathbb{I}(\ell)$ to $\wh{\mathbb{I}}^q(\ell)$, for each $\ell \in \mathcal{A}_{{\rm SL}_2,S}(\mathbb{Z}^t)$ \footnote{This simple observation, which has not been emphasized so much in the literature, was found during a discussion of the third author with Carlos Scarinci.}. That is, $\mathcal{O}(\mathcal{X}_{{\rm PGL}_2,S})$ is the algebra of classical observables, and the map \eqref{eq:deformation_quantization} gives an answer to the `quantum ordering problem' of how assign to each classical observable a quantum observable, such that the map does not depend on the choice of an ideal triangulation. This conjecture, which was believed to have much importance, remained open for about 10 years. In \cite{AK}, Dylan Allegretti and Hyun Kyu Kim, the third author of the present paper, constructed one such map $\wh{\mathbb{I}}^q$ for the first time, building on the work of Bonahon and Wong \cite{BW}, and showed that it satisfies many of the desired properties. In particular, for each even integral lamination $\ell$ on $S$, they constructed an element $\wh{\mathbb{I}}^q(\ell)$ which, for each triangulation $T$, can be written as a Laurent polynomial in $\wh{X}_e$'s ($e\in T$) with each coefficient being a Laurent polynomial in $q$ with integer coefficients, which in case $q=1$ coincides with the Laurent polynomial expression of $\mathbb{I}(\ell)$ in variables $X_e$'s ($e\in T$), under the identification $\wh{X}_e \leftrightarrow X_e$. As mentioned earlier, $\mathbb{I}(\ell)$ is a Laurent polynomial in $X_e$'s ($e\in T$) with {\em positive} integer coefficients; so it is a natural question to ask if this positivity phenomenon persists in the quantum version too. 

\vs

Indeed it does, and that is the main result of the present paper.
\begin{theorem}[main result: `Laurent' positivity of Allegretti-Kim quantum elements]
\label{thm:main}
For each even integral lamination $\ell \in \mathcal{A}_{{\rm SL}_2,S}(\mathbb{Z}^t)$ and each ideal triangulation $T$ of an oriented punctured surface $S$, the Allegretti-Kim quantum element $\wh{\mathbb{I}}^q(\ell) \in\mathcal{X}^q_{{\rm PGL}_2,S}$ constructed in \cite{AK}, corresponding to the Fock-Goncharov regular function $\mathbb{I}(\ell) \in \mathcal{O}(\mathcal{X}_{{\rm PGL}_2,S})$, is a Laurent polynomial in the quantum cluster $X$-variables $\wh{X}_e$'s ($e\in T$) with each coefficient being an element of $\mathbb{Z}_{\ge 0}[q,q^{-1}]$, i.e. a Laurent polynomial in $q$ with \ul{\em positive} integer coefficients.
\end{theorem}

In fact, this main theorem holds for a little more general class of surfaces than just punctured surfaces, as appropriate for the theory of cluster varieties. Namely, we allow $S$ to have circular boundary components with marked points (see Def.\ref{def:decorated_surface}). In this introduction, we restrict ourselves to the punctured surfaces, to simplify the discussion.

\vs

Notice that the statement is not obvious. For example, for a given classical expression $\mathbb{I}(\ell) = X_e X_f + 2 X_e X_f^{-1} + X_e^{-1} X_f^{-1}$, there may be many possible quantum expressions that recover the classical one as $q\to 1$, like $q \wh{X}_e \wh{X}_f + 2 \wh{X}_e \wh{X}_f^{-1} + q \wh{X}_e^{-1} \wh{X}_f^{-1}$, or $q \wh{X}_e \wh{X}_f + (q^3+q^{-3}) \wh{X}_e \wh{X}_f^{-1} + q \wh{X}_e^{-1} \wh{X}_f^{-1}$, or even something like $q \wh{X}_e \wh{X}_f + (q^3+q^{-3}+q^5+q^{-5}-2) \wh{X}_e \wh{X}_f^{-1} + q \wh{X}_e^{-1} \wh{X}_f^{-1} + (2-q-q^{-1}) \wh{X}_e^{-1} \wh{X}_f$. The properties of Allegretti-Kim's map $\wh{\mathbb{I}}^q$ which are omitted in the above discussion but are proven in \cite{AK} give good restriction to what $\wh{\mathbb{I}}^q(\ell)$ can be among all such possible quantum expressions, but do not precisely pin down one answer. One can easily see that our Thm.\ref{thm:main} gives quite strong an extra restriction on what $\wh{\mathbb{I}}^q(\ell)$ can be, since terms like $(2-q-q^{-1}) \wh{X}_e^{-1} \wh{X}_f$ are not allowed anymore. We expect that this restriction will help us when studying other properties of $\wh{\mathbb{I}}^q(\ell)$, see \S\ref{sec:further_research}; for example, if we write $\mathbb{I}(\ell)$ as the sum of monomials with coefficient $1$, then we now know that the deformation quantization map $\mathbb{I}(\ell) \mapsto \wh{\mathbb{I}}^q(\ell)$ is a certain term-by-term quantization, replacing each monomial with a quantum monomial with a $q$-power coefficient.

\vs

The problem of Laurent positivity of universally Laurent expressions is more widely known in the case of cluster $\mathcal{A}$ varieties, for they are more directly related to cluster algebras. A classical case is proved by Lee and Schiffler \cite{LS}, and a quantum case is proved by Davison \cite{Davison}. Lee-Schiffler's positivity also follows as a consequence of a result of \cite{GHKK} under some condition.

\vs

We note that the Laurent positivity for quantum regular functions on cluster $\mathcal{X}$ varieties from surfaces, which we proved in the present paper, is closely related to what is called the `strong positivity conjecture', related to `framed BPS states' and `framed protected spin characters' in the physics literature \cite{GMN} \footnote{This is pointed out to the third author by Dylan Allegretti.}. The framed protected spin characters in physics are expected to coincide with the coefficients $\in \mathbb{Z}[q,q^{-1}]$ of monomials in $\wh{X}_e$'s in the above element $\wh{\mathbb{I}}^q(\ell)$; the classical version of such correspondence is partially established in \cite{Al2}. Recently, Gabella \cite{Gabella} constructed a `quantum holonomy' for closed loops on punctured surfaces, in a way which is qualitatively quite different from those of Allegretti-Kim and Bonahon-Wong, using ideas from physics; the equality of Gabella's quantum holonomy and Allegretti-Kim's $\wh{\mathbb{I}}^q(\ell)$ is proved in the joint work \cite{KS} of the third author and Miri Son. One of Gabella's assertions is the Laurent positivity of the coefficients of his quantum holonomy. However, Gabella's proof of positivity \cite[\S6.4]{Gabella} is only very cursory and does not deal with all possible complications which might arise; see \S\ref{subsec:BW} of the present paper. We claim that dealing with such complications is actually the main difficulty, and that only our present work, together with \cite{KS}, provides a sound proof of his positivity assertion. 

\vs

We also note that the (quantum) Laurent positivity was proved in the `disk case' in \cite[Thm.4.7]{Al}.

\vs

As usual for positivity results in general, our main theorem hints to the existence of a categorification of each of the quantized basis element $\wh{\mathbb{I}}^q(\ell)$. We note that a categorification of the classical counterpart $\mathbb{I}(\ell)$ was partially established in \cite{Al2}.

\subsection{Turning into a topological and combinatorial ordering problem}

To explain our approach to proof of the main theorem Thm.\ref{thm:main}, we first review Allegretti-Kim's construction \cite{AK} of $\wh{\mathbb{I}}^q(\ell)$, which uses Bonahon-Wong's work \cite{BW}, which related the {\em skein algebra} of a punctured surface $S$ to the quantum Teichm\"uller space $\mathcal{X}^q_{{\rm PGL}_2,S}$. A caveat is that the discussion here is a short survey which is not completely precise but is meant to give the readers a rough idea only. Precise notations and constructions can be found in \S\ref{sec:applying} of the present paper.

\vs

The skein algebra $\mathcal{S}^A(S)$, for a parameter $A \in \mathbb{C}^*$,  is generated by {\em skeins}, which are isotopy classes of framed links in the three-dimensional space $S \times [0,1]$, satisfying some conditions. A link is a disjoint union of finitely many non-intersecting closed curves, and a framing on a link is a continuous choice of a tangent vector to $S\times [0,1]$ at each point of the link, so that the vector does not live in the tangent space to the link. One can thus view a framed link as being `a link that knows how much it is twisted', or a link with thickness, e.g. a `ribbon link'. Multiplication of two skeins is defined as vertically stacking one over the other, modded out by certain relations called the `skein relations', in which the parameter $A$ appears. For each skein, first deform the framing to `upward vertical framing', and project down the framed link to $S$; this way one can record a skein as a diagram on $S$, with `crossings' which indicate the different {\em elevations} $\in[0,1]$ of segments. A skein having crossings can be `resolved' to linear combination of skeins without crossings, with the help of skein relations.

\vs

For each skein, and for each chosen ideal triangulation $T$ of $S$, Bonahon and Wong constructed an algorithm to obtain a Laurent polynomial in $\wh{Z}_e$'s ($e\in T$) with coefficients in $\mathbb{Z}[\omega,\omega^{-1}]$, where $\omega^4 = q$, and $\wh{Z}_e$ is the square-root quantum variable satisfying $\wh{Z}_e^2 = \wh{X}_e$. They showed that, if we chose a different triangulation $T'$, then the resulting Laurent polynomial in the square-root quantum variables for $T'$ is related to the expression for $T$ via a square-root version of the quantum birational map between the quantum tori constructed in quantum Teichm\"uller theory, mentioned in the previous subsection. Thus, in a sense their (universally) Laurent polynomial expression for a skein is independent of the choice of $T$.

\vs

Let us denote by $\mathcal{Z}^\omega_{{\rm PGL}_2,S}$ the ring of all `quantum functions' that can be written as Laurent polynomials in $\wh{Z}_e$'s ($e\in T$) with coefficients in $\mathbb{Z}[\omega,\omega^{-1}]$, for each triangulation $T$. In particular, $\mathcal{X}^q_{{\rm PGL}_2,S}$ is a subring of $\mathcal{Z}^\omega_{{\rm PGL}_2,S}$. Bonahon and Wong's result \cite{BW} can be written as an algebra map
$$
{\rm Tr}^\omega_S : \mathcal{S}^A(S) \to \mathcal{Z}^\omega_{{\rm PGL}_2,S},
$$
where the parameter $A$ is put to be $\omega^{-2}$. What Allegretti and Kim \cite{AK} did is, given an integral lamination $\ell$ on $S$, for each constituent curve $\gamma$ of weight $1$ not retractible to a puncture, lift it to a skein $[\til{\gamma}]$ by giving it a constant elevation and the upward vertical framing everywhere, and then apply Bonahon-Wong's map to obtain an element ${\rm Tr}^\omega_S([\til{\gamma}])$ of $\mathcal{Z}^\omega_{{\rm PGL}_2,S}$; in particular, we deal with skeins having no crossings in their projected diagrams. Other constituent curves of $\ell$ are dealt with appropriately. So each constituent curve of $\ell$ gets assigned an element of $\mathcal{Z}^\omega_{{\rm PGL}_2,S}$; the quantum element $\mathbb{I}^\omega(\ell)$ is defined to be the product of all these elements. Meanwhile, for an even integral lamination $\ell$, it is shown in \cite{AK} that the element $\mathbb{I}^\omega(\ell)\in \mathcal{Z}^\omega_{{\rm PGL}_2,S}$ constructed this way lies in the subalgebra $\mathcal{X}^q_{{\rm PGL}_2,S}$, as a consequence of parity consideration. The Allegretti-Kim quantum element for an even integral lamination $\ell$ is then defined to be $\wh{\mathbb{I}}^q(\ell):=\mathbb{I}^\omega(\ell) \in \mathcal{X}^q_{{\rm PGL}_2,S}$. Some basic properties of these quantum elements follow immediately from Bonahon-Wong's results, and other important properties were proven separately in \cite{AK}.

\vs

To actually compute the image under the Bonahon-Wong map ${\rm Tr}^\omega_S$ of a constituent loop $\gamma$ not retractible to a puncture, we first choose an ideal triangulation $T$ of $S$; the constituent ideal arcs of $T$ divide $\gamma$ into the {\em loop segments}. We then continuously deform $\gamma$ so that each loop segment connects two distinct ideal arcs. Then, each loop segment$\subset S$ will be lifted to $S\times [0,1]$ at some constant elevation, i.e. in $S\times \{h\}$ for some $h\in [0,1]$; we may choose these elevations to be {\em any} numbers in $[0,1]$, under only one condition that the loop segments over one triangle have mutually distinct elevations. If we really chose elevations at loop segments randomly, then we get into trouble at the {\em junctures}, where the loop $\gamma$ meets ideal arcs of $T$. Each juncture is attached to two loop segments living in two triangles, and if these two loop segments are not given the same elevation, the lifted picture in $S\times [0,1]$ will not be continuous. Meanwhile, a {\em juncture-state} is a choice of sign$\in\{+,-\}$ at each juncture. For each juncture-state, for each ideal arc $e$ of $T$, the net sum of signs will be the power of the variable $\wh{Z}_e$; multiplying all these yield a monomial $\wh{Z}_e^a \wh{Z}_f^b \cdots$. According to some rule, an element in $\mathbb{Z}[\omega,\omega^{-1}]$ is assigned as a coefficient of this monomial, and these monomials are summed over all possible juncture-states, to yield the sought-for image ${\rm Tr}^\omega_S([\til{\gamma}])$. 

\vs

We notice that this coefficient might involve `minus' only when at some juncture of an ideal arc of $T$ there is a discrepancy of elevations of loop segments as mentioned. More precisely, what matters is only the ordering on the set of all loop segments in each triangle, induced by the elevations. We find that a sufficient condition for the coefficients to not involve any minus is that these orderings on loop segments in triangles are {\em compatible} at each ideal arc of $T$, that is, for each ideal arc, the ordering on the junctures of this arc induced by the ordering on the set of attached loop segments from one of the two triangles having this arc as one of their sides coincides with that induced by the ordering on loop segments from the other triangle.  See \S\ref{sec:applying} of the present paper for more details and justification of this assertion.

\vs

The major part of the present paper is devoted to show that this compatibility condition can be fulfilled, which is a purely topological and combinatorial problem:

\begin{theorem}[ordering problem for loop segments]
\label{thm:ordering_problem}
Let $T$ be any ideal triangulation of a punctured surface $S$, and let $\gamma$ be a simple closed curve in $S$ not retractible to a puncture or a point in $S$. Continuously deform $\gamma$ so that the ideal arcs of $T$ divide $\gamma$ into `loop segments', each of which connecting two distinct ideal arcs. 

Then, it is possible to give, for each ideal triangle of $T$, an ordering on the set of all loop segments living in this triangle, so that these orderings for triangles are compatible at each ideal arc of $T$ in the above sense.
\end{theorem}
As explained briefly so far, this theorem implies:
\begin{theorem}[Laurent positivity of some Bonahon-Wong quantum traces]
\label{thm:Laurent_positivity_of_BW}
Let $T,S,\gamma$ as in Thm.\ref{thm:ordering_problem}. Let $[K]$ be the skein in $S$ obtained by lifting $\gamma$ to a constant elevation with constant upward vertical framing. Then ${\rm Tr}^\omega_S([K])$ is a Laurent-positive element of $\mathcal{Z}^\omega_{{\rm PGL}_2,S}$. That is, for each ideal triangulation $T$ of $S$, ${\rm Tr}^\omega_S([K])$ is a Laurent polynomial in the square-root quantum variables $\wh{Z}_e$'s ($e\in T$) with coefficients in $\mathbb{Z}_{\ge 0}[\omega,\omega^{-1}]$. 
\end{theorem}
We note that Thm.\ref{thm:Laurent_positivity_of_BW} easily generalizes to a bordered surface $S$ and any skein $[K]$ in $S$ that is `closed' (i.e. $\partial K = {\O}$) and whose projected diagram in $S$ has no crossings. The statement of Thm.\ref{thm:Laurent_positivity_of_BW} is the most difficult and crucial part of the proof of our main result, Thm.\ref{thm:main}. However, there is one more important step needed for Thm.\ref{thm:main}, regarding the cases when the integral lamination $\ell$ has a constituent curve $\gamma$ that is not retractible to a puncture and has weight $k\ge 1$.  As explained above, in case $k=1$, this constituent curve contributes the factor ${\rm Tr}^\omega_S([\til{\gamma}])$ in the construction of $\mathbb{I}^\omega(\ell)$. In case $k>1$, this constituent curve contributes the factor $F_k({\rm Tr}^\omega_S([\til{\gamma}]))$ to the Allegretti-Kim element $\mathbb{I}^\omega(\ell)$, where $F_k(x) \in \mathbb{Z}[x]$ is the `$k$-th Chebyshev polynomial $F_k$',  defined recursively as $F_0=2$, $F_1 = x$, $F_{k+1} = x F_k - F_{k-1}$; note that the appearance of $F_k$ is not surprising, because ${\rm Tr}(M^k) = F_k({\rm Tr}(M))$ for any $2\times 2$ matrix $M$ with determinant $1$. The Laurent positivity of $F_k({\rm Tr}^\omega_S([\til{\gamma}]))$ is not immediate, for not all the coefficients of $F_k$ are positive. However, it is relatively easy show this, using some properties of $F_k$ and ${\rm Tr}^\omega_S([\til{\gamma}])$; we note that it is already done in the first arXiv version \cite{AK1} of the paper \cite{AK}. 

\subsection{How we solved}

A basic philosophy is to keep turning the problem into another one, so that it is easier to solve than before. The original problem we try to attack is Thm.\ref{thm:ordering_problem}, about how to give orderings on loop segments on each triangle, so that these orderings are compatible at each ideal arc. We consider yet another problem of giving orderings on the junctures on each ideal arc, i.e. the intersection points of the loop $\gamma$ and this ideal arc, so that these orderings are `compatible at each ideal triangle' in a certain sense. We prove that this new problem implies the original. Now, for orderings on junctures of an arc, we first look at all pairs of adjacent junctures. We investigate the orderings on each of these pairs, what these orderings on two-element-sets must satisfy, as a necessary condition for our purpose. Good thing about the ordering on this two-element-set is that it can be conveniently depicted as one inequality symbol $>$ or $<$ written in between the two adjacent junctures, which can be thought of as an orientation on the segment of the ideal arc delimited by an adjacent pair of junctures; call such a segment an {\em inner arc segment}. Later, to recover the actual ordering on the set of all junctures on an arc, per each inner arc segment we also choose a real number too, indicating the `difference' of two endpoint junctures. 

\vs

To solve the desired problem on giving orderings on the junctures on each arc so that these orderings are compatible at triangles, we find that we must study the relationship between the orientation and the `difference' number written on an inner arc segment $i_1$ and those on another inner arc segment $i_2$ that is `connected' to $i_1$ in a triangle via a `region' formed by loop segments. We thus study the regions of triangles divided by loop segments, and how they connect different inner arc segments. Not all the regions are needed, and we just need the ones having at least one inner arc segment in its boundary; we call them {\em narrow regions}. We then construct a special graph on the surface $S$ as follows: each narrow region corresponds to a vertex, and two narrow regions are connected by $m$ edges iff they share $m$ inner arc segments in their boundaries. It turns out that in our case we have $m\in \{0,1\}$. In practice, one can choose any one point in the interior of each narrow region and use it as a vertex, and connect these vertices by an edge that traverses exactly one inner arc segment once and not the loop. We call this graph the \ul{\em regional graph $\mathcal{R}$}; it depends on $S$, $T$, and $\gamma$, of course up to homotopy for the latter two, and each point of $\mathcal{R}$ has valence $1,2$, or $3$. Notice that the inner arc segments are in one-to-one correspondence with the edges of the regional graph $\mathcal{R}$; so we turn the problem into giving orientations and numbers to edges of $\mathcal{R}$, so that it induces orientations and numbers on inner arc segments, which would in turn induce orderings on junctures on each ideal arc, satisfying the desired compatibility. 

\vs

We first find a sufficient condition on the orientations and numbers on edges of $\mathcal{R}$ which would give us the desired result, and then show that it is indeed possible to find a choice of orientations and numbers on edges of $\mathcal{R}$ satisfying this condition. Both of these two tasks require elementary but somewhat arduous and careful arguments, which make use of the properties of $\mathcal{R}$ coming from its topological nature. One strength of our argument is that it is constructive; given any $S$, $T$, and $\gamma$, we provide an algorithm to produce an ordering on loop segments of each triangle so that these ordering are compatible at ideal arcs.

\vs

\noindent{\bf Acknowledgments.} This research was supported by the 2017 UREP program of Ewha Womans University, Department of Mathematics. We thank the referee for helpful comments. \qquad Hyun Kyu Kim: This research was supported by Basic Science Research Program through the National Research Foundation of Korea(NRF) funded by the Ministry of Education(grant number 2017R1D1A1B03030230). H.K. thanks Dylan Allegretti and Thang Le for help, discussion, questions, comments, and encouragements.

\section{Description of the ordering problem}

We shall describe the sought-for Thm.\ref{thm:ordering_problem} more explicitly, and its variants.

\subsection{Basic definitions}
\label{subsec:basic_definitions}

From the literature \cite{AK} \cite{P} \cite{FST} we recall definitions of basic concepts necessary to formulate the problem. No new concept is introduced in this subsection.

\begin{definition}
\label{def:decorated_surface}
A \ul{\em decorated surface} $S$ is a compact oriented surface with boundary, together with the choice of a (possibly empty) collection of distinguished points on the boundary, called the \ul{\em marked points}.
\end{definition}

So $S$ can be thought of as a compact surface of genus $g$ minus $s$ discs, with $m$ marked points on the boundary. The boundary $\partial S$ of $S$ is homeomorphic to disjoint union of circles. 
\begin{definition}
Call ${\rm int}(S) := S\setminus \partial S$ the \ul{\em interior} of $S$.

\vs

A component of $\partial S\setminus\{\mbox{marked points}\}$ not homeomorphic to a circle is called a \ul{\em boundary arc}. Let $N$ be the number of components of $\partial S \setminus \{\mbox{markted points}\}$. Throughout the paper, we assume
\begin{align}
\label{eq:S_condition}
\mbox{$N\ge 3$ in case $g=0$; \quad otherwise $N\ge 1$.}
\end{align}

\end{definition}
Shrink each component of $\partial S$ without a marked point to a \ul{\em puncture}. So, for example, if $S$ had no marked point at all, then after shrinking, $S$ would look like a compact surface of genus $g$ minus $s$ points. 
\begin{definition}
An \ul{\em ideal arc} in $S$ is a homotopy class of unoriented non-self-intersecting paths in ${\rm int}(S)$ running between punctures and marked points, not homotopic to a point of ${\rm int}(S)$, a puncture of $S$, or a boundary arc.
\end{definition}
To be more precise, it can be thought of as a homotopy class of a path in 
$$
{\rm int}(S) \cup(\{\mbox{punctures}\}\cup\{\mbox{marked points}\})
$$
whose endpoints lie in $\{\mbox{punctures}\}\cup\{\mbox{marked points}\}$. The homotopy is taken rel endpoints, and the two endpoints need not be distinct.
\begin{definition}
An \ul{\em ideal triangulation} $T$ of $S$ is a maximal collection of distinct ideal arcs and boundary arcs in $S$ that have simultaneous representative paths that mutually do not intersect except at their endpoints. Members of $T$ are called \ul{\em constituent arcs} of $T$.

\vs

An ideal triangulation $T$ divides $S$ into regions called \ul{\em ideal triangles}.

\vs

An ideal triangle of $T$ is delimited by its \ul{\em sides}, each being a constituent arc of $T$.

\vs

An ideal triangle having only two distinct sides is said to be \ul{\em self-folded}. The `multiplicity two' side of a self-folded triangle, i.e. the `middle' side, is called a \ul{\em self-folded arc}.
\end{definition}
It is well-known that a decorated surface $S$ satisfying \eqref{eq:S_condition} admits an ideal triangulation. For the study of all possible ideal triangulations of $S$, see \cite{P} \cite{FST}. Observe that we are reserving the word `edge' for later use; we only use `arcs' for an ideal triangulation.

\vs

The following somewhat ad-hoc terms are introduced for convenience.
\begin{definition}
\label{def:good_loop}
A \ul{\em good loop} in $S$ is a non-self-intersecting connected closed curve in ${\rm int}(S)$ that is not retractible to a point of ${\rm int}(S)$. In other words, a good loop is a non-contractible simple closed curve in $S$.

\vs

A good loop is said to be \ul{\em peripheral} if it is retractible to a puncture of $S$. 
\end{definition}

For the following definition and throughout the paper, we regard an ideal triangulation of $T$ as being a collection of representative paths of ideal arcs. When necessary we shall allow to continuously deform the paths.
\begin{definition}
\label{def:minimal_position}
A good loop $\gamma$ in $S$ is said to be in a \ul{\em minimal position} with respect to an ideal triangulation of $T$ if the number of intersections of it with the ideal arcs of $T$ is minimal, in the sense that $\gamma$ cannot be continuously deformed so that it has less number of intersections.
\end{definition}
Whenever we deal with a good loop $\gamma$ and an ideal triangulation $T$, we assume that $\gamma$ is in a minimal position. Note that a good loop never meets a boundary arc.

\subsection{The ordering problems}
\label{subsec:the_ordering_problems}

Now we introduce some new notions, in order to formulate our problem.
\begin{definition}
\label{def:basic}
Let $T$ be an ideal triangulation of a decorated surface $S$, and $\gamma$ be a non-peripheral good loop in $S$, in a minimal position with respect to $T$.

\vs

Denote the intersection points of $\gamma$ and the ideal arcs of $T$ by \ul{\em junctures}. 
Junctures divide the loop $\gamma$ into \ul{\em loop segments}. We say that a loop segment \ul{\em connects} the two not-necessarily distinct sides on which the two endpoints of the loop segment live in. Each loop segment is located in a unique \ul{\em corner} of a triangle, delimited by the two sides that this loop segments connects. We say that a loop segment is \ul{\em attached to} each of its endpoint junctures.
\end{definition}
One of the basic facts is that each triangle has three distinct corners, whether or not the triangle is self-folded. One also observes:
\begin{lemma}[basic facts about loop segments]
\label{lem:basic_lemma_on_endpoints_of_loop_segments}
The following hold.

\begin{enumerate}
\item[\rm (1)] The two endpoint junctures of a loop segment are distinct. 

\item[\rm (2)] A loop segment cannot live in a \ul{\em self-folded corner}, i.e. the corner of a self-folded triangle delimited from both sides by the self-folded arc.

\item[\rm (3)] A loop segment always connects two distinct ideal arcs.
\end{enumerate}
\end{lemma}

\begin{proof}
If the two endpoint junctures of a loop segment coincide, then this juncture must live in a self-folded arc of a self-folded triangle, and this loop segment itself forms a peripheral loop, contradicting to $\gamma$ being a non-peripheral loop. Hence part (1) is proved.

\vs

Notice that no loop segment looks like a `half-circle' attached to one ideal arc, bounding a half-disc region; if so, all loop segments living inside the closure of this half-disc region are all half-circles `parallel' to each other, and so $\gamma$ can be homotoped to remove all these half-circles. This means that $\gamma$ possessing such half-circles is not in a minimal position with respect to $T$.

\vs

Now, suppose part (2) is false, i.e. there is a loop segment living in a self-folded corner. As just seen, the two endpoints cannot coincide, so the situation is as in the left of Fig.\ref{fig:loop_segment_in_a_self-folded_triangle}, without loss of generality. Now, suppose that one is traveling along this loop segment towards the indicated direction. The next loop segment must then live inside the shaded region, hence its two endpoints also live in this same self-folded arc. Since this new loop segment cannot be a half-circle, it must go around the puncture and meet the arc as in the right of Fig.\ref{fig:loop_segment_in_a_self-folded_triangle}. Such situation must go on and on and never ends, which is absurd because $\gamma$ must be a simple closed curve. This proves part (2).

\begin{figure}
\hspace{20mm} \includegraphics[width=70mm]{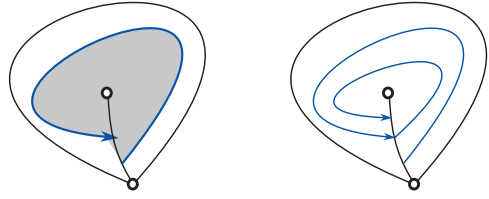} 
\caption{Loop segment in a self-folded triangle}
\label{fig:loop_segment_in_a_self-folded_triangle}
\end{figure}

\vs

Now, suppose part (3) is false, i.e. there is a loop segment whose two endpoints live in one ideal arc. If this arc is not self-folded, then this loop segment must be a half-circle, which we saw is impossible. So this arc must be a self-folded arc, which we saw is impossible by part (2). This proves part (3). 
\end{proof}

\begin{definition}
A \ul{\em triangle-ordering} on an ideal triangle in $T$ is the choice of a total ordering on the set of all loop segments living in this triangle.

\vs

An \ul{\em arc-ordering} on an ideal arc in $T$ is the choice of a total ordering on the set of all junctures living this this arc.

\vs

A triangle-ordering on an ideal triangle naturally \ul{\em induces} an arc-ordering for each of its sides.
\end{definition}

Now, Thm.\ref{thm:ordering_problem} can be rewritten as:

\vs

{\bf Equivalent form of Thm.\ref{thm:ordering_problem}. (triangle-ordering problem)} {\em Let $S,T,\gamma$ as in Def.\ref{def:basic}. There exists a choice of a triangle-ordering on each triangle of $T$ so that for each pair of triangles of $T$, for each side shared by these two triangles, the triangle-orderings of these two triangles induce the same arc-ordering on this common side.}

\vs

We find it difficult to directly attack this problem on triangle-orderings, and thus turn it into a problem on arc-orderings. To do this, we first investigate the relationship between triangle-orderings and arc-orderings. As mentioned already, any triangle-ordering on a triangle uniquely induces an arc-ordering on each of its sides. Now, if we give any arc-orderings on the sides of a triangle, does there exist a triangle-ordering on this triangle inducing the given arc-orderings, and if so, is it unique? Uniqueness is obvious, but existence is not; for this, we shall completely characterize the arc-orderings on the sides that can be induced from a triangle-ordering.

\begin{definition}
\label{def:compatible}
Arc-orderings on a pair of sides of an ideal triangle are said to be \ul{\em compatible} at this triangle if the arc-orderings on these sides induce the same ordering on the set of all loop segments connecting these sides.

\vs

Arc-orderings on the sides of an ideal triangle is said to be \ul{\em compatible} if arc-orderings on each pair of sides of the triangle are compatible.

\vs

Given arc-orderings on the sides of an ideal triangle, a triple of loop segments $j,k,l$ living at three distinct corners is called an \ul{\em insane triple} with respect to these arc-orderings if the ordering on their endpoint junctures on each side is `cyclic', i.e. induced by a clockwise or counterclockwise orientation on the sides of this triangle. See Fig.\ref{fig:three_loop_segments} for an example.

\vs

A choice of arc-orderings on the sides of an ideal triangle is said to be \ul{\em sane} at this triangle if there is no insane triple of loop segments with respect to these arc-orderings.
\end{definition}

\begin{figure}
\hspace{15mm} \includegraphics[width=60mm]{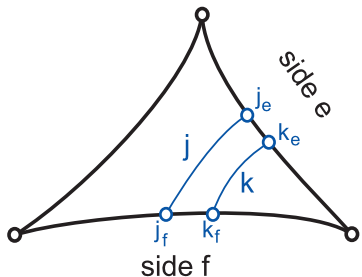} 
\caption{Criterion for compatibility of arc-orderings}
\label{fig:compatibility_criterion}
\end{figure}

\begin{lemma}[criterion for compatibility of arc-orderings]
\label{lem:criterion_for_compatibility_of_arc-orderings}
Arc-orderings on a pair of sides $e,f$ of an ideal triangle is compatible at this triangle if and only if for each two loop segments $j,k$ connecting these two sides, their endpoint junctures $j_e,j_f,k_e,k_f$ as in Fig.\ref{fig:compatibility_criterion} satisfy either $j_e<k_e$ and $j_f<k_f$ simultaneously, or $j_e>k_e$ and $j_f>k_f$ simultaneously. \qed
\end{lemma}

\begin{lemma}
\label{lem:self-folded_is_always_sane}
Any choice of arc-orderings on the sides of a self-folded triangle is sane.
\end{lemma}
\begin{proof} By Lem.\ref{lem:basic_lemma_on_endpoints_of_loop_segments}.(2), there cannot be a triple of loop segments living in all three corners of a self-folded triangle.
\end{proof}

The following is an easy observation:
\begin{lemma}[triangle-ordering to arc-orderings]
\label{lem:triangle-ordering_to_arc-orderings}
Arc-orderings on the sides of an ideal triangulation induced from a triangle-ordering are compatible and sane.
\end{lemma}

\begin{proof}
Compatibility is obvious. For any loop segments $j,k,l$ living in three corners, there is a `smallest' one with respect to the given triangle-ordering, say $j$. In the notations as in Fig.\ref{fig:three_loop_segments}, we have $j_e<l_e$ and $j_f<k_f$, so the orderings on the endpoints of $j,k,l$ is not cyclic. Hence there is no insane triple.
\end{proof}

\begin{figure}
\hspace{15mm} \includegraphics[width=60mm]{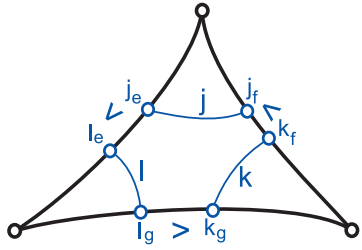} 
\caption{Example of an insane triple of loop segments $j,k,l$}
\label{fig:three_loop_segments}
\end{figure}

\vs

More important is that the converse also holds:

\begin{lemma}[arc-orderings to triangle-ordering]
\label{lem:arc-orderings_to_triangle-ordering}
If arc-orderings on the sides of an ideal triangulation are compatible and sane, then there exists a unique triangle-ordering that induce these arc-orderings.
\end{lemma}

\begin{proof}
Let's first show the existence. We will construct a triangle-ordering that induces the given arc-orderings. We describe an algorithm to assign the numbers $1,2,\ldots,n$ to the loop segments, where $n$ is the total number of loop segments in this triangle, and the numbers represent the ordering. A hypothesis of this algorithm is that we are given arc-orderings on the sides of a triangle that are compatible and sane.

\vs

Step1: On each arc, find the `smallest' juncture, according to the arc-ordering.

\vs

$\leadsto$ Claim1 : two junctures among these three are connected by a loop segment.

~ $\because$ Assume this is not true. For each of these junctures, consider the loop segment in this triangle attached to this juncture; by assumption, these three loop segments are distinct. Suppose first that two of these loop segments live in a same corner. By Lem.\ref{lem:basic_lemma_on_endpoints_of_loop_segments}, this corner is delimited by two distinct arcs, say $e$ and $f$. Now, it is easy to see, from the minimality of the smallest junctures on $e,f$ and from Lem.\ref{lem:criterion_for_compatibility_of_arc-orderings}, that the arc-orderings on $e,f$ is not compatible, which is absurd. Suppose now that all these three loop segments live in distinct corners. By Lem.\ref{lem:basic_lemma_on_endpoints_of_loop_segments}, this triangle must be non-self-folded. By the minimality of the smallest junctures on the arcs, we see that these loop segments form an insane triple of loop segments, which is absurd. (end of proof of Claim1)

\vs

Step2: To the loop segment found by Claim1, assign the smallest number among the numbers in $1,2,\ldots,n$ that are not assigned yet.

\vs

Step3. Erase this loop segment, together with its two endpoint junctures.

\vs

$\leadsto$ Claim2 : The new picture with one loop segment erased inherits arc-orderings on the sides which are compatible and sane.

~ $\because$ Note that the orderings on the set of endpoints of loop segments in each arc in the new picture coincide with those in the previous picture, before erasing one loop segment, because these orderings have nothing to do with the erased segment. 

\quad Suppose not compatible. Then there are two loop segments in the same corner of a new picture, so that the criterion in Lem.\ref{lem:criterion_for_compatibility_of_arc-orderings} fails. 
Then this criterion for these two loop segments also fails in the previous picture, meaning that the previous picture is not compatible, which is absurd. Now, suppose not sane. Then there is an insane triple of loop segments in the new picture. Since the insanity is about the orderings of the endpoint junctures of these three loop segments, these loop segments is also an insane triple in the previous picture. This contradicts to the sanity of the previous picture. (end of proof of Claim2)

\vs

Step4. With this new picture, go to Step1.

\vs

This way we assign $1,2,\ldots,n$ to loop segments, i.e. get a triangle-ordering. Notice that in this process, junctures are erased in the ascending order on each arc. So, at each arc, the $i$-th smallest juncture is connected with the $i$-th erased loop segment among all loop segments that are connected to this arc (not among all loop segments in the triangle). This means that our triangle-ordering induces the given arc-ordering on each arc. 

\vs

For uniqueness, suppose that there is a different triangle ordering that induces the given arc-orderings. Then there are two loop segments whose order between them in our triangle-ordering is different from that in this new one. There is one arc connected to both of these two loop segments, and on this arc, these two triangle-orderings induce different arc-orderings, which is a contradiction.
\end{proof}

\begin{theorem}[arc-ordering problem]
\label{thm:arc-ordering}
Let $S,T,\gamma$ be as in Def.\ref{def:basic}. There exists a choice of an arc-ordering on each ideal arcs of $T$ so that these arc-orderings are compatible and sane at every ideal triangle.
\end{theorem}
Lem.\ref{lem:triangle-ordering_to_arc-orderings} tells us Thm.\ref{thm:ordering_problem}$\Rightarrow$Thm.\ref{thm:arc-ordering}, and Lem.\ref{lem:arc-orderings_to_triangle-ordering} tells us Thm.\ref{thm:arc-ordering}$\Rightarrow$Thm.\ref{thm:ordering_problem}. Hence it is enough to prove this Thm.\ref{thm:arc-ordering}.

\vs

A key idea of our argument came from the consideration of yet another problem, which is easier. Namely, instead of investigating an arc-ordering on an ideal arc, i.e. an ordering on the set of all junctures of an arc, we study the ordering on each pair of junctures that are next to each other in an arc.
\begin{definition}
For each ideal arc having at least one juncture, the junctures on this arc divide this arc into \ul{\em arc segments}. An \ul{\em inner arc segment} is an arc segment bounded by junctures only. See Fig.\ref{fig:arc_segments}. Two junctures living in one ideal arc are said to be \ul{\em adjacent} if and only if they bound an inner arc segment.

\vs

An \ul{\em arc-binary-ordering} on an ideal arc is a choice of an ordering on each pair of adjacent junctures living in this arc, depicted in the picture by the inequality sign $>$ or $<$ drawn on the each corresponding inner arc segment, as if it is an orientation on the inner arc segment.
\end{definition}
The `compatibility', but not the `sanity', of arc-binary-orderings on the sides of a triangle can be defined in a straightforward manner, similarly as for arc-orderings.
\begin{theorem}[arc-binary-ordering problem]
\label{thm:arc-binary-ordering}
Let $S,T,\gamma$ be as in Def.\ref{def:basic}. There exists a choice of an arc-binary-ordering on each ideal arcs of $T$ so that these arc-binary-orderings are compatible at every ideal triangle.
\end{theorem}
Thm.\ref{thm:arc-ordering} obviously implies Thm.\ref{thm:arc-binary-ordering}. Although it is not clear whether Thm.\ref{thm:arc-binary-ordering} implies Thm.\ref{thm:arc-ordering}, this easier problem provides an insight for an approach to Thm.\ref{thm:arc-ordering}, as we shall see in the following section.

\begin{figure}
\hspace{0mm} \includegraphics[width=60mm]{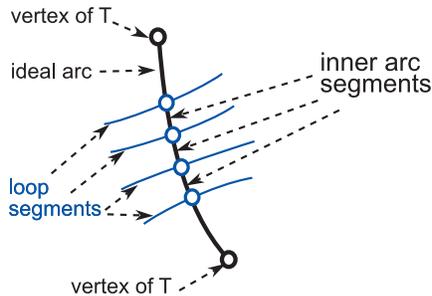} 
\caption{Inner arc segments}
\label{fig:arc_segments}
\end{figure}

\section{Solving the ordering problem}

\subsection{Dyadic arc-orderings}

An arc-ordering on an ideal arc naturally induces an arc-binary-ordering on the arc. This assignment arc-ordering$\mapsto$arc-binary-ordering is onto, but not one-to-one. We now consider a section of this assignment, i.e. a way to construct an arc-ordering from an arc-binary-ordering so that this arc-ordering induces the original arc-binary-ordering.

\vs

First, notice that an {\em ordering} on a set $A$ can be thought of as having an order-preserving injection $\pi : A\to B$, where $B$ is a totally ordered set. For our case $A$ will always be a finite set. When $A$ has $n$ elements, a standard choice of $B$ would be $B = \{1,2,\ldots,n\}$ with the usual ordering. So an ordering on $A$ is an assignment to each element $a\in A$ a number $\pi(a) \in \{1,\ldots,n\}$, so that $a<b$ with respect to the ordering on $A$ if and only if $\pi(a)<\pi(b)$. This choice of $B$ is convenient, because each element $a\in A$ is the $\pi(a)$-th `smallest' element of $A$. Another convenient choice for $B$ is $\mathbb{R}$ with the usual ordering, which we often employ in the present paper. This means that we assign a real number $\pi(a)$ to each element $a\in A$, so that $a<b$ iff $\pi(a)<\pi(b)$. An ordering on $A$ can be represented by several different order-preserving injections $\pi:A\to B$. Conversely, given a set $A$, a totally ordered set $B$, and a map $\pi : A \to B$, one can define a partial ordering on $A$ by declaring $a\le b$ in $A$ iff $\pi(a) \le \pi(b)$; only when $\pi$ is injective, this partial ordering is a total ordering on $A$.

\vs

{\bf Notation for junctures and inner arc segments.} Suppose that an ideal arc has $n$ junctures, labeled by $j_1,j_2,\ldots,j_n$, located on the arc in this order, so that $j_a$ is adjacent to $j_{a+1}$, for each $a=1,\ldots,n-1$; denote by $i_a$ this inner arc segment bounded by the junctures $j_a,j_{a+1}$.

\vs

An arc-ordering on this arc can be represented as a map $\pi : \{j_1,\ldots,j_n\} \to \mathbb{R}$. The arc-ordering lets us compare each two junctures, which we denote by the inequality $j_a < j_b$, which is equivalent to the condition $\pi(j_a)< \pi(j_b)$. That is, $\pi$ assigns a real number to each juncture, allowing us to compare the `size' of junctures, i.e. which one is the biggest, etc. To record the corresponding arc-binary-ordering in the picture, for each $a=1,\ldots,n-1$, we indicate the orientation on the inner arc segment $i_a$ by the symbol $>$ (resp. $<$) written on the arc segment, in case $j_a > j_{a+1}$ (resp. $j_a < j_{a+1}$). So the orientation arrow is directing towards a smaller juncture of the two.

\vs

In addition to the orientation symbol, we also write down the positive real number $|\pi(j_a) - \pi(j_{a+1})|$ on the inner arc segment $i_a$, indicating the `difference'; call this number the \ul{\em difference number} for this inner arc segment. With such orientation with a positive real difference number given on each inner arc segment, one can reconstruct a map $\til{\pi} : \{j_1,\ldots,j_n\} \to \mathbb{R}$ recursively; let $\til{\pi}(j_1)$ be any real number, then define $\til{\pi}(j_2)$ to be the unique real number according to the orientation and the difference number written on the inner arc segment $i_1$ bounded by $j_1$ and $j_2$, then define $\til{\pi}(j_3)$ uniquely, etc. It is easy to see that $\pi$ and $\til{\pi}$ differ only by the overall addition of a single constant.

\vs

Now suppose that we are given an arc-binary-ordering on an ideal arc. That is, on each inner arc segment, an orientation is given. We then would like to choose positive real number for each inner arc segment, and construct a map $\til{\pi}:\{j_1,\ldots,j_n\} \to \mathbb{R}$ as just described. If such constructed $\til{\pi}$ is injective, one obtains a total ordering on the set of juncture $\{j_1,\ldots,j_n\}$, i.e. an arc-ordering. In order to guarantee the injectivity of $\til{\pi}$, we consider the following special way of assigning the difference numbers to inner arc segments.
\begin{lemma}[arc-binary-ordering to arc-ordering]
\label{lem:arc-binary-ordering_to_arc-ordering}
Suppose that an arc-binary-ordering is given on an ideal arc, i.e. an orientation is given on each inner arc segment. Suppose that each inner arc segment is given a positive real number of the form $2^m$ for some nonnegative integer $m$, so that distinct inner arc segments have distinct numbers. Then a map $\til{\pi}: \{j_1,\ldots,j_n\} \to \mathbb{R}$ constructed as above, using these orientations and difference numbers on inner arc segments, is injective, and yields a unique arc-ordering on this arc which induces the original arc-binary-ordering.
\end{lemma}
We will shortly prove this. Such arc-orderings deserve a name, because not all arc-orderings can be obtained this way.
\begin{definition}
\label{def:dyadic}
An arc-ordering that can be obtained in the above situation, i.e. with the difference numbers being distinct $2^m$'s, is said to be \ul{\em dyadic}.
\end{definition}
 For example, in case there are four junctures $j_1,j_2,j_3,j_4$, the ordering given by the map $\pi(j_1) = 2$, $\pi(j_2) = 4$, $\pi(j_3) = 1$, $\pi(j_4) = 3$, is not dyadic. Why isn't this not dyadic? What properties do dyadic arc-orderings have that general arc-orderings do not have? We notice that, for a dyadic arc-ordering, there is a very convenient way to determine the ordering on any two junctures on this arc, as follows:
\begin{lemma}[how to read dyadic arc-ordering]
\label{lem:how_to_read_dyadic_arc-ordering}
Suppose a dyadic arc-ordering is given on an ideal arc; that is, orientations and distinct difference numbers of the form $2^m$ are assigned to inner arc segments. For any two distinct junctures $j_a$ and $j_b$ on this arc, the ordering on these two junctures agrees with the orientation of the inner arc segment whose assigned difference number $2^m$ is the biggest among the ones appearing in between the junctures $j_a$ and $j_b$ on the arc. 
\end{lemma}

\ul{\it Proof of Lemmas \ref{lem:arc-binary-ordering_to_arc-ordering} and \ref{lem:how_to_read_dyadic_arc-ordering}.} Suppose that orientations and distinct difference numbers of the form $2^m$ are assigned to inner arc segments of an ideal arc, whose junctures are $j_1,\ldots,j_n$ located in this order. For each $a=1,\ldots,n-1$, record the orientation on the inner arc segment $i_a$ as the number $\varepsilon_a \in \{-1,+1\}$, so that $\varepsilon_a=1$ indicates $j_a < j_{a+1}$ and $\varepsilon_a=-1$ indicates $j_a > j_{a+1}$. One can view $\varepsilon_a$ as the sign of the difference number $2^{m_a}$ assigned to the inner arc segment $i_a$.

\vs

Then, for a map $\til{\pi} : \{j_1,\ldots,j_n\} \to \mathbb{R}$ appearing in Lem.\ref{lem:arc-binary-ordering_to_arc-ordering}, the difference $\til{\pi}(j_{a+1}) - \til{\pi}(j_a)$ for the adjacent junctures $j_a,j_{a+1}$ equals the signed difference number $\varepsilon_a 2^{m_a}$, for each $a=1,\ldots,n-1$; such is the defining property of $\til{\pi}$. So, if $j_a$ and $j_b$ are any two distinct junctures, say with $a<b$, the difference value $\til{\pi}(j_b) - \til{\pi}(j_a)$ is
$$
\til{\pi}(j_b) - \til{\pi}(j_a) = \textstyle \sum_{c=a}^{b-1} (\til{\pi}(j_{c+1}) - \til{\pi}(j_c)) = \sum_{c=a}^{b-1} \varepsilon_c 2^{m_c} = \varepsilon_a 2^{m_a} + \varepsilon_{a+1} 2^{m_{a+1}} + \cdots + \varepsilon_{b-1} 2^{m_{b-1}}.
$$
Note that $2^{m_a},2^{m_{a+1}},\ldots,2^{m_{b-1}}$, which are the difference numbers appearing in between the junctures $j_a$ and $j_b$, are mutually distinct; let $2^{m_d}$ be the largest among them. Then the absolute value of the sum of the signed differences with $\varepsilon_d 2^{m_d}$ omitted is
\begin{align*}
|\varepsilon_a e^{2m_a} + \cdots + \wh{\varepsilon_d 2^{m_d}} + \cdots + \varepsilon_{b-1} 2^{m_{b-1}}| 
& \le |\varepsilon_a 2^{m_a}| + \cdots + \wh{|\varepsilon_d 2^{m_d}|} + \cdots + |\varepsilon_{b-1} 2^{m_{b-1}}| \\
& = 2^{m_a} + \cdots + \wh{2^{m_d}} + \cdots + 2^{m_{b-1}} \\
& \le 2^0 + 2^1 + 2^2 + \cdots + 2^{m_d -1 } = 2^{m_d} - 1 \\
& < 2^{m_d},
\end{align*}
where the hat $\,\wh{\,\,}\,$ denotes the omitted term, and the second inequality holds because $m_a,m_{a+1},\ldots,\wh{m_d},\ldots,m_{b-1}$ are mutually distinct members among the integers $0,1,2,\ldots,m_d-1$. Since
\begin{align}
\label{eq:partial_sum_of_signed_differences}
\varepsilon_a 2^{m_a} + \cdots + \varepsilon_{b-1} 2^{m_{b-1}}
= \varepsilon_d 2^{m_d} + (\varepsilon_a e^{2m_a} + \cdots + \wh{\varepsilon_d 2^{m_d}} + \cdots + \varepsilon_{b-1} 2^{m_{b-1}}) 
\end{align}
where $(\varepsilon_a e^{2m_a} + \cdots + \wh{\varepsilon_d 2^{m_d}} + \cdots + \varepsilon_{b-1} 2^{m_{b-1}})$ is a number whose absolute value is strictly less than $2^{m_d}$, the sign of \eqref{eq:partial_sum_of_signed_differences} is completely determined by the term $\varepsilon_d 2^{m_d}$. That is, $\varepsilon_a 2^{m_a} + \cdots + \varepsilon_{b-1} 2^{m_{b-1}}>0$ if $\varepsilon_d=1$, and $\varepsilon_a 2^{m_a} + \cdots + \varepsilon_{b-1} 2^{m_{b-1}}<0$ if $\varepsilon_d=-1$. Thus indeed the sign of $\til{\pi}(j_b)-\til{\pi}(j_a)$ is completely determined by the orientation $\varepsilon_d$ at the inner arc segment $i_d$ having the largest difference number $2^{m_d}$ in between the junctures $j_a$, $j_b$. More precisely, if $j_d<j_{d+1}$ then $j_a<j_b$, and if $j_d>j_{d+1}$ then $j_a>j_b$, as desired for Lem.\ref{lem:how_to_read_dyadic_arc-ordering}. 

\vs

In particular, the difference value $\til{\pi}(j_b) - \til{\pi}(j_a)$ is nonzero, for any two distinct junctures $j_a$, $j_b$, proving the injectivity of $\til{\pi}$. As already mentioned before, the signed difference numbers for adjacent junctures completely determine the function $\til{\pi}$ up to overall addition of a single constant. So such $\til{\pi}$ yields a well-defined ordering on the junctures $\{j_1,\ldots,j_n\}$, so that $\til{\pi} : \{j_1,\ldots,j_n\} \to \mathbb{R}$ is order-preserving. Thus Lem.\ref{lem:arc-binary-ordering_to_arc-ordering} is proved. \qed

\vs

One consequence of the above lemma is as follows. Suppose we have a dyadic ordering on an ideal arc, with notations as in the above proof. Among all difference numbers $2^{m_1},\ldots,2^{m_{n-1}}$ we see in this arc, let $2^{m_d}$ be the largest one. Then, by Lem.\ref{lem:how_to_read_dyadic_arc-ordering}, we either have $j_a < j_b$ for all $a=1,2,\ldots,d$ and all $b=d+1,d+2,\ldots,n$, or $j_a > j_b$ for all $a=1,2,\ldots,d$ and all $b=d+1,d+2,\ldots,n$. So the largest inner arc segment $i_d$ partitions the junctures into the `smaller' group $\{j_1,\ldots,j_d\}$ and the `larger' group $\{j_{d+1},\ldots,j_n\}$ (or vice versa). Now, for each group, such separation must occur again, etc. Hence the name `dyadic'. One notices that, for the above mentioned example $\pi(j_1) = 2$, $\pi(j_2) = 4$, $\pi(j_3) = 1$, $\pi(j_4) = 3$, such separation of junctures is not possible. Although the dyadic arc-orderings may seem to form quite a restrictive class of arc-orderings, we find them convenient to handle when constructing and checking various statements, thanks to their special property in Lem.\ref{lem:how_to_read_dyadic_arc-ordering}.

\subsection{Narrow regions, and the regional graph $\mathcal{R}$}

The `easier' problem Thm.\ref{thm:arc-binary-ordering} is about the existence of a compatible choice of orientations on inner arc segments. The compatibility condition of these orientations leads to investigation of how different inner arc segments are `connected' to each other, which inspired the following notions.
\begin{definition}
Let $S,T,\gamma$ be as in Def.\ref{def:basic}. On each ideal triangle of $T$, the loop segments in it divide the triangle into several regions, which we call \ul{\em small regions}. A small region whose boundary contains an inner arc segment is called a \ul{\em narrow region}.

\vs

For a narrow region, each inner arc segment appearing in its boundary is called an \ul{\em end} of the narrow region. A narrow region having $k$ ends in total is called a \ul{\em $k$-end (narrow) region}.
\end{definition}

\begin{figure}
\hspace{0mm} \includegraphics[width=129mm]{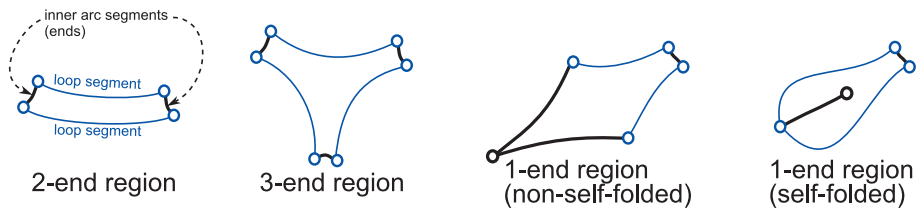} 
\caption{Three kinds of narrow regions (see Fig.\ref{fig:triangle1} and Fig.\ref{fig:triangle2})}
\label{fig:narrow_regions}
\end{figure}

\begin{figure}
\hspace{10mm} \includegraphics[width=60mm]{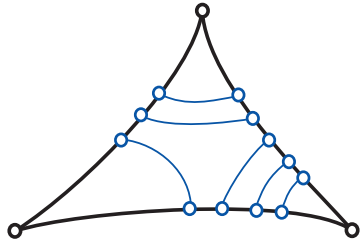} 
\caption{Example of a triangle having loop segments at all corners}
\label{fig:triangle1}
\end{figure}

\begin{figure}
\hspace{10mm} \includegraphics[width=60mm]{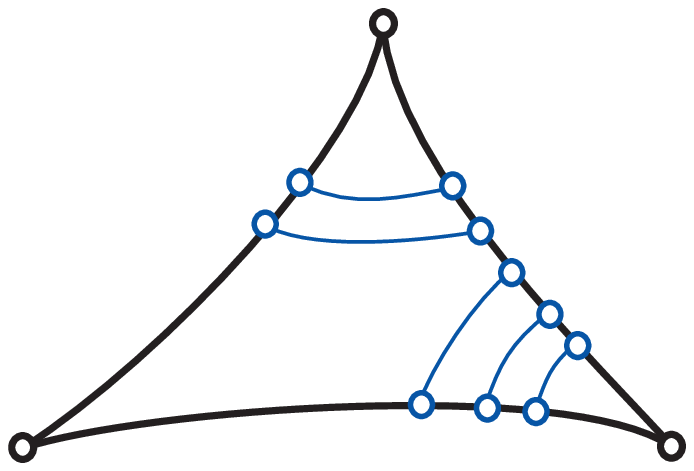} \includegraphics[width=60mm]{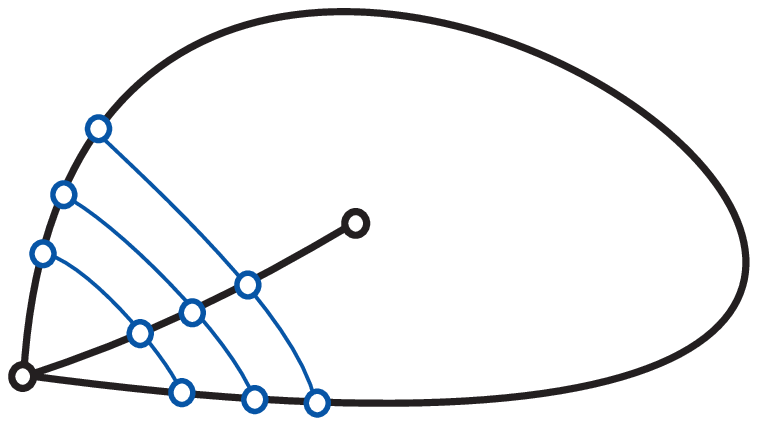} 
\caption{Example of a triangle having loop segments at only two corners}
\label{fig:triangle2}
\end{figure}

We notice that $k$ can be only $1$, $2$, or $3$; so there are three kinds of narrow regions. See Fig.\ref{fig:narrow_regions}. Each narrow region carries the information on how inner arc segments are `connected' in a triangle. We find it necessary to study the relationship between such `connectivity' in adjacent triangles, i.e. how adjacent narrow regions sharing an end are related. Such information is encoded in a graph we define as follows. Recall that an (undirected) \ul{\em graph} consists of a set of \ul{\em vertices} and a set of \ul{\em edges}, where an edge is an unordered pair of vertices, representing a line connecting these vertices.

\begin{definition}
\label{def:regional_graph}
Let $S,T,\gamma$ be as in Def.\ref{def:basic}. The \ul{\em regional graph $\mathcal{R}$} for this data $(S,T,\gamma)$, is the graph defined as follows. The set of vertices of $\mathcal{R}$ is in bijection with the set of all narrow regions. Two distinct vertices of $\mathcal{R}$ are connected by $m$ edges of $\mathcal{R}$ if the corresponding two narrow regions have $m$ inner arc segments in common in their boundaries. We declare that $\mathcal{R}$ has no cycle of length $1$ (i.e. a self-loop).
\end{definition}

It is natural to set that $\mathcal{R}$ has no self-loop, because there is no narrow region such that two of its ends  are identified (i.e. glued); if there is such a narrow region, then these glued inner arc segments must be on a self-folded ideal arc, and one can easily see that in this case one loop segment forms a peripheral loop by itself, which is absurd.

\begin{lemma}
\label{lem:ends_live_in_different_arcs}
For each $k$-end narrow region, the $k$ ends live in $k$ distinct ideal arcs.
\end{lemma}

\begin{proof}
Suppose some two ends of a $k$-end narrow region live in a same ideal arc. Then it follows that there is a loop segment which is a part of the boundary of this $k$-end narrow region that does not connect two distinct ideal arcs. This contradicts to Lem.\ref{lem:basic_lemma_on_endpoints_of_loop_segments}.(3).
\end{proof}

\begin{corollary}
\label{cor:3-end_narrow_region_cannot_occur_in_a_self-folded_triangle}
A $3$-end narrow region cannot occur in a self-folded triangle. \qed
\end{corollary}

\begin{lemma}
The number $m$ in the description of $\mathcal{R}$ (Def.\ref{def:regional_graph}) can only be $0$ or $1$.
\end{lemma}

\begin{proof}
Suppose $m = 2$ for some two vertices of $\mathcal{R}$, corresponding to narrow regions $N, N'$. Each of $N,N'$ has at least two ends which are identified with two ends of the other of the two narrow regions $N,N'$. For each $N,N'$, by Lem.\ref{lem:ends_live_in_different_arcs} these two ends live in distinct ideal arcs. So the two triangles where $N,N'$ live in must have at least two ideal arcs in common, which in particular are not self-folded. So the situation must be like in Fig.\ref{fig:two_triangles_double_glued}, hence some two loop segments form a peripheral loop around the puncture which is the common endpoint of these two ideal arcs. Part of a non-peripheral loop $\gamma$ forming a peripheral loop is absurd. 

\vs

Suppose $m=3$ for some two narrow regions $N,N'$. Each of $N,N'$ is a $3$-end narrow region, and the three ends live in distinct ideal arcs, by Lem.\ref{lem:ends_live_in_different_arcs}. So the two triangles containing $N,N'$ must share all three sides, so these two triangles form the entire triangulation of the surface. The only such decorated surfaces is either the sphere with three punctures, or the once-punctured torus. For the former case, as in the left of Fig.\ref{fig:two_triangles_triple_glued}, at least three pairs of loop segments form peripheral loops, which is absurd. For the latter case, as in the right of Fig.\ref{fig:two_triangles_triple_glued}, these six loop segments form a closed loop hence the whole loop $\gamma$, which by inspection is a peripheral loop, which is absurd.
\end{proof}

\begin{figure}
\hspace{10mm} \includegraphics[width=100mm]{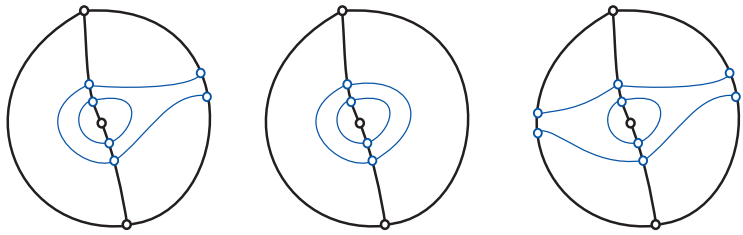} 
\caption{Narrow regions sharing two ends}
\label{fig:two_triangles_double_glued}
\end{figure}

\begin{figure}
\hspace{10mm} \includegraphics[width=50mm]{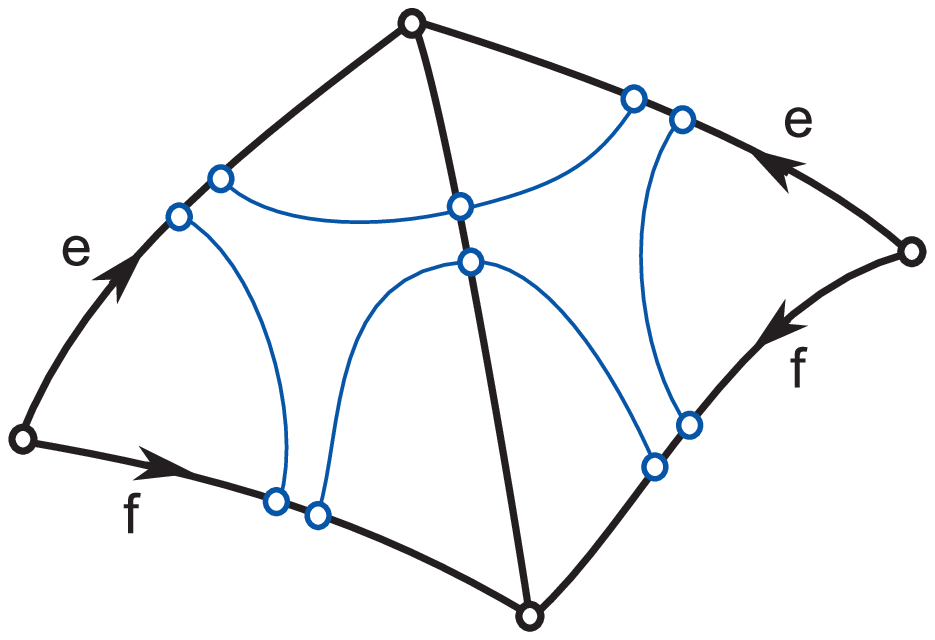}  \hspace{10mm} \includegraphics[width=50mm]{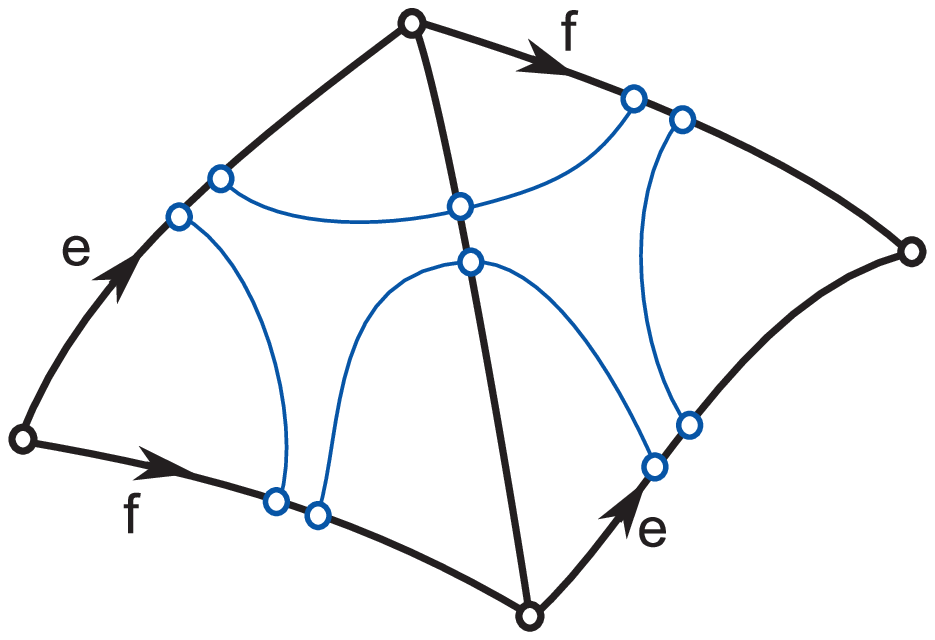} 
\caption{Narrow regions sharing three ends}
\label{fig:two_triangles_triple_glued}
\end{figure}

In practice, it is convenient to give labels to narrow regions, hence accordingly to the vertices of $\mathcal{R}$. See Fig.\ref{fig:regional_graphs} for examples of  regional graph $\mathcal{R}$, in case $S$ is a once-punctured torus or a twice-punctured torus. Notice that the regional graph $\mathcal{R}$ need not be connected.

\vs

It is clear that the \ul{\em valence} of a vertex of $\mathcal{R}$, i.e. the number of edges of $\mathcal{R}$ \ul{\em attached} to this vertex, can be $1,2$, or $3$. The $k$-end narrow region of $S$ corresponds to a $k$-valent vertex of $\mathcal{R}$. Since $\mathcal{R}$ has no self-loop, the $k$ edges attached to a $k$-valent vertex of $\mathcal{R}$ are mutually distinct.

\begin{figure}
\hspace{0mm} \includegraphics[width=100mm]{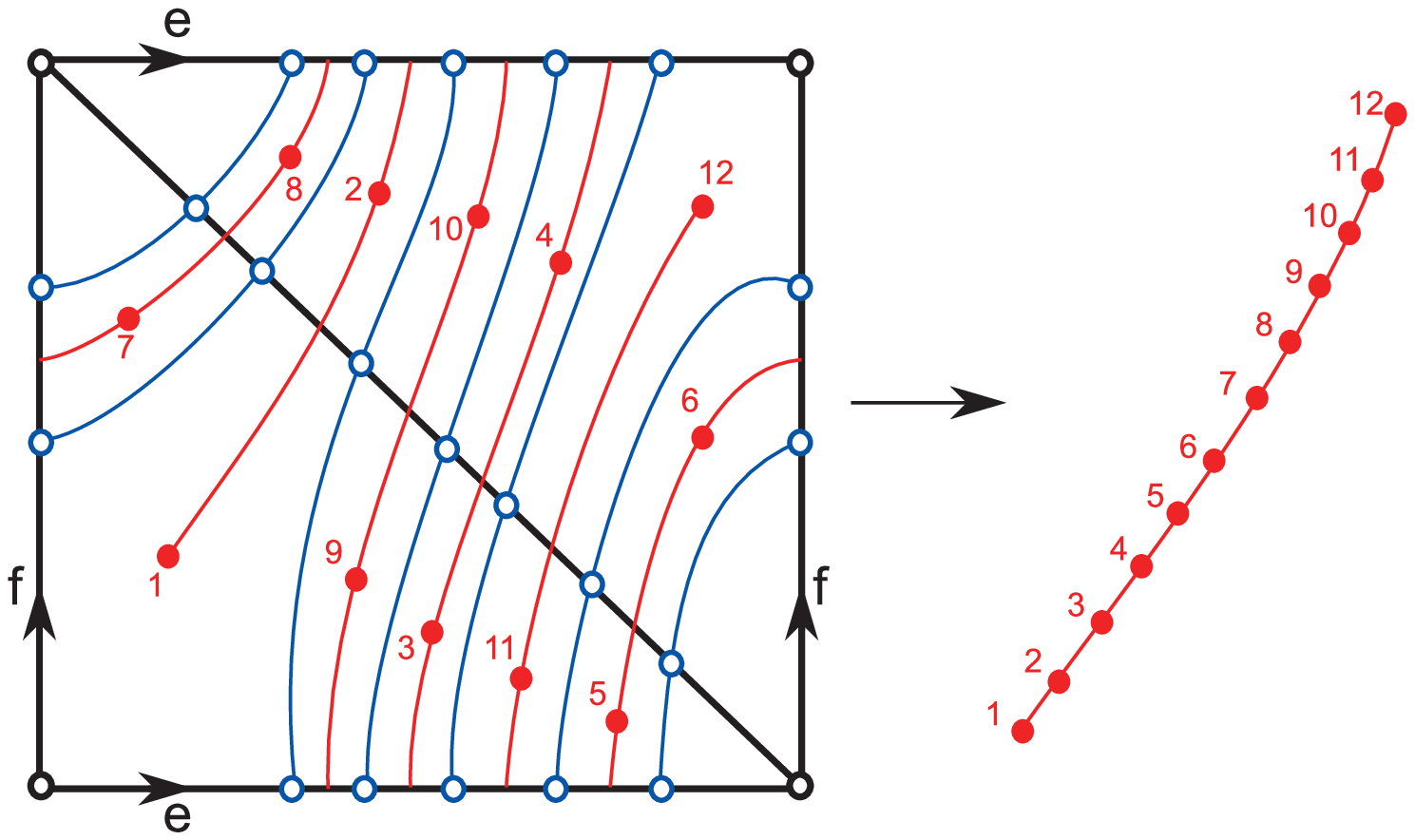}  

\vspace{2mm}

\includegraphics[width=120mm]{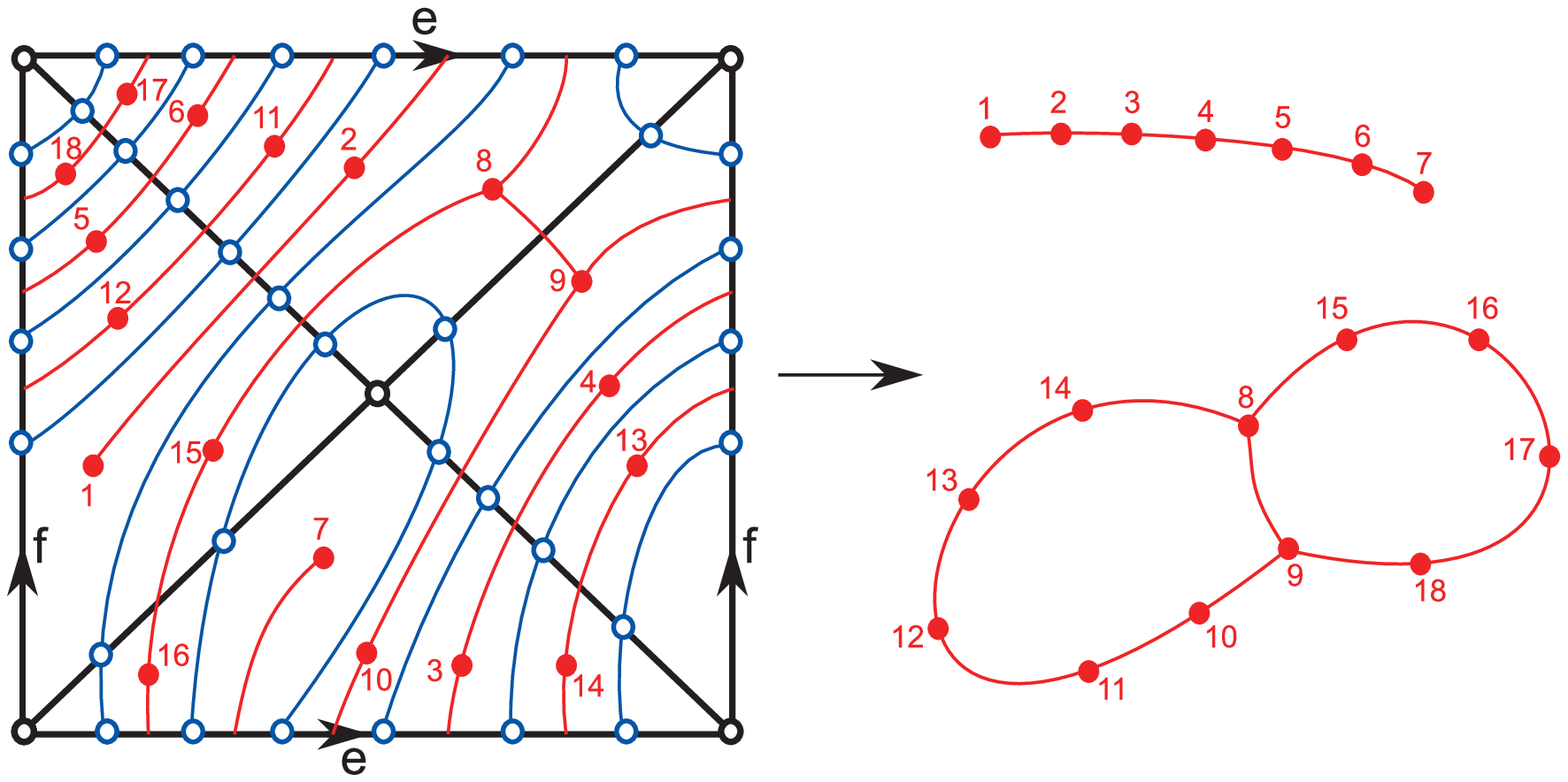} 
\caption{Examples of regional graphs: once-punctured torus and twice-punctured torus}
\label{fig:regional_graphs}
\end{figure}

\vs

We shall use the fact that the graph $\mathcal{R}$ is constructed from an oriented surface $S$. It is helpful to think of $\mathcal{R}$ as living on the surface $S$ as follows. For each narrow region, choose a point in the interior, and use it as a vertex of $\mathcal{R}$. For each inner arc segment, choose a path in $S$ with endpoints being the chosen interior points of the two narrow regions that have this inner arc segment in their boundaries, so that this path traverses this inner arc segment exactly once, and traverses no other arc segment nor the loop $\gamma$; view this path as being an edge of the regional graph $\mathcal{R}$. 

\vs

Notice that the set of all inner arc segments for the data $(S,T,\gamma)$ is naturally in bijection with the set of all edges of the regional graph $\mathcal{R}$. Meanwhile, our strategy to prove Thm.\ref{thm:arc-ordering} is to find dyadic arc-orderings on ideal arcs satisfying the desired conditions of compatibility and sanity. Recall that, to construct a dyadic arc-ordering on an ideal arc is to choose an orientation and a difference number $2^m$ on each inner arc segment. Using the above mentioned bijection, the data on inner arc segments can be transferred to the same kind of data on the edges of $\mathcal{R}$, and vice versa. So, on each edge of $\mathcal{R}$ we shall find a choice of orientation on the edge and a number $2^m$ which we call a \ul{\em weight} on the edge, so that the corresponding data on inner arc segments yield dyadic arc-orderings on the ideal arcs that satisfy the desired conditions.

\vs

We need to fix a concrete way of such `transferring' of data on inner arc segments to/from those on edges of $\mathcal{R}$. The numbers (or weights) $2^m$ can be transferred in an obvious manner, while for the transfer of orientations we make use of the orientation of the oriented surface $S$.

\begin{definition}
\label{def:transfer}
The choice of an orientation and a difference number $2^m$ on an inner arc segment is said to be \ul{\em transferred} from the choice of an orientation and a weight on the corresponding edge of the regional graph $\mathcal{R}$ if

1) the weight on this edge of $\mathcal{R}$ is the same number $2^m$, and

2) the orientation on this edge of $\mathcal{R}$, drawn on the surface $S$, and the orientation on the inner arc segment are as if the orientation on the edge `turns to right' at the intersection of this edge with the inner arc segment. See Fig.\ref{fig:transfer}.
\end{definition}

\begin{figure}
\hspace{0mm} \includegraphics[width=125mm]{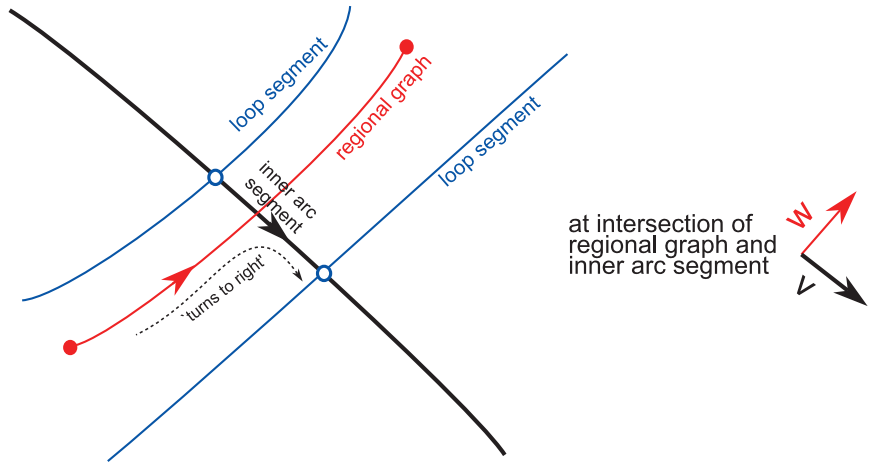}  
\caption{Transfer of orientation from edge of regional graph to inner arc segment}
\label{fig:transfer}
\end{figure}

In 2), the notion of `turns to right' can be made precise, by using the orientation on the surface $S$. Or, condition 2) can be written alternatively as:

\vs

{\em 2') let $p$ be the point of intersection of the inner arc segment and the corresponding edge of $\mathcal{R}$ drawn on $S$, and let $v$ be a positively-oriented basis of the tangent space at $p$ to the inner arc segment, and $w$ a positively-oriented basis of the tangent space at $p$ to the edge of $\mathcal{R}$; we require $\{v,w\}$ to be a positively-oriented basis of the tangent space at $p$ to the surface $S$. See Fig.\ref{fig:transfer}.}

\vs

To be more precise, we must make sure that the inner arc segment and the edge of $\mathcal{R}$ are smooth near $p$. For the notion of `positively-oriented basis of the tangent space' we use the chosen orientation on each relevant (sub)manifold, i.e. the inner arc segment, edge of $\mathcal{R}$, and $S$.

\vs

Notice that there is a unique choice of orientation and difference number on each inner arc segment that is transferred from any given choice of orientation and weight on each edge of $\mathcal{R}$, and also vice versa.

\vs

Before moving on to handle the orientations and weights on edges of $\mathcal{R}$, we study the structure of $\mathcal{R}$. First, recall some basic notions from graph theory:

\begin{definition}
A \ul{\em subgraph} of a graph $\mathcal{G}$ is a graph whose set of vertices is a subset of the set of all vertices of $\mathcal{G}$, and whose set of edges is a subset of the set of all edges of $\mathcal{G}$.

\vs

We say that a graph $\mathcal{G}$ is \ul{\em connected} if any two vertices of $\mathcal{G}$ can be connected by a sequence of edges of $\mathcal{G}$.

\vs

A \ul{\em connected component} of a graph $\mathcal{G}$ is a maximal connected subgraph of $\mathcal{G}$, i.e. a connected subgraph $\mathcal{G}'$ of $\mathcal{G}$ that is not a subgraph of a connected subgraph of $\mathcal{G}$ distinct from $\mathcal{G}'$.
\end{definition}
Any graph $\mathcal{G}$ decomposes into `disjoint union' of its connected components. That is, each vertex or each edge of $\mathcal{G}$ belongs to a unique connected component of $\mathcal{G}$, and two vertices from two different connected components cannot be connected by a sequence of edges. We find it handy to have the following simple lemma, whose proof is a straightforward exercise left to readers. 
\begin{lemma}[a connected component criterion]
\label{lem:component_criterion}
A subgraph $\mathcal{G}'$ of a graph $\mathcal{G}$ is a disjoint union of some connected components of $\mathcal{G}$ if and only if for each vertex $v$ of $\mathcal{G}'$, all edges of $\mathcal{G}$ attached to $v$ belong to $\mathcal{G}'$. \qed
\end{lemma}

For our purpose, we classify the connected components of our regional graph $\mathcal{R}$ as follows.
\begin{definition}
\label{def:types}
A connected component of the regional graph $\mathcal{R}$ is said to be of \ul{\em type I} if it contains a $1$-valent vertex, and \ul{\em type II} otherwise.
\end{definition}
So, each vertex of a type II connected component of $\mathcal{R}$ has valence $2$ or $3$; in a sense, such component has no `open end', and hence is `closed'. The following observation, which we find quite amusing, says that the existence of a type II connected component of $\mathcal{R}$ is a somewhat rare phenomenon, and encodes an interesting topological property of the loop $\gamma$.

\begin{lemma}[implication of the existence of a type II component of $\mathcal{R}$]
\label{lem:type_II_topological_implication}
Let $S,T,\gamma$ be as in Def.\ref{def:basic}, and let $\mathcal{R}$ be the corresponding regional graph. Suppose that $\mathcal{R}$ has a type II connected component, say $\mathcal{R}'$. Then, the union of all narrow regions corresponding to the vertices of $\mathcal{R}'$ is a subsurface of $S$ enclosed by the loop $\gamma$, and this subsurface contains no puncture or a boundary component of $S$.
\end{lemma}

\begin{proof}
Let $S'$ be this subsurface. Let's investigate the boundary of $S'$. Notice that each narrow region corresponding to a vertex of $\mathcal{R}'$ is either a $2$-end region or a $3$-end region; this is because each vertex of $\mathcal{R}'$ is $2$-valent or $3$-valent, since it is of type II. For $k=2,3$, the boundary of a $k$-end narrow region is the union of $k$ inner arc segments, i.e. $k$ ends, and $k$ loop segments; note that this is not true for $k=1$. Since $\mathcal{R}'$ is a connected component, observe from Lem.\ref{lem:component_criterion} that, for each vertex $v$ of $\mathcal{R}'$, every vertex of $\mathcal{R}$ that is connected to $v$ by one edge of $\mathcal{R}$ belongs to $\mathcal{R}'$. So, for each narrow region $N$ constituting the subsurface $S'$, every narrow region in $S$ that shares an end with $N$ is also one of the constituent narrow regions of $S'$. 

\vs

For $k=2,3$, for a $k$-end narrow region, let's call the $k$ ends and $k$ loop segments constituting its boundary the \ul{\em boundary pieces} of this $k$-end narrow region; these boundary pieces are all distinct, thanks to  Lem.\ref{lem:ends_live_in_different_arcs}. So, for $k=2,3$, a $k$-end narrow region has $2k$ boundary pieces. Then $S'$ can be thought of as obtained by gluing some $2$-end regions and $3$-end regions along some of their boundary pieces, where each gluing `glues' two whole boundary pieces of a same kind. In particular, the boundary of $S'$ is the union of some boundary pieces of its constituent narrow regions. Meanwhile, we just saw that, for each $k$-end narrow region $N$ constituting $S'$, each of the $k$ ends among the boundary pieces of $N$ is glued to an (end) boundary piece of some constituent narrow region of $S'$. Hence it follows that the boundary of $S'$ is the union of some loop segments. 

\vs

A similar argument as above also shows that, for each loop segment constituting the boundary of $S'$, each of its two endpoint junctures are glued to an endpoint juncture of a loop segment constituting the boundary of $S'$. Hence the boundary of $S'$ is itself a one-dimensional manifold without ($0$-dimensional) boundary. Since $\gamma$ has only one connected component, its only non-empty one-dimensional submanifold is $\gamma$ itself. Hence the boundary of $S'$ is the whole $\gamma$. Finally, notice that each constituent narrow region of $S'$ has no puncture or part of a boundary component of $S$ in its interior nor on its boundary. So $S'$ does not contain a puncture or a boundary component of $S$. 
\end{proof}

This lemma leads to the following crucial structure result on $\mathcal{R}$.
\begin{corollary}[structure of regional graph $\mathcal{R}$]
\label{cor:structure_of_R}
Let $S,T,\gamma$ be as in Def.\ref{def:basic}, and let $\mathcal{R}$ be the corresponding regional graph. Then the number of type II connected components of $\mathcal{R}$ is at most one.
\end{corollary}

\begin{proof}
Note that cutting the surface $S$ along $\gamma$ yields one or two connected components, which means that either $\gamma$ divides $S$ into two distinct regions, or it does not divide $S$ into distinct regions at all.  Suppose either that $\gamma$ does not divide $S$ into distinct regions, or that $\gamma$ divides $S$ into two distinct regions, each of which contains a puncture or a boundary component of $S$. Then there is no subsurface of $S$ enclosed by $\gamma$ that does not contain a puncture or a boundary component of $S$. Hence, by Lem.\ref{lem:type_II_topological_implication}, $\mathcal{R}$ does not have a connected component of type II. 

\vs

Suppose now that $\gamma$ divides $S$ into two distinct regions, one of which contains no puncture or a boundary component of $S$. Then the other region must contain a puncture or a boundary component of $S$, because the two regions constitute the whole surface $S$ which contains at least one puncture or a boundary component. Hence there is only one subsurface of $S$ enclosed by $\gamma$ having no puncture or a boundary component of $S$. Thus, by Lem.\ref{lem:type_II_topological_implication}, $\mathcal{R}$ has at most one connected component of type II. 
\end{proof}

We will also need the following property of a type II connected component of $\mathcal{R}$:
\begin{lemma}[type II component contains a $3$-valent vertex]
\label{lem:type_II_component_contains_a_3-valent_vertex}
Let $S,T,\gamma$ be as in Def.\ref{def:basic}, and let $\mathcal{R}$ be the corresponding regional graph. Suppose $\mathcal{R}$ has a type II connected component, say $\mathcal{R}'$. Then $\mathcal{R}'$ has at least one $3$-valent vertex.
\end{lemma}

\begin{proof}
As before, let $S'$ be the subsurface of $S$ obtained as the union of all narrow regions corresponding to the vertices of $\mathcal{R}'$. Suppose that the constituent narrow regions of this subsurface $S'$ are all $2$-end narrow regions. Each $2$-end narrow region is homeomorphic to a rectangle, and its four boundary pieces consist of two inner arc segments and two loop segments, the two kinds appearing alternatingly on the boundary. These $2$-end narrow regions are glued along their boundary pieces. Since all these $2$-end narrow regions lie in $S'$ enclosed by the loop $\gamma$, one observes that the gluing among them are always along their inner-arc-segment boundary pieces, and never along their loop-segment boundary pieces. This is because the boundary of the subsurface $S'$ resulting after all the gluing is the entire loop $\gamma$, so no loop segment of $\gamma$ should be missing.

\vs

Pick any one constituent $2$-end narrow region of $S'$, and consider another constituent $2$-end narrow region adjacent to it, sharing an inner arc segment. One thinks of gluing these two narrow regions along their common inner-arc-segment boundary piece. Then glue another adjacent constituent narrow region, etc. By induction, at each step, one observes that the subsurface obtained so far by gluing is either homeomorphic to a rectangle whose boundary consists of inner arc segments and loop segments, or a rectangle with two opposite inner-arc-segment sides are glued, i.e. homeomorphic to a cylinder or a M\"obius strip, in which case no more gluing along inner-arc-segment boundary piece is allowed. However, out of such inductive gluing we must be able to obtain the whole subsurface $S'$, which itself is an oriented $2$-manifold with boundary, having only one boundary component coinciding with the whole loop $\gamma$. So, by the above inductive gluing, it is impossible to obtain $S'$, which is a contradiction.

\vs

Therefore, at least one of the constituent narrow regions of $S'$ is a $3$-end narrow region. Hence the desired claim.
\end{proof}

As we shall soon see, the vertices of $\mathcal{R}$ are dealt with differently, according to whether it belongs to a type I connected component, or to a type II connected component. Hence we label them as follows.
\begin{definition}
\label{def:types_of_vertices}
A vertex of $\mathcal{R}$ is said to be of \ul{\em type I} if it belongs to a type I connected component of $\mathcal{R}$, and \ul{\em type II} if it belongs to a type II connected component of $\mathcal{R}$.
\end{definition}

\subsection{A sufficient condition on orientations and weights on $\mathcal{R}$}
\label{subsec:sufficient_condition_for_data_on_R}

We find one sufficient condition on orientations and weights on the regional graph $\mathcal{R}$ that solves the desired problem, Thm.\ref{thm:arc-ordering}. Without further ado, we state it:
\begin{proposition}[a sufficient condition on orientations and weights on $\mathcal{R}$]
\label{prop:sufficient_condition_for_data_on_R}
Let $S,T,\gamma$ be as in Def.\ref{def:basic}, and $\mathcal{R}$ be the corresponding regional graph. Suppose that an orientation and a weight is assigned to each edge of $\mathcal{R}$, satisfying all of the following conditions:
\begin{enumerate}
\item[\rm 1)] Each edge of $\mathcal{R}$ is assigned a weight $2^m$ for some positive integer $m$, and distinct edges are assigned distinct weights.

\item[\rm 2)] For each $2$-valent vertex of $\mathcal{R}$, the orientations and the weights assigned to the two edges attached to this vertex are as in the left of Fig.\ref{fig:flows}; that is, one incoming with weight $2^m$ for some $m$, and the other outgoing with weight $2^{m+1}$.

\item[\rm 3)] For each $3$-valent vertex of $\mathcal{R}$, the orientations and the weights assigned to the three edges attached to this vertex must satisfy:

$\bullet$ if this vertex is of type I: \quad among the three edges attached to this vertex, there is one incoming with $2^m$ for some $m$, called the \ul{\em flow-in edge}, there is one outgoing with $2^{m+1}$, called the \ul{\em flow-out edge}, and the remaining edge is given either orientation and a weight $2^r$ such that $r<m$, called the \ul{\em left-over edge}. See the right of Fig.\ref{fig:flows}.

$\bullet$ if this vertex is of type II: \quad the three edges attached to this vertex are not all incoming, nor all outgoing; that is, there is at least one incoming one and at least one outgoing one.

\item[\rm 4)] In case $\mathcal{R}$ has a connected component of type I and a connected component of type II, the weight on any edge of any type I connected component is larger than the weight on any edge of a type II connected component.

\end{enumerate}
Then, the orientations and difference numbers on the inner arc segments transferred (Def.\ref{def:transfer}) from the orientations and weights on the edges of $\mathcal{R}$ yield dyadic arc-orderings on the ideal arcs of $T$ that satisfy Thm.\ref{thm:arc-ordering}, i.e. these arc-orderings are compatible and sane at every ideal triangle.

\end{proposition}

\begin{figure}
\hspace{00mm} \includegraphics[width=50mm]{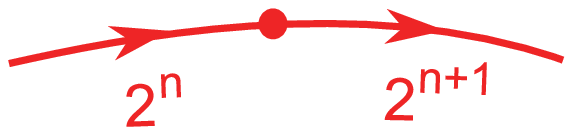} \quad \includegraphics[width=60mm]{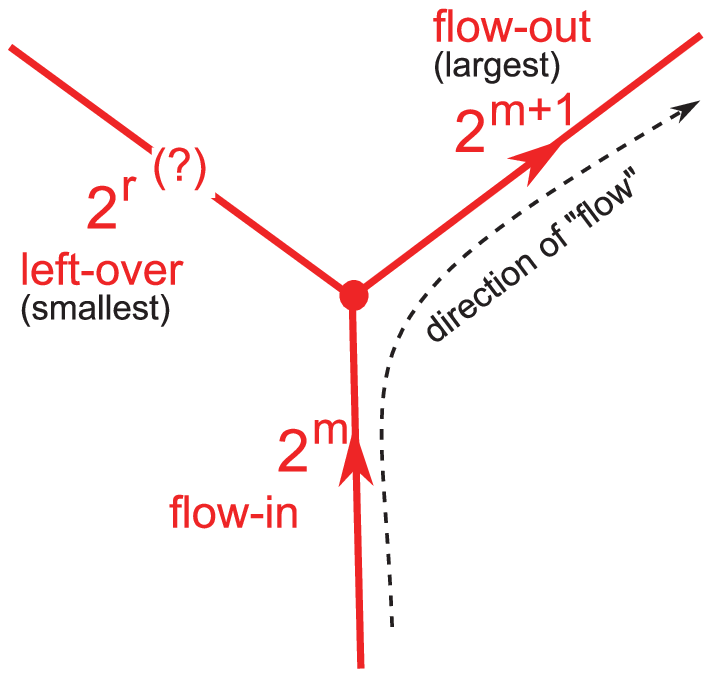}
\caption{sufficient condition on 2-valent vertices and type I 3-valent vertices}
\label{fig:flows}
\end{figure}

\begin{proof}
{\bf [constructing arc-orderings]} Suppose that orientations and weights are given to all edges of $\mathcal{R}$, satisfying the conditions 1), 2), 3), and 4). By Def.\ref{def:transfer}, this data transfers to orientations and difference numbers on all inner arc segments. By condition 1), notice that on each ideal arc, the difference numbers assigned to inner arc segments are mutually distinct and are of the form $2^m$ for positive integers $m$. So by Lem.\ref{lem:arc-binary-ordering_to_arc-ordering}, this data yields a well-defined arc-ordering on each arc, which we called a dyadic arc-ordering (Def.\ref{def:dyadic}). 

\vs

{\bf [compatibility of arc-orderings]} Pick any ideal triangle, and let's check the compatibility of the arc-orderings on its sides. If this triangle has no loop segment, its sides have no junctures, hence no arc-ordering at all, so there's nothing to check. So assume that there is at least one loop segment in this triangle. In view of Lem.\ref{lem:criterion_for_compatibility_of_arc-orderings}, to check the compatibility of arc-orderings in a triangle, we choose any two sides $e,f$ of a triangle, such that there is at least one loop segment connecting these sides; by Lem.\ref{lem:basic_lemma_on_endpoints_of_loop_segments}.(2), $e,f$ must be distinct. Notice that the compatibility for these sides $e,f$ is automatically satisfied if there is only one loop segment connecting these sides; so assume that there are at least two.

\vs

Let $l_1,l_2,\ldots,l_n$ be all the loop segments connecting these sides $e,f$, located in this order `from' the common endpoint vertex of $e,f$. That is, $l_1$ is the closest from the common vertex, and then $l_2$, etc. We just assumed that $n\ge 2$. For each $a=1,\ldots,n$, let $j_a,j'_a$ be the endpoint junctures of the loop segment $l_a$ living in the sides $e,f$ respectively. Then, on side $e$, we have junctures $j_1,\ldots,j_n$ located in this order, and on side $f$ we have junctures $j'_1,\ldots,j'_n$ located in this order. On each $e,f$, there may be more junctures than these ones; however, for each $a=1,\ldots,n-1$, the junctures $j_a,j_{a+1}$ are adjacent to each other in $e$, and $j'_a,j'_{a+1}$ are adjacent to each other in $f$. For each $a=1,\ldots,n-1$, let $i_a$ be the inner arc segment in $e$ bounded by $j_a,j_{a+1}$, and let $i_a'$ be the inner arc segment in $f$ bounded by $j'_a,j'_{a+1}$. On each $i_1,\ldots,i_{n-1}$ an orientation and a difference number is given, yielding a function $\til{\pi}:\{j_1,\ldots,j_n\} \to \mathbb{R}$, which in turn yields an ordering on $\{j_1,\ldots,j_n\}$. Likewise, the orientations and difference numbers on $i_1',\ldots,i'_{n-1}$ yield a function $\til{\pi}' : \{j_1',\ldots,j_n'\}\to \mathbb{R}$, yielding an ordering on $\{j_1',\ldots,j_n'\}$. Compatibility as defined in Def.\ref{def:compatible} is saying that these two orderings are `compatible' in the sense that $j_a < j_b \Leftrightarrow j_a' < j_b'$, for any distinct $a,b\in \{1,\ldots,n\}$ (Lem.\ref{lem:criterion_for_compatibility_of_arc-orderings}). So let's check whether this really holds.

\vs

Pick any distinct $a,b\in \{1,\ldots,n\}$; without loss of generality, $a>b$. See Fig.\ref{fig:compatibility_proof_example} for a non-self-folded example of this situation. Consider the region of this triangle `in between' the loop segments $l_a$ and $l_b$; this region consists of $a-b$ $2$-end narrow regions connecting $e$ and $f$. In this paragraph, we only consider this region, not the whole triangle. Among all the inner arc segments in this region (i.e. appearing in between $l_a$ and $l_b$), i.e. among $i_a,\ldots,i_{b-1}$, $i'_a,\ldots,i'_{b-1}$, find the one having the largest difference number. Without loss of generality, let this `maximal' inner arc segment occur on $e$, i.e. is $i_d$ for some $d\in\{b,\ldots,a-1\}$. Let $2^m$ be the difference number for this $i_d$. It is an end of a $2$-end narrow region connecting $e$ and $f$, the other end being the inner arc segment $i'_d$ in $f$. This narrow region corresponds to a $2$-valent vertex of $\mathcal{R}$, and the inner arc segments $i_d,i'_d$ correspond to the two edges of $\mathcal{R}$ attached to this vertex. The edge of $\mathcal{R}$ corresponding to $i_d$ has weight $2^m$, so by condition 2) the edge of $\mathcal{R}$ corresponding to $i'_d$ must have weight either $2^{m-1}$ or $2^{m+1}$, which is the difference number on $i'_d$. By maximality of $2^m$, it must be $2^{m-1}$, hence is the second biggest among all difference numbers appearing in this region between $l_a$ and $l_b$. Since the difference numbers appearing in this region are all of the form $2^r$ with distinct integers $r$, it follows that $2^m$ is the largest difference number on $e$, and $2^{m-1}$ is the largest one on $f$, in this region. Hence, by Lem.\ref{lem:how_to_read_dyadic_arc-ordering}, the orientation on the corresponding inner arc segment $i_d$ completely determines the ordering on $j_a$ and $j_b$, and that on $i'_d$ determines the ordering on $j'_a$ and $j'_b$. On the other hand, by condition 2) and in view of Def.\ref{def:transfer}, the orientation on $i_d$ and $i'_d$ are `compatible' in the sense that either we have $j_d<j_{d+1}$ and $j'_d<j'_{d+1}$, or we have $j_d>j_{d+1}$ and $j'_d>j'_{d+1}$. Thus, either we have $j_a<j_b$ and $j'_a<j'_b$, or we have $j_a>j_b$ and $j'_a>j'_b$, as desired. This proof also covers the self-folded triangle case.

\vs

\begin{figure}
\hspace{0mm} \includegraphics[width=100mm]{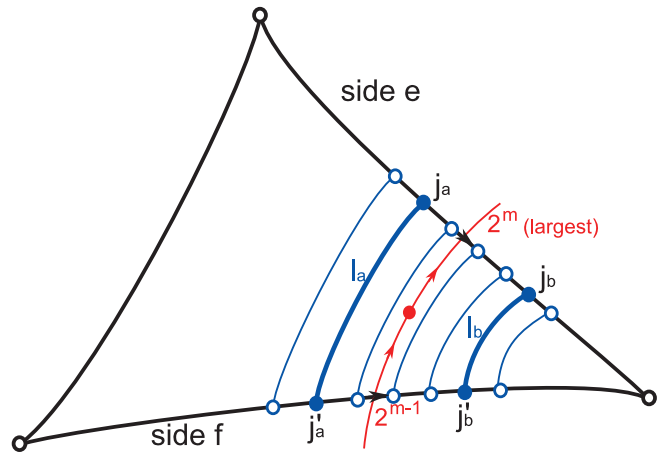} 
\caption{Compatibility check: example situation}
\label{fig:compatibility_proof_example}
\end{figure}

\vs

{\bf [sanity of arc-orderings]} Pick any ideal triangle, and let's show that the arc-orderings on its sides are sane in the sense of Def.\ref{def:compatible}; that is, we must show that there is no insane triple of loop segments in this triangle with respect to these arc-orderings. Let's assume that each of the three corners of the triangle has a loop segment, because otherwise there's nothing to check. In particular, as pointed out in Lem.\ref{lem:basic_lemma_on_endpoints_of_loop_segments}.(3) and Lem.\ref{lem:self-folded_is_always_sane}, we may assume that this triangle is not self-folded; for self-folded, there is nothing to check. Now pick any triple of loop segments, say $j,k,l$, living in three different corners. Let $e,f,g$ be the ideal arcs, so that $j$ connects $e,f$, while $k$ connects $f,g$, and $l$ connects $g,e$. The two endpoint junctures $j$ are denoted by $j_e,j_f$, each living in $e,f$ respectively. Likewise, denote the endpoint junctures of $k$ by $k_f,k_g$, and those of $l$ by $l_g,l_e$. See Fig.\ref{fig:biggest_in_corridor} for an example.

\vs

\ul{(the insanity criterion)} 

Before proceeding, we discuss how to check whether these loop segments $j,k,l$ are insane or not. On the arc $e$, find the inner arc segment $i_e$ with the largest difference number appearing in between the junctures $j_e$ and $l_e$; the orientation on this inner arc segment $i_e$ determines the ordering on the two junctures $j_e$ and $l_e$, by Lem.\ref{lem:how_to_read_dyadic_arc-ordering}. For convenience, let's call the inner arc segment \ul{\em largest} if the corresponding difference number is the largest among the ones in consideration. Likewise, let $i_f$ be the largest inner arc segment on the arc $f$ in between $j_f$ and $k_f$, and $i_g$ be the largest inner arc segment on the arc $g$ in between $k_g$ and $l_g$. So, the loop segments $j,k,l$ are insane iff the orientations on these inner arc segments $i_e,i_f,i_g$ are `cyclic'. In terms of the orientations on the corresponding edges of the regional graph $\mathcal{R}$ drawn on the surface $S$, one observes from the `transferring' relation (Def.\ref{def:transfer}) that this cyclicity condition for $i_e,i_f,i_g$ is equivalent to the orientations on the three edges of $\mathcal{R}$ corresponding to $i_e,i_f,i_g$ being either all pointing `\ul{\em inward}' toward the interior of the triangle or all pointing `\ul{\em outward}'. To summarize, on each arc find the largest inner arc segment located inside this region bounded by $j,k,l$. Look at the orientations of the edges of $\mathcal{R}$ corresponding to these three inner arc segments. If they are all pointing inward or all pointing outward, then $j,k,l$ are insane. Otherwise, $j,k,l$ are not insane. 

(end of insanity criterion)

\vs

So, to show that $j,k,l$ are not insane, we may restrict our attention to the region in this triangle `\ul{\em inside}' these three loop segments $j,k,l$, or `bounded by' $j,k,l$; what happens outside this region is not relevant. This region consists of narrow regions, exactly one of which is a $3$-end narrow region, and the remaining, if any, are $2$-end narrow regions.

\vs

For convenience, in this proof with a fixed choice of $j,k,l$, the \ul{\em largest number on the arc $e$} refers to the the largest number among the difference numbers assigned to inner arc segments in $e$ between $j_e$ and $l_e$; that is, we omit the phrase `between $j_e$ and $l_e$'. The inner arc segment to which the largest number on $e$ is assigned is called the \ul{\em largest inner arc segment on the arc $e$}, and the corresponding edge of $\mathcal{R}$ the \ul{\em largest edge of $\mathcal{R}$ for the arc $e$}. Likewise for the arcs $f$ and $g$. 

\vs

Now, consider all the inner arc segments appearing in the region inside $j,k,l$, i.e. inner arc segments on arcs $e,f,g$ living in between the endpoint junctures of $j,k,l$. Consider the difference numbers assigned to them; we refer to these numbers as \ul{\em difference numbers inside $j,k,l$}. By condition 1), these numbers are mutually distinct and are of the form $2^m$ for positive integers $m$.

\vs

[Case 1: when the largest among the difference numbers inside $j,k,l$ occurs at an inner arc segment that is one end of a $2$-end narrow region $N$ in this triangle]

\vs

Let this number be $2^m$, and without loss of generality, suppose that this inner arc segment is in the arc $e$, which the loop segments $j,l$ intersect with, as in Fig.\ref{fig:biggest_in_corridor}. Also, without loss of generality, suppose that this $2$-end narrow region $N$ connects the arcs $e$ and $g$; almost same proof shall work for the case when it connects $e$ and $f$. Then, by the condition 2) of the present Lemma that we are trying to prove, the difference number on the inner arc segment that is the other end of this $2$-end narrow region $N$ is either $2^{m-1}$ or $2^{m+1}$. Since this other inner arc segment is also inside $j,k,l$, and since $2^m$ must be the largest difference number inside $j,k,l$, it can't be $2^{m+1}$, so it must be $2^{m-1}$. Hence, in turn, by condition 2), we also know the orientations on the two edges of $\mathcal{R}$ attached to the vertex corresponding to this $2$-end narrow region $N$; these orientations go `from' $2^{m-1}$ `to' $2^m$, as depicted in Fig.\ref{fig:biggest_in_corridor}, or, equivalently, the $2^m$-edge of $\mathcal{R}$ is pointing outward and the $2^{m-1}$-edge of $\mathcal{R}$ is pointing inward with respect to this triangle. Note that $2^{m-1}$ is the second largest number inside $j,k,l$ in this triangle. Thus $2^{m-1}$ is the largest number on arc $g$, while $2^m$ is the largest number on arc $e$; and we just saw that the corresponding edges of $\mathcal{R}$ are inward and outward, respectively. Hence, by the `insanity criterion' above, $j,k,l$ are not insane. [end of Case 1]

\begin{figure}
\includegraphics[width=80mm]{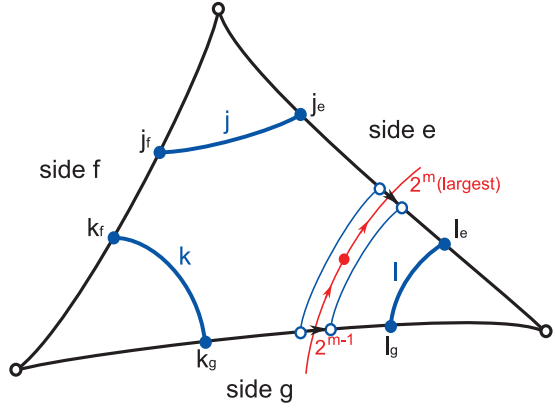} 
\caption{Case 1, when biggest $2^m$ inside $j,k,l$ occurs at a $2$-end narrow region}
\label{fig:biggest_in_corridor}
\end{figure}

\vs

[Case 2 : when the largest among the differences numbers inside $j,k,l$ occurs at an inner arc segment that is one end of the unique $3$-end narrow region $\til{N}$ inside $j,k,l$.]

\vs

[Case 2-I : the vertex of $\mathcal{R}$ corresponding to this $3$-end narrow region $\til{N}$ is of type I (Def.\ref{def:types_of_vertices})]

\vs

Let the largest number inside $j,k,l$ be $2^m$. Without loss of generality, suppose that this largest end of the $3$-end narrow region $\til{N}$, to which $2^m$ is assigned, is on the arc $e$. Consider the three edges of $\mathcal{R}$ corresponding to the three ends of this $3$-end narrow region $\til{N}$. In particular, $2^m$ is the largest among the three weights on these three edges of $\mathcal{R}$. 

\vs

By condition 3) of the present Lemma, it must be that this largest $2^m$-edge is the flow-out edge, and in particular, it is pointing outward with respect to this triangle. Then, again by condition 3), one of the remaining two edges of $\mathcal{R}$ for this $3$-end narrow region $\til{N}$ is the flow-in edge, and hence is given the weight $2^{m-1}$ with the inward orientation. See Fig.\ref{fig:biggest_in_3-way-juncture} for an example when the inner arc segment for this $2^{m-1}$-edge lies in the arc $g$. Since $2^m$ is the largest number inside $j,k,l$, it must be that it is the largest number on the arc $e$, and that $2^{m-1}$ is the second largest number inside $j,k,l$, and hence is the largest number on the arc $g$. Since the $2^m$-edge of $\mathcal{R}$ is outward and $2^{m-1}$-edge of $\mathcal{R}$ is inward, by the `insanity criterion', we see that $j,k,l$ are not insane. [end of Case 2-I]

\vs

[Case 2-II : the vertex of $\mathcal{R}$ corresponding to this $3$-end narrow region $\til{N}$ is of type II (Def.\ref{def:types_of_vertices})]

\vs

Let's show that in this case, the region of this triangle inside $j,k,l$ consists of just one narrow region, namely the $3$-end narrow region $\til{N}$. Suppose not, so that there is a $2$-end narrow region lying inside $j,k,l$. Then there must be at least one $2$-end narrow region $L$ inside $j,k,l$ that is `adjacent to', i.e. sharing a common loop segment with, the $3$-end narrow region $\til{N}$. 

\vs

The present Case 2-II is assuming the existence of a type II connected component of the regional graph $\mathcal{R}$, hence by Lem.\ref{lem:type_II_topological_implication}, $\gamma$ must enclose a subsurface $S'$ of $S$ containing no puncture or a boundary component of $S$. Since the whole surface $S$ has a puncture or a boundary component of $S$, the subsurface $S'$ cannot equal $S$. Hence the loop $\gamma$ divides the surface $S$ into two distinct regions, one being the subsurface $S'$. The region other than $S'$, which we may call $S''$, contains all punctures and boundary components of $S$. Note that $S'$ is located on one `side' with respect to $\gamma$, and $S''$ on the other `side' with respect to $\gamma$. For example, if we give an orientation to the loop $\gamma$, then we may say that one of $S'$ and $S''$ is at the `right' of the loop $\gamma$, and the other is at the `left' of $\gamma$.

\vs

Note now that the narrow regions $\til{N}$ and $L$ are at different `sides' with respect to the common loop segment, hence with respect to the loop $\gamma$. Since $\til{N}$ corresponds to a type II $3$-valent vertex of $\mathcal{R}$, Lem.\ref{lem:type_II_topological_implication} says that $\til{N}$ belongs to the subsurface $S'$. Hence $L$ does not belong to $S'$, and belongs to the other region $S''$. We claim that the vertex of $\mathcal{R}$ corresponding to the narrow region $L$ is contained in a type I connected component of $\mathcal{R}$. If not, then it is contained in a type II connected component. Since a type II connected component is unique ($\because$ Cor.\ref{cor:structure_of_R}), Lem.\ref{lem:type_II_topological_implication} says that the narrow region $L$ is contained in the subsurface $S'$, which is a contradiction; so the claim is proved. In particular, $\mathcal{R}$ has a connected component of type I and a connected component of type II. Therefore, by condition 4), the weight for any of the two edges of $\mathcal{R}$ attached to the vertex of $\mathcal{R}$ corresponding to $L$ is larger than the weight on any edge of $\mathcal{R}$ attached to the vertex of $\mathcal{R}$ corresponding to $\til{N}$. So the difference numbers on the two ends of the $2$-end narrow region $L$ are larger than any of the difference numbers on the three ends of $\til{N}$. Since $L$ is also located inside $j,k,l$, this contradicts to the assumption of Case 2 that the largest difference number inside $j,k,l$ occur at an end of $\til{N}$.

\vs

So, indeed, in this Case 2-II, there cannot exist a $2$-end narrow region inside $j,k,l$. Hence the region of the triangle inside $j,k,l$ coincides with the $3$-end narrow region $\til{N}$. In particular, on each arc $e,f,g$, there is only one inner arc segment lying in this region inside $j,k,l$. Notice that the orientations on the edges of $\mathcal{R}$ corresponding to the three ends of $\til{N}$ are neither all inward nor all outward, by condition 3).  So, by the `insanity criterion', $j,k,l$ are not insane. [end of Case 2-II].
\end{proof}

\begin{figure}
\includegraphics[width=83mm]{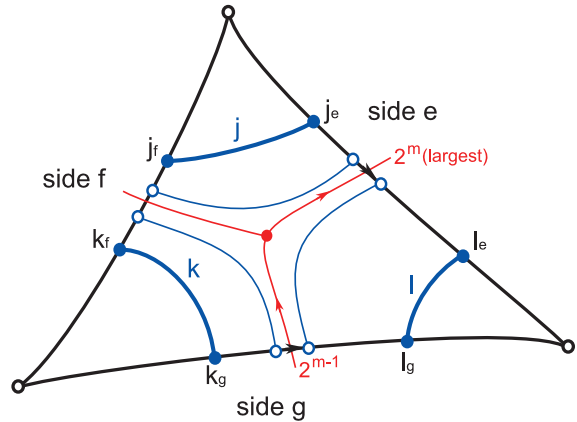} 
\caption{Case 2-I: biggest $2^m$ inside $j,k,l$ occurs at $3$-end narrow region of type I}
\label{fig:biggest_in_3-way-juncture}
\end{figure}

\subsection{Existence of good orientations and weights on $\mathcal{R}$}

Now it only remains to find a choice of orientations and weights on the edges of $\mathcal{R}$ that meets the condition of the above Prop.\ref{prop:sufficient_condition_for_data_on_R}. This is the most technical part of the present paper. 
\begin{proposition}[existence of good orientations and weights on edges of $\mathcal{R}$]
\label{prop:existence_of_data_on_R}
Let $S,T,\gamma$ be as in Def.\ref{def:basic}, and $\mathcal{R}$ be the corresponding regional graph. Then there exists a choice of orientations and weights on edges of $\mathcal{R}$ satisfying all conditions of Prop.\ref{prop:sufficient_condition_for_data_on_R}.
\end{proposition}

We devote the present subsection for a proof of Prop.\ref{prop:existence_of_data_on_R}. We shall describe an algorithm to construct orientations and weights on the edges of $\mathcal{R}$ satisfying the desired conditions. We deal with each connected component of $\mathcal{R}$ separately. From now on, denote by $\mathcal{G}$ a connected component of $\mathcal{R}$. Lem.\ref{lem:component_criterion} tells us that $\mathcal{G}$ is a `full subgraph' of $\mathcal{R}$; that is, for each pair of vertices of $\mathcal{G}$, if these vertices are connected by an edge of $\mathcal{R}$, then this edge is contained in $\mathcal{G}$. Notice also that the valence of each vertex of $\mathcal{G}$ is $1,2$, or $3$, and that $\mathcal{G}$ has no cycle of length 1 (i.e. $\mathcal{G}$ has no self-loop).

\begin{definition}

Let $\mathcal{G}$ be a connected component of $\mathcal{R}$.

$\bullet$ For an edge of $\mathcal{G}$, the two vertices of $\mathcal{G}$ that are connected by this edge are called the \ul{\em endpoint vertices} of this edge. We say that this edge is \ul{\em attached to} each of its endpoint vertices.

$\bullet$ For a subset $S$ of the set of all edges of $\mathcal{G}$, a \ul{\em chain $\mathcal{C}$ in $S$} is a sequence $e_1,e_2,\ldots,e_m$ of mutually distinct edges in $\mathcal{G}$ (with $m\ge 1$) such that for each $i=1,2,\ldots,m-1$, the two edges $e_i,e_{i+1}$ have a common endpoint vertex. We say that $e_1,\ldots,e_m$ are the \ul{\em constituent edges of $\mathcal{C}$}.

$\bullet$ This $m$ is called the \ul{\em length} of the chain $\mathcal{C}$, and we write $m = |\mathcal{C}|$.

$\bullet$ For a chain $\mathcal{C} = e_1,e_2,\ldots,e_m$, we say that $\mathcal{C}$ \ul{\em starts from} the edge $e_1$ and \ul{\em ends at} the edge $e_m$, and that $e_1$ is the \ul{\em starting edge of $\mathcal{C}$} and $e_m$ is the \ul{\em ending edge of $\mathcal{C}$}.

$\bullet$ For a chain $\mathcal{C} = e_1,e_2,\ldots,e_m$ with $m\ge2$, let $v$ be the endpoint vertex of $e_1$ that is not an endpoint vertex of $e_2$, and let $w$ be the endpoint vertex of $e_m$ that is not an endpoint vertex of $e_{m-1}$. We say that $v$ is the \ul{\em departing vertex of $\mathcal{C}$}, and $w$ is the \ul{\em terminating vertex of $\mathcal{C}$}. For each $i=1,\ldots,m-1$, the common endpoint vertex of $e_i$ and $e_{i+1}$ is called a \ul{\em middle vertex of $\mathcal{C}$}.

$\bullet$ We declare that a chain $\mathcal{C}$ of length $1$ consists of the choice of an edge $e$, together with the choice of one endpoint vertex of $e$ to be the departing vertex of $\mathcal{C}$; then the other endpoint vertex of $e$ is the terminating vertex of $\mathcal{C}$. 

$\bullet$ For two chains $\mathcal{C} = e_1,\ldots,e_m$ and $\mathcal{C}' = e'_1,\ldots,e'_M$ in $S$, we say that $\mathcal{C}$ \ul{\it extends} $\mathcal{C}'$ if $m\ge M$ and there exists $i\in \{0,1,\ldots,m-M'\}$ such that $e_{i+j} = e'_j$ for all $j=1,2,\ldots,M$.

$\bullet$ A chain in a subset $S$ is called \ul{\em maximal} if it cannot be extended to a chain in $S$ different from itself, i.e. there is no member $e$ in $S$ such that either $e,e_1,e_2,\ldots,e_m$ or $e_1,e_2,\ldots,e_m,e$ is a chain in $S$.

$\bullet$ For chain $\mathcal{C} =e_1,e_2,\ldots,e_m$, define the \ul{\em reverse chain $\ol{\mathcal{C}}$} of $\mathcal{C}$ as the chain $\ol{\mathcal{C}} = e_m,e_{m-1},\ldots,e_2,e_1$, whose departing vertex is the terminating vertex of $\mathcal{C}$ and whose terminating vertex is the departing vertex of $\mathcal{C}$.
\end{definition}
It is easy to see that the reverse chain of a chain is indeed also a chain, and that the reverse chain of a maximal chain is also maximal. The departing vertex and the terminating vertex of a chain may coincide. The departing vertex of a chain may also be a middle vertex of the same chain. The terminating vertex of a chain may also be a middle vertex of the same chain.

\vs

Let $S_0$ be the set of all edges of the connected component $\mathcal{G}$ of $\mathcal{R}$. We shall partition $S_0$, i.e. the graph $\mathcal{G}$, into chains as follows, in a recursive manner. At each $N$-th stage with $N\ge 0$, we will construct a chain $\mathcal{C}_N$ inside $S_N$ satisfying a certain condition, then let $S_{N+1} := S_N \setminus \mathcal{C}_N$. In particular, we will have a sequence of nested sets $S_0 \supset S_1 \supset S_2 \supset \cdots$. The initial stage is $N=0$, when we have $S_0$ at our hand and have to construct $\mathcal{C}_0$.

\begin{definition}[partial valence]
For a subset $S$ of the set of all edges of $\mathcal{G}$, and a vertex $v$ of $\mathcal{G}$, denote by $S(v)$ the set of all edges in $S$ attached to $v$, and by $|S(v)|$ the size of $S(v)$.
\end{definition}

\begin{lemma}
\label{lem:partial_valence}
If $S\subset S'$, then $S(v) \subset S'(v)$, hence $|S(v)| \le |S'(v)|$. \qed
\end{lemma}

We shall require that each $S_N$ should satisfy the following condition.

\begin{definition}[notion of sustainability of an edge set]
\label{def:sustainable}
A subset $S$ of the set of all edges of $\mathcal{G}$ is said to be \ul{\em sustainable} if all the following three conditions hold:

\vs

(a) For each $2$-valent vertex $v$ of $\mathcal{G}$, we have $|S(v)| \in \{0,2\}$.

\vs

(b) For each $3$-valent vertex $v$ of $\mathcal{G}$, we have $|S(v)| \in \{0,1,3\}$.

\vs

(c) If $S$ is non-empty, then there exists a vertex $v$ of $\mathcal{G}$ such that $|S(v)|=1$.
\end{definition}

Condition (c) is used in the following situation.
\begin{lemma}[existence of a maximal chain with a prescribed departing vertex]
\label{lem:existence_of_maximal_chain}
Let $S$ be any nonempty subset of the set of all edges of $\mathcal{G}$. Let $v$ be a vertex of $\mathcal{G}$ such that $|S(v)|=1$. Let $S(v) = \{e\}$. Then there exists a maximal chain $\mathcal{C}$ in $S$, such that $\mathcal{C}$ starts from $e$ and the departing vertex of $\mathcal{C}$ is $v$.
\end{lemma}

\begin{proof}
For any chain $\mathcal{C}' = e_1,\ldots,e_m$ in a subset $S$ of edges of $\mathcal{G}$, the existence of a maximal chain $\mathcal{C}$ extending $\mathcal{C}'$, is obvious. Namely, one can construct $\mathcal{C}$ recursively from $\mathcal{C}'$ step by step, as follows. Let $\mathcal{C}'_0 = \mathcal{C}'$. If $\mathcal{C}'_0$ is maximal in $S$, we are done. If not, there is an edge $e$ in $S$ such that either $e,\mathcal{C}'_0$, i.e. the sequence $e,e_1,\ldots,e_m$, or $\mathcal{C}'_0,e$, i.e. the sequence $ e_1,\ldots,e_m,e$, is a chain in $S$; let this new chain be $\mathcal{C}'_1$. If $\mathcal{C}'_1$ is maximal, we are done; if not, find an edge $f$ in $S$ s.t. either $f,\mathcal{C}'_1$ or $\mathcal{C}'_1,f$ is a chain in $S$. Let this chain be $\mathcal{C}'_2$. And so on, until one obtains a maximal chain in $S$. This process stops in a finitely many steps, because in our case, the graph $\mathcal{R}$, hence also $\mathcal{G}$, has only finitely many edges in total. 

\vs

For this lemma, let $\mathcal{C}' = e$ be the chain of length $1$ in $S$. Hence there exists a maximal chain $\mathcal{C}$ in $S$ extending $\mathcal{C}'$. Suppose that $v$ is neither the departing vertex nor the terminating vertex of $\mathcal{C}$. Since $v$ is an endpoint vertex of $e$ which is a constituent of $\mathcal{C}$, it follows that $v$ is a middle vertex of $\mathcal{C}$. This means there is an edge $f$ in $S$ such that $e,f$ are both attached to $v$, and the chain $\mathcal{C}$ extends either the chain $e,f$ or the chain $f,e$. But then $S(v)$ contains $e$ and $f$, so $|S(v)|\ge 2$, contradicting to the assumption $|S(v)|=1$. So $v$ is either the departing vertex or the terminating vertex of $\mathcal{C}$. In the former case, it follows that $\mathcal{C}$ starts from $e$ with the departing vertex $v$, so we are done. In the latter case, the reverse chain $\ol{\mathcal{C}}$ works.
\end{proof}

\begin{lemma}[initial sustainability for type I]
If $\mathcal{G}$ is a type I connected component of $\mathcal{R}$, then $S_0$ is sustainable.
\end{lemma}

\begin{proof}
Since $S_0$ contains all edges of $\mathcal{G}$, obviously (a) and (b) hold for $S_0$; if $v$ is a $k$-valent vertex of $\mathcal{G}$ with $k\in\{2,3\}$, then $|S_0(v)| = k$. By definition of a type I connected component (Def.\ref{def:types}), $S_0$ has an edge attached to a $1$-valent vertex of $\mathcal{G}$, so (c) holds.
\end{proof}

However, when $\mathcal{G}$ is of type II, $S_0$ is not sustainable, because of condition (c). For this case, we shall construct a chain $\mathcal{C}_0$ in a special way as follows. 
\begin{lemma}[the initial chain for a type II component]
\label{lem:initial_chain_for_type_II}
If $\mathcal{G}$ is a type II connected component of $\mathcal{R}$, then there exists a chain $\mathcal{C}_0$ in $S_0$ such that 
\begin{enumerate}
\item[\rm (1)] the departing vertex of $\mathcal{C}_0$ is a $3$-valent vertex of $\mathcal{G}$,

\item[\rm (2)] the departing vertex of $\mathcal{C}_0$ coincides with the terminating vertex of $\mathcal{C}_0$,

\item[\rm (3)] no middle vertex of $\mathcal{C}_0$ coincides with the departing vertex of $\mathcal{C}_0$.
\end{enumerate}
From now on, choose one such $\mathcal{C}_0$ and fix it. The departing vertex of $\mathcal{C}_0$ is called the \ul{\em special type II vertex}.
\end{lemma}

\begin{proof}
By assumption, $\mathcal{G}$ is nonempty. As in the proof of Lem.\ref{lem:existence_of_maximal_chain}, there exists a maximal chain $\mathcal{C}_0'$ of $\mathcal{G}$. If $\mathcal{C}'_0$ already satisfies the desired condition, we are done. Write $\mathcal{C}'_0 = e_1,e_2,\ldots,e_m$, with the corresponding vertices $v_0,v_1,\ldots,v_m$. That is, $v_0$ is the departing vertex, $v_m$ is the terminating vertex, and $v_i$ is the middle vertex shared by $e_i,e_{i+1}$, for each $i=1,\ldots,m-1$. First claim is that $v_0,v_1,\ldots,v_m$ are not mutually distinct. Suppose they are. The valence of $v_m$ is either $2$ or $3$ because $\mathcal{G}$ is of type II, and $e_m$ is one of the edges attached to $v_m$. Let $e'$ any edge of $\mathcal{G}$ that is attached to $v_m$ and is different from $e_m$. Then $e'$ is not one of $e_1,\ldots,e_m$. It cannot be $e_m$. If it is $e_i$ for some $i=1,\ldots,m-1$, then $v_m$ must have appeared among $v_0,v_1,\ldots,v_{m-1}$, which is a contradiction. Hence $e_1,\ldots,e_m,e'$ is a chain in $\mathcal{G}$, contradicting to the maximality of $\mathcal{C}'_0$. So $v_i= v_j$ for some $i,j\in \{0,1,\ldots,m\}$.

\vs

Find such pair of $i,j$ with a minimal value of $|i-j|$; we may assume $i<j$. Let $\mathcal{C}_0''$ be the chain $e_{i+1},e_{i+2},\ldots,e_j$, whose corresponding vertices are $v_i,v_{i+1},\ldots,v_j$. In particular, the departing vertex $v_i$ equals its terminating vertex $v_j$, and $v_i,v_{i+1},\ldots,v_{j-1}$ are mutually distinct. So $\mathcal{C}''_0$ is a cycle of length $j-i$ inside the graph $\mathcal{G}$. Suppose that all the vertices $v_i,v_{i+1}, \ldots,v_{j-1}$ of $\mathcal{C}''_0$ are $2$-valent. If they are, then for the subgraph $\mathcal{G}'$ formed by this cycle $\mathcal{C}''_0$, for each vertex of $\mathcal{G}'$, all edges of $\mathcal{G}$ attached to this vertex belongs to $\mathcal{G}'$, hence by Lem.\ref{lem:component_criterion} the subgraph $\mathcal{G}'$ is a connected component of $\mathcal{G}$, for $\mathcal{G}'$ is connected. But $\mathcal{G}$ itself is a connected component of $\mathcal{R}$, hence $\mathcal{G}$ is a connected graph, therefore $\mathcal{G}'$ must be the entire $\mathcal{G}$. Then it follows that all vertices of $\mathcal{G}$ are $2$-valent, contradicting to Lem.\ref{lem:type_II_component_contains_a_3-valent_vertex} which says that any type II connected component of $\mathcal{R}$ has at least one $3$-valent vertex. Therefore, the assumption that all vertices of $\mathcal{C}''_0$ are $2$-valent is false, so some $v_k$ among them is $3$-valent (with $k\in\{i,i+1,\ldots,j-1\}$). Let $\mathcal{C}_0$ be the cycle obtained by cyclically shifting $\mathcal{C}_0''$ so that its departing vertex is $v_k$. That is, let $\mathcal{C}_0 = e_{k+1},e_{k+2},\ldots,e_j,e_{i+1},e_{i+2},\ldots,e_k$, with corresponding vertices being $v_k,v_{k+1},\ldots,v_j=v_i,v_{i+1},v_{i+2},\ldots,v_k$. Because $\mathcal{C}''_0$ is a cycle, i.e. $v_i = v_j$, one notices that $\mathcal{C}_0$ defined as such is also a well-defined chain in $\mathcal{G}$, and that the vertices $v_k,v_{k+1},\ldots,v_j=v_i,v_{i+1},v_{i+2},\ldots,v_{k-1}$ are mutually distinct. The departing vertex $v_k$ of $\mathcal{C}_0$ is $3$-valent and coincides with the terminating vertex of $\mathcal{C}_0$. A middle vertex of $\mathcal{C}_0$, i.e. a vertex among $v_{k+1},\ldots,v_j=v_i,v_{i+1},\ldots,v_{k-1}$, does not coincide with the departing and terminating vertex $v_k$ of $\mathcal{C}_0$.
\end{proof}

\vs

\begin{lemma}[initial sustainability for type II]
If $\mathcal{G}$ is a type II connected component of $\mathcal{R}$, let $S_1 := S_0 \setminus \mathcal{C}_0$, where $\mathcal{C}_0$ is as constructed in Lem.\ref{lem:initial_chain_for_type_II}. Then $S_1$ is sustainable.
\end{lemma}

\begin{proof}
Let $v$ be a $2$-valent vertex of $\mathcal{G}$; let $e,f$ be the edges of $\mathcal{G}$ attached to $v$. Suppose one of these two edges appears in $\mathcal{C}_0$, say $e$. Then $v$ must be a middle vertex of $\mathcal{C}_0$, because the departing and terminating vertex of $\mathcal{C}_0$ is $3$-valent, by Lem.\ref{lem:initial_chain_for_type_II}. Hence it follows that $f$ also appears in $\mathcal{C}_0$, next to $e$. Likewise, $f$ appearing in $\mathcal{C}_0$ implies that $e$ appears in $\mathcal{C}_0$. So, $e,f$ either both appear in $\mathcal{C}_0$, or both are absent from $\mathcal{C}_0$. Hence $|S_1(v)|=0$ or $2$.

\vs

Let $v$ be a $3$-valent vertex of $\mathcal{G}$; let $e,f,g$ be the edges of $\mathcal{G}$ attached to $v$. Suppose one of these three edges appears in $\mathcal{C}_0$, say $e$. So $S_1(v) \subseteq \{f,g\}$. Suppose $S_1(v) = \{f,g\}$. This means $f,g$ do not appear in $\mathcal{C}_0$, so $v$ cannot be a middle vertex of $\mathcal{C}_0$. Hence $v$ is the departing and terminating vertex of $\mathcal{C}_0$. Hence $e$ is either the starting edge or the ending edge of $\mathcal{C}_0$. In case $e$ is the starting edge, then the ending edge must be $f$ or $g$, for the ending edge is an edge attached to the terminating vertex $v$ and must be different from $e$. Likewise, in case $e$ is the ending edge, then the starting edge must be $f$ or $g$. So, in either case, $f$ or $g$ also appears in $\mathcal{C}_0$, contradicting to the assumption. Hence $S_1(v)$ cannot be $\{f,g\}$, so $|S_1(v)|<2$. So we proved that $|S_1(v)|$ is one of $0$, $1$, or $3$, where $3$ means none of $e,f,g$ appears in $\mathcal{C}_0$. 

\vs

Suppose $S_1$ is non-empty. Let $v$ be the departing and terminating vertex of $\mathcal{C}_0$, which is $3$-valent; see Lem.\ref{lem:initial_chain_for_type_II}. Let $e,f,g$ be the edges of $\mathcal{G}$ attached to $v$. The starting edge of $\mathcal{C}_0$ is one of these three, say $e$, and the ending edge of $\mathcal{C}_0$ is another one, say $f$. So $S_1(v) \subseteq \{g\}$. Let's show $S_1(v) = \{g\}$. If not, it means $S_1(v) = {\O}$, i.e. $g$ appears in $\mathcal{C}_0$. Then $v$ must also be a middle vertex of $\mathcal{C}_0$, contradicting to condition 3) of Lem.\ref{lem:initial_chain_for_type_II}. Hence indeed $S_1(v)=\{g\}$, meaning $|S_1(v)|=1$. So all conditions (a), (b), (c) for $S_1$ being sustainable (Def.\ref{def:sustainable}) are satisfied.
\end{proof}

\begin{definition}[description of the $N$-th inductive stage]
\label{def:N-th_stage}
Suppose that a subset $S_N$ of the set of all edges of $\mathcal{G}$ is sustainable. Construct a maximal chain $\mathcal{C}_N$ in $S_N$, and the next set $S_{N+1}$ as follows.

\vs

\ul{Step 1 of $N$-th stage:} \quad perform exactly one of the following 
two cases.

\vs

Case 1) When $S_N$ is empty, we are done; i.e. the entire recursive process is completed at this moment. Note then that  $S_0$ is the disjoint union of chains $\mathcal{C}_0,\mathcal{C}_1,\ldots,\mathcal{C}_{N-1}$.

\vs

Case 2) When $S_N$ is non-empty, find a vertex $v$ of $\mathcal{G}$ s.t. $|S_N(v)|=1$, whose existence is guaranteed by condition (c) of Def.\ref{def:sustainable}. Then choose a maximal chain $\mathcal{C}_N'$ in $S_N$, such that $\mathcal{C}'_N$ starts from $e$ and the departing vertex of $\mathcal{C}'_N$ is $v$ (such chain exists by Lem.\ref{lem:existence_of_maximal_chain}).

\vs

\vs

\ul{Step 2 of $N$-th stage:} \quad perform exactly one of the following two cases.

\vs

Case 1) If the terminating vertex of the chain $\mathcal{C}_N'$ is not a middle vertex of the chain $\mathcal{C}_N'$ itself, then let $\mathcal{C}_N := \mathcal{C}_N'$. 

\vs

Case 2) If the terminating vertex of the chain $\mathcal{C}_N'$ is a middle vertex of the chain $\mathcal{C}_N'$ itself, then let $\mathcal{C}_N := \ol{\mathcal{C}'_N}$, the reverse chain of $\mathcal{C}_N'$. In particular, $\mathcal{C}_N$ is also maximal.

\vs

Write the chain $\mathcal{C}_N$ resulting from Step 2 as
$$
\mathcal{C}_N = e_{N,1},e_{N,2},\ldots,e_{N,r_N},
$$
where $r_N = |\mathcal{C}_N|\ge 1$ is the length of $\mathcal{C}_N$.

\vs

\ul{Step 3 of $N$-th stage:} \quad Let $S_{N+1} := S_N\setminus \mathcal{C}_N$. 

\end{definition}

As the notation suggests, we then would feed in $S_{N+1}$ into the above algorithm to perform the $(N+1)$-th stage, to construct a chain $\mathcal{C}_{N+1}$, etc. For this algorithm to go on, a key thing to check is whether $S_{N+1}$ is also sustainable. We shall verify this in the following three lemmas. Meanwhile, notice that $S_{N+1}$ is a proper subset of $S_N$, so the inductive process finishes after finitely many stages.

\begin{lemma}[a behavior on $3$-valent vertex]
\label{lem:behavior_on_3-valent_vertex}
Suppose $S_0,S_1,\ldots,S_N$ are constructed so far, with $N\ge 1$ and $S_1,\ldots,S_N$ being sustainable. Let $v$ be a $3$-valent vertex of $\mathcal{G}$ such that $|S_N(v)|=1$. Then $v$ is either a middle vertex of $\mathcal{C}_{N'}$ for some $N'<N$, or the special type II vertex in Lem.\ref{lem:initial_chain_for_type_II}.
\end{lemma}

\begin{proof}
By construction, we have $S_0 \supset S_1 \supset \cdots \supset S_N$, hence $|S_0(v)| \ge |S_1(v)| \ge \cdots \ge |S_N(v)|=1$ by Lem.\ref{lem:partial_valence}. Since $S_0$ has all edges of $\mathcal{G}$, it follows $|S_0(v)|=3$. Condition (b) of the sustainability (Def.\ref{def:sustainable}) of $S_1,\ldots,S_N$ gives $|S_i(v)| \in \{0,1,3\}$ for each $i=1,\ldots,N$. So there exists $N'\ge 0$ such that $|S_i(v)|=3$ for all $i=0,1,\ldots,N'$ and $|S_i(v)|=1$ for all $i=N'+1,\ldots,N$; in particular $N'<N$. So the three edges of $\mathcal{G}$ attached to $v$, say $e,f,g$, all belong to $S_{N'}$, and only one of them, say $g$, belongs to $S_{N'+1} = S_{N'} \setminus \mathcal{C}_{N'}$. This means that the other two edges $e,f$ belong to the chain $\mathcal{C}_{N'}$, and that $g$ does not. In particular, $v$ is at least one of: the departing vertex, the terminating vertex, or a middle vertex of $\mathcal{C}_{N'}$. 

\vs

Suppose that $v$ is not a middle vertex of $\mathcal{C}_{N'}$. Then $v$ is either the departing vertex or the terminating vertex of $\mathcal{C}_{N'}$. Suppose that $v$ is the departing vertex; then the starting edge of $\mathcal{C}_{N'}$ must be $e$ or $f$, say $e$. Since $f$ also appears in the chain $\mathcal{C}_{N'}$, and since its endpoint vertex $v$ is not a middle vertex of $\mathcal{C}_{N'}$, it must be that $v$ is the terminating vertex of $\mathcal{C}_{N'}$ and that $f$ is the ending edge of $\mathcal{C}_{N'}$. Likewise, $v$ being the terminating vertex of $\mathcal{C}_{N'}$ implies $v$ being the departing vertex of $\mathcal{C}_{N'}$. In any case, $v$ is both the departing and the terminating vertex of $\mathcal{C}_{N'}$. 

\vs

Now assume that $\mathcal{C}_{N'}$ is not the chain $\mathcal{C}_0$ of the type II connected component $\mathcal{G}$. Then $\mathcal{C}_{N'}$ must have been constructed by the $N'$-th inductive stage in Def.\ref{def:N-th_stage}. In view of Step 2 of $N'$-th inductive stage (Def.\ref{def:N-th_stage}), the chain $\mathcal{C}'_{N'}$ is either $\mathcal{C}_{N'}$ or $\ol{\mathcal{C}_{N'}}$; hence $v$ is also the departing and the terminating vertex of $\mathcal{C}'_{N'}$. In view of Step 1 of $N'$-th stage, it must be that $|S_{N'}(v)|=1$, contradicting to our case $|S_{N'}(v)|=3$.

\vs

So we showed that $v$ is either a middle vertex of $\mathcal{C}_{N'}$, or the departing and terminating vertex of $\mathcal{C}_0$ of the type II connected component of $\mathcal{G}$, i.e. the special type II vertex.
\end{proof}

\begin{lemma}[property of the chain $\mathcal{C}_N$]
\label{lem:property_of_C_N}
Let $N\ge 0$.
\begin{enumerate}
\item[\rm (1)] Neither the departing vertex nor the terminating vertex of the constructed chain $\mathcal{C}_N$ is a $2$-valent vertex of $\mathcal{G}$.

\item[\rm (2)] If a vertex $v$ of $\mathcal{G}$ is the departing vertex or the terminating vertex of $\mathcal{C}_N$, then one of the following holds:

(2-1) $v$ is a $1$-valent vertex,

(2-2) $v$ is a $3$-valent vertex that is also a middle vertex of the chain $\mathcal{C}_{N'}$ for some $N'\le N$,

(2-3) $v$ is the special type II vertex.

\item[\rm (3)] The terminating vertex of $\mathcal{C}_N$ is not a middle vertex of $\mathcal{C}_N$.
\end{enumerate}
\end{lemma}

\begin{proof}
First, suppose that $N=0$ and $\mathcal{G}$ is a type II connected component of $\mathcal{R}$. Then all the statements of the present Lemma follows immediately from Lem.\ref{lem:initial_chain_for_type_II}. So, now let's exclude this case. Then $\mathcal{C}_N$ must have been constructed by the $N$-th inductive stage of Def.\ref{def:N-th_stage}. In particular, $\mathcal{C}_N$ is a chain in the sustainable set $S_N$. 

\vs

(1) This is a consequence of (2).

\vs

(2) Let $v$ be the departing vertex of the chain $\mathcal{C}'_N$ in $S_N$ constructed in Step 1 of Def.\ref{def:N-th_stage}. So $|S_N(v)|=1$. In view of the sustainability condition of $S_N$ (Def.\ref{def:sustainable}), $v$ is either $1$-valent or $3$-valent. In case $v$ is $3$-valent, by Lem.\ref{lem:behavior_on_3-valent_vertex} it follows that $v$ is either a middle vertex of a chain $\mathcal{C}_{N'}$ for some $N'<N$, or the special type II vertex. So, in any case, this departing vertex $v$ of $\mathcal{C}'_N$ satisfies the desired condition.

\vs

Let's now show that the terminating vertex $w$ of $\mathcal{C}_N'$ also satisfies the condition. Suppose it were the Case 2 at Step 2 of Def.\ref{def:N-th_stage}. Then $w$ is also a middle vertex of $\mathcal{C}_N'$, hence is a common endpoint vertex of some consecutive constituent edges $e,f$ of $\mathcal{C}_N'$. One observes that the ending edge $g$ of $\mathcal{C}_N'$ cannot be $e$ or $f$; if it were, then both endpoints of this edge are $w$, so that this edge is a self-loop, which is absurd. So there are three distinct edges $e,f,g$ attached to $w$, meaning that $w$ is $3$-valent. Meanwhile, $\mathcal{C}_N = \ol{\mathcal{C}'_N}$, so $w$ is a middle vertex of $\mathcal{C}_N$, as desired. Suppose now it were the Case 1 at Step 2. If $w$ is $1$-valent, we are done. Suppose not. Let $g$ be the ending edge of $\mathcal{C}_N'$, so $g$ is attached to $w$. This means $g\in S_N(w)$, so $|S_N(w)|\ge 1$. In case $w$ is $2$-valent, by condition (a) of the sustainability condition for $S_N$ (Def.\ref{def:sustainable}), it follows $|S_N(w)|=2$. This means that the edge $h$ of $\mathcal{G}$ attached to $w$ that is not $g$ is also in $S_N$. Moreover, note that $h$ cannot appear in $\mathcal{C}'_N$, because the only way for $h$ to appear in $\mathcal{C}'_N$ is as the starting edge of $\mathcal{C}'_N$ with $w$ being the departing vertex of $\mathcal{C}'_N$, implying $w=v$, which is impossible because $|S_N(v)|=1$. Hence the chain $\mathcal{C}_N'$ can be extended by adding $h$ at the (right) end, contradicting to its maximality. In case $w$ is $3$-valent, by condition (b) of the sustainability condition for $S_N$, it follows $|S_N(w)|=1$ or $3$. Suppose $|S_N(w)|=3$. Note that $w$ is neither the departing vertex of $\mathcal{C}'_N$ (because $w\neq v$, for $|S_N(w)|=3$ and $|S_N(v)|=1$) nor a middle vertex of $\mathcal{C}'_N$ (by Case 1 of Step 2). Hence,  neither of the two edges of $\mathcal{G}$ attached to $w$ different from the ending edge of $\mathcal{C}'_N$ appears in the chain $\mathcal{C}_N'$, so $\mathcal{C}'_N$ can be extended by adding one of these edges at the (right) end, contradicting to its maximality. So it must be $|S_N(w)|=1$. Then by Lem.\ref{lem:behavior_on_3-valent_vertex}, $w$ is either a middle vertex of a chain $\mathcal{C}_{N'}$ for some $N'<N$, or the special type II vertex. Done.

\vs

So we showed that both the departing vertex and the terminating vertex of $\mathcal{C}'_N$ satisfy the desired condition. Since $\mathcal{C}_N$ is either $\mathcal{C}'_N$ or the reversed chain $\ol{\mathcal{C}'_N}$, we are done.

\vs

(3) At Step 2 of Def.\ref{def:N-th_stage}, if it were the Case 1, then we are good. If it were the Case 2, then the terminating vertex $v$ of $\mathcal{C}_N$ is the departing vertex $v$ of $\mathcal{C}_N'$. Looking at Step 1, we have $|S_N(v)|=1$, so $v$ is either $1$-valent or $3$-valent, in view of condition (a) of the sustainability of $S_N$. If $v$ is a $1$-valent vertex, then it cannot be a middle vertex of $\mathcal{C}_N$, for it cannot be the common endpoint vertex of two distinct edges of $\mathcal{G}$. Suppose now that $v$ is a $3$-valent vertex with $|S_N(v)|=1$. Since $\mathcal{C}_N \subset S_N$, there can be at most $1$ edge in $\mathcal{C}_N$ attached to $v$. Hence, $v$ cannot be a middle vertex of $\mathcal{C}_N$; if it is, then there should be at least two distinct edges in $\mathcal{C}_N$ attached to $v$. So (3) is proved. 
\end{proof}

\begin{lemma}
[sustainability is preserved at each inductive step]
\label{lem:induction_process}
For $N\ge 0$, the set $S_{N+1}$, constructed by the $N$-th induction stage of Def.\ref{def:N-th_stage}, is sustainable.

\end{lemma}

\begin{proof}
Note $S_{N+1} = S_N\setminus \mathcal{C}_N$, hence $S_{N+1} \subset S_N$. Recall that the $N$-th induction stage assumes the sustainability of $S_N$.

\vs

(a) Let $v$ be a $2$-valent vertex of $\mathcal{G}$. If $|S_N(v)|=0$, then $|S_{N+1}(v)|=0$, because $S_{N+1} \subset S_N$ and therefore $|S_{N+1}(v)|\le |S_N(v)|$ by Lem.\ref{lem:partial_valence}. So condition (a) for $S_{N+1}$ is satisfied. Suppose $|S_N(v)|=2$. If none of the two edges of $\mathcal{G}$ attached to $v$ belongs to $\mathcal{C}_N$, then $|S_{N+1}(v)|=2$, so (a) is satisfied. Suppose $\mathcal{C}_N$ has an edge $e$ attached to $v$. By Lem.\ref{lem:property_of_C_N}.(1), $v$ cannot be the departing vertex or the terminating vertex of $\mathcal{C}_N$, so it must be a middle vertex of $\mathcal{C}_N$, hence is the common endpoint vertex of some consecutive constituent edges $e_{N,i},e_{N,i+1}$ of the chain $\mathcal{C}_N$. These edges $e_{N,i},e_{N,i+1}$ are the two edges of $\mathcal{G}$ attached to $v$. Since they are in $\mathcal{C}_N$, they do not belong to $S_{N+1} = S_N \setminus\mathcal{C}_N$, so $|S_{N+1}(v)|=0$. Hence (a) is satisfied.

\vs

(b) Let $v$ be a $3$-valent vertex of $\mathcal{G}$. Suppose $|S_N(v)|\le 1$. Then since $0\le |S_{N+1}(v)|\le |S_N(v)|$, it must be that $|S_{N+1}(v)|$ is either $0$ or $1$, so (b) is satisfied. Suppose now $|S_N(v)|=3$. Since $|S_{N+1}(v)|\le |S_N(v)|$, it suffices to check that $|S_{N+1}(v)|$ cannot be $2$. Suppose it is $2$. This means that all three edges of $\mathcal{G}$ attached to $v$ belong to $S_N$, and that only one of them, say $e$, is in the chain $\mathcal{C}_N$. So $e=e_{N,i}$ for some $i\in \{1,2,\ldots,r_N\}$. In case $v$ is a middle vertex of $\mathcal{C}_N$, then it is the common endpoint vertex of some consecutive constituent edges $e_{N,j},e_{N,j+1}$ of $\mathcal{C}_N$, which in particular are attached to $v$; this contradicts to the assumption that only one edge of $\mathcal{G}$ attached to $v$ belongs to $\mathcal{C}_N$. So $v$ is not a middle vertex of $\mathcal{C}_N$. However, since it is an endpoint of a constituent edge $e$ of $\mathcal{C}_N$, and since it is not a middle vertex of $\mathcal{C}_N$, it must be at least one of: the departing vertex or the terminating vertex $\mathcal{C}_N$. By Lem.\ref{lem:property_of_C_N}.(2), $v$ must be either a middle vertex of $\mathcal{C}_{N'}$ for some $N'\le N$, or the special type II vertex. In the former case, by what we just saw, it must be that $N'<N$, which in particular implies $N\ge 1$. So, $v$ is the common endpoint vertex of some consecutive constituent edges $e_{N',k},e_{N',k+1}$ of $\mathcal{C}_{N'}$.  As $S_N = S_0 \setminus (\mathcal{C}_0 \cup \mathcal{C}_1 \cup \cdots \cup \mathcal{C}_{N-1})$, it follows that the edges $e_{N',k},e_{N',k+1}$, which are attached to $v$, do not belong to $S_N$; this is a contradiction to $|S_N(v)|=3$. Now, in the latter case of Lem.\ref{lem:property_of_C_N}.(2), i.e. when $v$ is the special type II vertex, then by Lem.\ref{lem:initial_chain_for_type_II}, the starting edge and the ending edge of the chain $\mathcal{C}_0$ are two distinct edges attached to $v$. Since $S_{N+1} = S_0 \setminus (\mathcal{C}_0 \cup \cdots \cup \mathcal{C}_N)$, we see that these two edges do not belong to $S_{N+1}$, hence $|S_{N+1}(v)|\le 1$, contradicting to the current assumption $|S_{N+1}(v)|=2$. So, in any case, the assumption $|S_{N+1}(v)|=2$ leads to a contradiction, hence is false. Therefore indeed $|S_{N+1}(v)|$ can only be one of $0,1,3$. So (b) is satisfied.
 
\vs

(c) If $S_{N+1}$ is empty, there is nothing to check for (c). Suppose $S_{N+1}$ is not empty, and suppose that $S_{N+1}$ does not have an edge attached to a $1$-valent vertex. In order to show that (c) is satisfied, in view of condition (a), it suffices to show that $S_{N+1}$ has an edge $e$ attached to a $3$-valent vertex $v$ such that $|S_{N+1}(v)|=1$. Suppose not. For each edge $e$ of $S_{N+1}$, for each endpoint vertex $v$ of $e$, we have $|S_{N+1}(v)|\ge 1$ because $e\in S_{N+1}(v)$. So the valence of $v$ must be $2$ or $3$, for if the valence is $1$ then $|S_{N+1}(v)|$ would be $1$. If $v$ is $2$-valent, then by condition (a) it follows $|S_{N+1}(v)|=2$. If $v$ is $3$-valent, then by condition (b) and our assumption $|S_{N+1}(v)|\neq 1$ it follows that $|S_{N+1}(v)|=3$. This means that, for any endpoint vertex $v$ of any edge $e$ of $S_{N+1}$, all edges of $\mathcal{G}$ attached to $v$ belong to $S_{N+1}$. So, by Lem.\ref{lem:component_criterion} we see that $S_{N+1}$ together with all endpoint vertices of its members form a union of connected components of $\mathcal{G}$. Since $\mathcal{G}$ is connected, it has only one connected component, namely itself. Hence $S_{N+1}$ must be the entire graph $\mathcal{G}$. This is absurd, because $S_{N+1}$ is a proper subset of $S_0$, which is the set of all edges of $\mathcal{G}$. Hence we showed what we wanted to show, and (c) is satisfied.
\end{proof}

\vspace{5mm}

In the end, by running the inductive process (Def.\ref{def:N-th_stage}) until it finishes, we obtain a partition of $S_0$, the set of all edges of any chosen connected component $\mathcal{G}$ of the regional graph $\mathcal{R}$, into chains $\mathcal{C}_0, \mathcal{C}_1,\ldots, \mathcal{C}_M$, for some $M\ge 0$.

\vs

We assign orientations to members $e_{i,1},\ldots,e_{i,r_i}$ of $\mathcal{C}_i$ so that the orientations are directing toward the `forward traveling direction', i.e. `from the departing vertex to the terminating vertex'; in case $r_i\ge 2$, it looks like 
$$
\xymatrix@C-2mm{
\ar[r]^{e_{i,1}} & \ar[r]^{e_{i,2}} & \ar[r]^{\cdots} & \ar[r]^{e_{i,r_i}} &
}
$$

\vs

Consider the sequence of edges in $\mathcal{G}$ constructed as the concatenation of chains $\mathcal{C}_M,\mathcal{C}_{M-1},\ldots,\mathcal{C}_0$ arranged in this order, i.e.
\begin{align}
\label{eq:sequence_of_edges_of_G}
e_{M,1},e_{M,2},\ldots,e_{M,r_M}, \, e_{M-1,1},e_{M-1,2},\ldots,e_{M-1,r_{M-1}}, \ldots, e_{0,1},e_{0,2},\ldots,e_{0,r_0}.
\end{align}
Notice that each edge of $\mathcal{G}$ appears exactly once in this sequence. 

\vs

Now, let $\mathcal{G}_1,\ldots,\mathcal{G}_d$ be all the connected components of the regional graph $\mathcal{R}$. If $\mathcal{R}$ has a type II connected component, then by Cor.\ref{cor:structure_of_R} there is only one type II connected component; in this case, re-label the connected components if necessary, so that $\mathcal{G}_1$ is the type II component. For each $\mathcal{G}_i$, we have a sequence of its edges as constructed in eq.\eqref{eq:sequence_of_edges_of_G}. Concatenate these sequences for $\mathcal{G}_1,\mathcal{G}_2,\ldots,\mathcal{G}_d$, arranged in this order; so the sequence for $\mathcal{G}_1$ comes first. Then we obtain the sequence of edges of $\mathcal{R}$, such that each edge of $\mathcal{R}$ appears exactly once. For this sequence, assign weights $2^1,2^2,\ldots,2^{|\mathcal{R}|}$ to the members of this sequence in this order, where $|\mathcal{R}|$ is the number of edges of $\mathcal{R}$. This way we assigned orientations and weights on all edges of $\mathcal{R}$.

\vs

Finally, let us prove Prop.\ref{prop:existence_of_data_on_R}; that is, let's show that these orientations and weights assigned on the edges of $\mathcal{R}$ satisfy all the conditions 1), 2), 3), and 4) of Prop.\ref{prop:sufficient_condition_for_data_on_R}.

\vs

\ul{\it Proof of Prop.\ref{prop:existence_of_data_on_R}}.

\vs

\ul{[condition 1) of Prop.\ref{prop:sufficient_condition_for_data_on_R}]} : \quad trivially satisfied.

\vs

\ul{[condition 4) of Prop.\ref{prop:sufficient_condition_for_data_on_R}]} : \quad Suppose $\mathcal{R}$ has a type I connected component and a type II connected component. Then the unique type II connected component is $\mathcal{G}_1$, which comes `before' other connected components, which are all of type I. In the above construction of weights, notice that the weight given to any edge of $\mathcal{G}_1$ is smaller than the weight given to any edge of $\mathcal{G}_i$, for any $i\ge 2$. So condition 4) is satisfied.

\vs

\ul{[condition 2) of Prop.\ref{prop:sufficient_condition_for_data_on_R}]} : \quad Let $v$ be any $2$-valent vertex of $\mathcal{R}$. By Lem.\ref{lem:property_of_C_N}.(1) and Lem.\ref{lem:initial_chain_for_type_II}, we see that $v$ is not the departing vertex nor the terminating vertex of any of the constructed chains $\mathcal{C}_N$. Meanwhile, note that each vertex of $\mathcal{R}$, in particular $v$, is an endpoint vertex of an edge of $\mathcal{R}$, and that each edge of $\mathcal{R}$ belongs to exactly one of these chains $\mathcal{C}_N$. Note also that each vertex of a constituent edge of a chain is at least one of: the departing vertex, the terminating vertex, or a middle vertex of the chain. It follows that $v$ is a middle vertex of some chain $\mathcal{C}_N$. By construction of the orientations and the weights given on the members of $\mathcal{C}_N$, we see that condition 2) is satisfied on this $2$-valent vertex $v$.

\vs

\ul{[condition 3) of Prop.\ref{prop:sufficient_condition_for_data_on_R}]} : \quad Let $v$ be any $3$-valent vertex of $\mathcal{R}$. Note that, in the above inductive process, right after we construct the `last' (or, `final') chain $\mathcal{C}_M$ for the connected component $\mathcal{G}$ where $v$ belongs, we defined $S_{M+1} := S_M \setminus \mathcal{C}_M$ at Step 3 of the $M$-th stage (Def.\ref{def:N-th_stage}). Note that $\mathcal{C}_M$ being the last chain means that $S_{M+1}$ is the empty set. Now, for each $N=1,2,\ldots,M+1$, we know $|S_N(v)|$ belongs to $\{0,1,3\}$, by the sustainability condition (b) (Def.\ref{def:sustainable}) and the sustainability result (Lem.\ref{lem:induction_process}). By construction of $S_N$'s, we have $S_0 \supset S_1 \supset \cdots \supset S_M \supset S_{M+1}$, and therefore $|S_0(v)| \ge |S_1(v)| \ge \cdots \ge |S_M(v)| \ge |S_{M+1}(v)|$ by Lem.\ref{lem:partial_valence}. Note $|S_0(v)| = 3$, for $S_0$ is the set of all edges of $\mathcal{G}$, and note $|S_{M+1}(v)|=0$, for $S_{M+1}$ is the empty set. Find the largest $N \in \{0,1,\ldots,M+1\}$ such that $|S_N(v)| > 0$; then $0\le N\le M$, and $|S_N(v)|$ is either $3$ or $1$, while $|S_{N+1}(v)|=0$. 

\vs

Suppose $|S_N(v)|=3$. This means all three edges of $\mathcal{G}$ attached to $v$ belong to $S_N$, and none of them belong to $S_{N+1} = S_N \setminus \mathcal{C}_N$. Thus all these three edges belong to the chain $\mathcal{C}_N$. We claim that $v$ is a middle vertex of $\mathcal{C}_N$; otherwise, each of the three edges attached to $v$ is the starting edge or the ending edge of $\mathcal{C}_N$, which is impossible because these three edges are distinct. Hence, $v$ is the common endpoint vertex of some consecutive constituent edges $e_{N,i}, e_{N,i+1}$ of the chain $\mathcal{C}_N = e_{N,1},\ldots,e_{N,r_N}$; these two edges are attached to $v$. We now claim that this $i$ is the unique number in $0,1,\ldots,r_N-1$ such that the common vertex of $e_{N,i},e_{N,i+1}$ is $v$. If there is another $i'$ such that $v$ is the common vertex of $e_{N,i'}, e_{N,i'+1}$, then all of the edges $e_{N,i},e_{N,i+1},e_{N,i'}, e_{N,i'+1}$ are attached to $v$, hence cannot all be distinct, and the only possibility is either $i+1=i'$ or $i'+1=1$. In the former case, both endpoints of $e_{N,i+1}$ are $v$, and in the latter case, both endpoints of $e_{N,i'+1}$ are $v$; this is absurd, because our graph $\mathcal{G}$ has no self-loop. So uniqueness of $i$ is proved.

Let $e = e_{N,j} \in \mathcal{C}_N$ be the remaining edge of $\mathcal{G}$ attached to $v$; in particular, $j \notin \{i,i+1\}$. Suppose that $j\neq 1$, i.e. $e=e_{N,j}$ is not the starting edge of $\mathcal{C}_N$. In case $e=e_{N,j}$ is not the ending edge of $\mathcal{C}_N$ either, then at the two endpoint vertices of $e=e_{N,j}$ are attached the edges $e_{N,j-1}$ and $e_{N,j+1}$ of $\mathcal{C}_N$, respectively. Since $v$ is an endpoint vertex of $e=e_{N,j}$, it follows that either $e_{N,j-1}$ or $e_{N,j+1}$ is attached to $v$. So $v$ is a common vertex of $e_{N,j-1},e_{N,j}$, or that of $e_{N,j},e_{N,j+1}$; this contradicts to the uniqueness of the above $i$. So it must be that $e=e_{N,j}$ is the ending edge of $\mathcal{C}_N$, i.e. $e=e_{N,j}=e_{N,r_N}$. In case $v$ is not the terminating vertex of $\mathcal{C}_N$, then it must be a middle vertex that is a common endpoint vertex of $e_{N,r_N-1},e_{N,r_N}$, again contradicting to the uniqueness of the above $i$. So it follows that $v$ is the terminating vertex. However, since $v$ is a middle vertex of $\mathcal{C}_N$, Lem.\ref{lem:property_of_C_N}.(3) tells us that $v$ cannot be the terminating vertex of $\mathcal{C}_N$. So the assumption $j\neq 1$ cannot be true, and therefore we have $j=1$, i.e. $e=e_{N,1}$ is the starting edge of $\mathcal{C}_N$ (we can also show that $v$ is the departing vertex of $\mathcal{C}_N$). By construction of orientations and weights on the constituent edges of $\mathcal{C}_N$, it follows that the orientation on $e_{N,i}$ is incoming (toward $v$), that on $e_{N,i+1}$ is outgoing, and the weights on $e_{N,1},e_{N,i},e_{N,i+1}$ are $2^r,2^m,2^{m+1}$ for some positive integers $r,m$ with $r<m$; see the left of Fig.\ref{fig:condition3}. So, whether $v$ is of type I or type II, the condition 3) of Prop.\ref{prop:sufficient_condition_for_data_on_R} holds.

\begin{figure}
\hspace{10mm} \includegraphics[width=100mm]{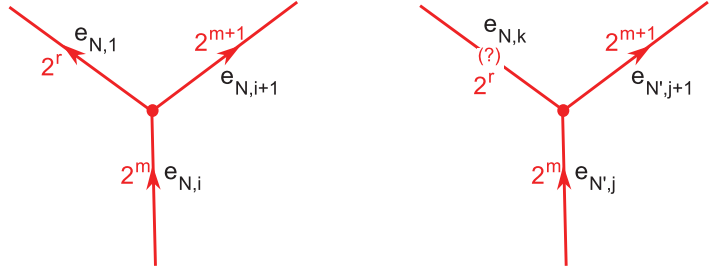}
\caption{examples for checking condition 3}
\label{fig:condition3}
\end{figure}

\vs

Suppose now $|S_N(v)|=1$, where $N$ was the largest number such that $|S_N(v)|>0$; so $|S_{N+1}(v)|=0$. By Lem.\ref{lem:behavior_on_3-valent_vertex}, $v$ is either a middle vertex of a chain $\mathcal{C}_{N'}$ for some $N'<N$, or the special type II vertex (Lem.\ref{lem:initial_chain_for_type_II}). In the former case, $v$ is the common endpoint vertex of some consecutive constituent edges $e_{N',j},e_{N',j+1}$ of $\mathcal{C}_{N'}$. Since $|S_{N+1}(v)| = |(S_N\setminus \mathcal{C}_N)(v)|=0$ and $|S_N(v)|=1$, it means that one edge of $\mathcal{G}$ attached to $v$ belongs to $\mathcal{C}_N$, so this edge is $e_{N,k}$ for some $k$ (one can prove that $k$ is either $1$ or $r_N$). Since $N'\neq N$, this edge must be distinct from $e_{N',j},e_{N',j+1}$, for $\mathcal{C}_{N'}$ and $\mathcal{C}_N$ are disjoint. By construction of orientations and weights on the constituent edges of $\mathcal{C}_N$ and $\mathcal{C}_{N'}$, it follows that the orientation on $e_{N',j}$ is incoming (toward $v$), that on $e_{N',j+1}$ is outgoing, and the weights on $e=e_{N,k},e_{N',j},e_{N',j+1}$ are $2^r,2^m,2^{m+1}$ for some positive integers $r,m$ with $r<m$; see the right of Fig.\ref{fig:condition3}. So, whether $v$ is of type I or type II, the condition 3) of Prop.\ref{prop:sufficient_condition_for_data_on_R} is satisfied. Now, for the latter case of Lem.\ref{lem:behavior_on_3-valent_vertex}, i.e. when $v$ is the special type II vertex, from Lem.\ref{lem:initial_chain_for_type_II} one observes that in the chain $\mathcal{C}_0$, the starting edge is outgoing from $v$ and the ending edge is incoming toward $v$. Hence the type II version of condition 3) of Prop.\ref{prop:sufficient_condition_for_data_on_R} is satisfied.

\vs

This finishes the proof of Prop.\ref{prop:existence_of_data_on_R}. \qed

\vs

Let us summarize what we have proved so far. In the present subsection we proved Prop.\ref{prop:existence_of_data_on_R}, which says that there exists a choice of orientations and weights on the edges of the regional graph $\mathcal{R}$ satisfying all conditions of Prop.\ref{prop:sufficient_condition_for_data_on_R}. The statement of Prop.\ref{prop:sufficient_condition_for_data_on_R} is that the orientations and difference numbers on inner arc segments transferred from such orientations and weights on $\mathcal{R}$ yield (via Lem.\ref{lem:arc-binary-ordering_to_arc-ordering}) dyadic arc-orderings on the ideal arcs of $T$ that are compatible and sane at every ideal triangle. Hence the arc-ordering problem is solved, i.e. Thm.\ref{thm:arc-ordering} is proved. Since Thm.\ref{thm:arc-ordering} implies Thm.\ref{thm:ordering_problem} as mentioned already in \S\ref{subsec:the_ordering_problems}, we proved the original ordering problem of loop segments, i.e. Thm.\ref{thm:ordering_problem}. In the next section, we shall finally prove the main Thm.\ref{thm:main}, using this Thm.\ref{thm:ordering_problem}.

\section{Applying the ordering problem to quantum Teichm\"uller theory}
\label{sec:applying}

\subsection{Chekhov-Fock algebra and its square-root version}

Here we briefly recall the constructions in Allegretti-Kim \cite{AK} and Bonahon-Wong \cite{BW}, and show how our topological result Thm.\ref{thm:ordering_problem} implies the algebraic positivity result Thm.\ref{thm:main}. For more details, please refer to those two original papers.

\vs

We shall also refine the notations and statements employed in the introduction section. We mostly try to follow notations in \cite{AK}, but modify when necessary. First, choose a decorated surface $S$ and an ideal triangulation $T$ of $S$ (see \S\ref{subsec:basic_definitions}). Let $t_1,t_2,\ldots,t_m$ be all the ideal triangles of $T$. For each $i=1,\ldots,m$, let $e_{i1}, e_{i2}, e_{i3}$ denote the sides of the triangle $t_i$, so that these sides occur in the clockwise order in $t_i$. The \ul{\em triangle square-root algebra $\mathcal{Z}^\omega_{t_i}$} is the algebra over $\mathbb{Z}[\omega,\omega^{-1}]$ generated by $\wh{Z}_{i1}, \wh{Z}_{i2}, \wh{Z}_{i3}$, and their inverses, with relations
\begin{align}
\label{eq:commutation_in_Z_in_triangle}
\wh{Z}_{i1} \wh{Z}_{i2} = \omega^2 \wh{Z}_{i2} \wh{Z}_{i1}, \quad \wh{Z}_{i2} \wh{Z}_{i3} = \omega^2 \wh{Z}_{i3} \wh{Z}_{i2}, \quad \wh{Z}_{i3} \wh{Z}_{i1} = \omega^2 \wh{Z}_{i1} \wh{Z}_{i3}.
\end{align}
Here $\omega$ can be thought of as a formal parameter symbol, or as a fixed nonzero complex number. Now consider the tensor product algebra $\bigotimes_{i=1}^m \mathcal{Z}^\omega_{t_i}$. Identify each element of $\mathcal{Z}^\omega_{t_i}$ as an element of this tensor product algebra, via the natural embedding map. For each non-self-folded ideal arc $e$ of $T$, let $t_i,t_j$ be the two triangles having $e$ as one of their sides, and let $e_{ia}$ and $e_{jb}$ be the sides of $t_i$ and $t_j$ corresponding to $e$; define $\wh{Z}_e := \wh{Z}_{ia} \wh{Z}_{jb}$ as an element of the tensor product. When $e$ is a self-folded side $e_{ia}=e_{i \, a+1}$ of the triangle $t_i$, define $\wh{Z}_e := \omega^{-1} \wh{Z}_{ia} \wh{Z}_{i \,a+1}$. When $e$ is a boundary arc of $T$, say the edge $e_{ia}$ of the triangle $t_i$, define $\wh{Z}_e := \wh{Z}_{ia}$ (this boundary arc case is missing in \cite{BW}).

\vs

Define the \ul{\em Chekhov-Fock square-root algebra $\mathcal{Z}^\omega_T$} as the subalgebra of $\bigotimes_{i=1}^m \mathcal{Z}^\omega_{t_i}$ generated by these elements $\wh{Z}_e$, for $e\in T$. The defining set of relations of $\mathcal{Z}^\omega_T$ for these generators is
\begin{align}
\label{eq:commutation_in_Z}
\wh{Z}_e \wh{Z}_f = \omega^{2\varepsilon_{ef}} \wh{Z}_f \wh{Z}_e, \qquad \forall e,f\in T,
\end{align}
where the constant $\varepsilon_{ef} \in \{-2,-1,0,1,2\}$ is defined as
\begin{align*}
\varepsilon_{ef} & = a_{ef} - a_{fe}, \\
\quad a_{ef} & = \mbox{the number of corners of triangles in $T$ delimited by $e$ in the left and $f$ in the right,}
\end{align*}
where the notion of `left' and `right' can be chosen consistently, using the orientation on the surface $S$. The actual \ul{\em Chekhov-Fock algebra} $\mathcal{X}^q_T$ is the $\mathbb{Z}[q,q^{-1}]$-subalgebra of $\mathcal{Z}^\omega_T$ generated by the squares of generators
\begin{align}
\label{eq:X_and_Z}
\mbox{$\wh{X}_e := \wh{Z}_e^2$, \,\, $e\in T$, \quad with \quad $q := \omega^4$.}
\end{align}
The defining relations of the generators of $\mathcal{X}^q_T$ are
$$
\wh{X}_e \wh{X}_f = q^{2\varepsilon_{ef}} \wh{X}_f \wh{X}_e, \qquad \forall e,f\in T.
$$
Soon we will see why needed to consider the square-roots of the variables too. 

\vs

Per each change of ideal triangulation $T \leadsto T'$, there is a quantum coordinate change map between the corresponding algebras $\mathcal{X}^q_T$ and $\mathcal{X}^q_{T'}$, which is a quantization of the classical coordinate change formula for the exponential shear coordinates associated to the change $T\leadsto T'$. To be more precise, for any $T,T'$, there is an isomorphism
\begin{align}
\label{eq:Phi_q}
\Phi^q_{TT'} : {\rm Frac}(\mathcal{X}^q_{T'}) \to {\rm Frac}(\mathcal{X}^q_T)
\end{align}
between the skew-fields of fractions of the Chekhov-Fock algebras \footnote{See \cite{BW} and references therein for skew-fields of fractions}, which satisfy the consistency relations
$$
\Phi^q_{TT'} \circ \Phi^q_{T'T''} = \Phi^q_{TT''},
$$
(see \cite{F} \cite{CF} \cite{Kash} \cite{FG09}) and therefore let us identify all ${\rm Frac}(\mathcal{X}^q_T)$ for different $T$'s in a consistent manner.  A `(rational) function' on the quantum Teichm\"uller space can be viewed as an element of ${\rm Frac}(\mathcal{X}^q_T)$, for any chosen $T$, and such element can be viewed as an element of ${\rm Frac}(\mathcal{X}^q_{T'})$ for any different $T'$, via the map $\Phi^q_{TT'}$. A function is considered to be {\em regular with respect to $T$} if it belongs to the subalgebra $\mathcal{X}^q_T$ of ${\rm Frac}(\mathcal{X}^q_T)$. A  function that is regular with respect to $T$ may not be regular with respect to a different $T'$,  because the subalgebra $\Phi^q_{T'T}(\mathcal{X}^q_T)$ of ${\rm Frac}(\mathcal{X}^q_{T'})$ may not lie inside $\mathcal{X}^q_{T'}$. A  function is called \ul{\em regular} if it is regular with respect to all triangulations $T$, i.e. can be expressed, for each ideal triangulation $T$, as a Laurent polynomial in the variables $\wh{X}_e$'s, $e\in T$, with coefficients in $\mathbb{Z}[q,q^{-1}]$. The ring of all regular functions, which we denote by $\mathcal{O}^q(\mathcal{X})$ and was written as $\mathcal{X}^q_{{\rm PGL}_2,S}$ in the introduction, is thus the intersection of the subalgebras $\mathcal{X}^q_T$ of ${\rm Frac}(\mathcal{X}^q_T)$ for all $T$, put inside one ambient skew-field via the maps $\Phi^q_{TT'}$. That is, for each $T$, we can realize $\mathcal{O}^q(\mathcal{X})$ as a subalgebra of ${\rm Frac}(\mathcal{X}^q_T)$, in a consistent manner. The main task undertaken in \cite{AK} is to construct a nice `$\mathbb{Z}$-basis' of this ring of quantum regular functions $\mathcal{O}^q(\mathcal{X})$. We put the quotation marks to `basis' because it is not yet proved to be actually a basis, even for punctured surfaces $S$.

\vs

The Allegretti-Kim `basis' of $\mathcal{O}^q(\mathcal{X})$ is enumerated by even integral laminations, which we describe now. 
\begin{definition}[\cite{F} \cite{FG06} \cite{AK}]
An \ul{\em integral lamination} \footnote{This word does not seem to have a universal definition but has many versions, hence must be carefully defined each time.} $\ell$ in $S$ is a homotopy class of a (possibly empty) collection of finitely many mutually-non-intersecting good loops (Def.\ref{def:good_loop}) in $S$ with the choice of an integer weight for each constituent loop, under the following condition and the equivalence relation:
\begin{enumerate}
\item[\rm 1)] The weight on a constituent loop can be negative only when the loop is peripheral (Def.\ref{def:good_loop});

\item[\rm 2)] A lamination having a constituent loop with zero weight is equivalent to the lamination with that loop removed;

\item[\rm 3)] A lamination having homotopic loops of weights $a$ and $b$ is equivalent to the lamination with one of these loops removed and the weight $a+b$ on the other.
\end{enumerate}
The collection of all integral laminations is denoted by $\mathring{\mathcal{A}}_{\rm L}(S,\mathbb{Z})$.
\end{definition}
In fact, more general laminations (constituting the usual $\mathcal{A}_{\rm L}(S,\mathbb{Z})$) should allow components that are not loops but curves ending at points on boundary edges, but we will now use them in our paper.
\begin{definition}[Fock coordinates of laminations; \cite{F} \cite{FG06}]
\label{def:Fock_coordinates}
Let $\ell \in \mathring{\mathcal{A}}_{\rm L}(S,\mathbb{Z})$ and $T$ be an ideal trianguation of $S$. Suppose each of the constituent loops of $\ell$ is in a minimal position (Def.\ref{def:minimal_position}) with respect to $T$. For each ideal arc $e$ of $T$, define the number $a_e (\ell) = a_{T,e}(\ell)$ to be $\frac{1}{2}$ times the total weight of $\ell$ on the arc $e$.
\end{definition}
So each number $a_e(\ell)$ is an half-integer, i.e. is in $\frac{1}{2}\mathbb{Z}$.
\begin{definition}[laminations with integer coordinates; \cite{FG06} \cite{AK}]
\label{def:integral_laminations}
An \ul{\em even integral lamination on $S$} \footnote{This is not a widely used term, and only used in the present paper.} is an element $\ell \in \mathring{\mathcal{A}}_{\rm L}(S,\mathbb{Z})$ whose Fock coordinate $a_{T,e}(\ell)$ is an integer for each ideal triangulation $T$ and each ideal arc $e$ of $T$. Denote by $\mathring{\mathcal{A}}_{{\rm SL}_2,S}(\mathbb{Z}^t)$ the set of all even integral laminations on $S$.
\end{definition}
\begin{lemma}[\cite{FG06}]
For $\ell \in \mathring{\mathcal{A}}_{\rm L}(S,\mathbb{Z})$, if $a_{T,e}(\ell) \in \mathbb{Z}$ holds for all ideal arcs $e$ of a triangulation $T$, then it is true also for any other ideal triangulation. 
\end{lemma}

The Allegretti-Kim construction \cite{AK} can be thought of as an injective map
\begin{align}
\label{eq:AK_map}
\wh{\mathbb{I}}^q : \mathring{\mathcal{A}}_{{\rm SL}_2,S}(\mathbb{Z}^t) \hookrightarrow \mathcal{O}^q(\mathcal{X}),
\end{align}
satisfying certain desired properties, so that the image of $\wh{\mathbb{I}}^q$ is a nice `basis' of $\mathcal{O}^q(\mathcal{X})$; in \cite{AK} only punctured surfaces are considered, but it is straightforward to extend to decorated surfaces. In practice, this map $\wh{\mathbb{I}}^q$ is constructed as a map
\begin{align}
\label{eq:AK_map_for_T}
\wh{\mathbb{I}}^q_T : \mathring{\mathcal{A}}_{{\rm SL}_2,S}(\mathbb{Z}^t) \hookrightarrow \mathcal{X}^q_T \subset {\rm Frac}(\mathcal{X}^q_T)
\end{align}
for each triangulation $T$, satisfying $\Phi^q_{TT'} \circ \wh{\mathbb{I}}^q_{T'} = \wh{\mathbb{I}}^q_T$ for each pair of triangulations $T,T'$. 

\vs

Before going into the construction of this map, we need one more refinement; namely, we must consider square-roots of generators, in order to apply the results of \cite{BW} to construct such a map \eqref{eq:AK_map_for_T}. First, there is a square-root version of \eqref{eq:Phi_q}: there exists an algebra isomorphism \cite{Hiatt}
$$
\Theta^\omega_{TT'} : \wh{\mathcal{Z}}^\omega_{T'} \to \wh{\mathcal{Z}}^\omega_{T}
$$
where $\wh{\mathcal{Z}}^\omega_T$ is a subalgebra of ${\rm Frac}(\mathcal{Z}^\omega_T)$ satisfying a certain parity condition; in fact, what is constructed in \cite{Hiatt} is such a map for algebras only involving ideal arcs but not boundary arcs, but it is not hard to extend it to boundary arcs too, and in fact we only need to deal with ideal arcs in our case. Namely, $\wh{\mathcal{Z}}^\omega_T$ is the span of all elements of the form $PQ^{-1}$ with $Q\in \mathcal{T}^q_T$ and $P$ a {\em balanced} element of $\mathcal{Z}^\omega_T$, which is a linear combination of monomials such that for each triangle $t$ of $T$ the powers of generators for the sides of $t$ in each monomial add up to even integers. These isomorphisms restrict to the previous $\Phi^q_{TT'}$, and also enjoy the consistency relation $\Theta^\omega_{TT'} \circ \Theta^\omega_{T'T''} = \Theta^\omega_{TT''}$, letting us identify all $\wh{\mathcal{Z}}^\omega_T$ for different $T$'s. Now, let $\mathcal{O}^\omega(\mathcal{Z})$ be the subalgebra of $\wh{\mathcal{Z}}^\omega_T$ given by the intersection of all $\Theta^\omega_{TT'}(\mathcal{Z}^\omega_{T'})$, where $T'$ ranges over all possible ideal triangulations. So an element of $\mathcal{O}^\omega(\mathcal{Z})$ is a function on the quantum Teichm\"uller space that can be expressed, for each ideal triangulation $T$, as a Laurent polynomial in the variables $\wh{Z}_e$'s, $e\in T$, with coefficients in $\mathbb{Z}[\omega,\omega^{-1}]$. In particular, $\mathcal{O}^q(\mathcal{X}) \subset \mathcal{O}^\omega(\mathcal{Z})$ under the relationship \eqref{eq:X_and_Z}, and note that $\mathcal{O}^\omega(\mathcal{Z})$ was denoted by $\mathcal{Z}^\omega_{{\rm PGL}_2,S}$ in the introduction. Allegretti and Kim \cite{AK} construct a map
$$
\mathbb{I}^\omega : \mathring{\mathcal{A}}_{\rm L}(S,\mathbb{Z}) \to \mathcal{O}^\omega(\mathcal{Z}),
$$
whose restriction to $\mathring{\mathcal{A}}_{{\rm SL}_2,S}(\mathbb{Z}^t)$ is the desired map $\wh{\mathbb{I}}^q$. Again, in practice, what is constructed is a map
$$
\mathbb{I}^\omega_T : \mathring{\mathcal{A}}_{\rm L}(S, \mathbb{Z}) \hookrightarrow \mathcal{Z}^\omega_T \cap \wh{\mathcal{Z}}^\omega_T \subset \wh{\mathcal{Z}}^\omega_T
$$
 for each ideal triangulation $T$, satisfying $\Theta^\omega_{TT'} \circ \mathbb{I}^\omega_{T'} = \mathbb{I}^\omega_T$ for each pair of ideal triangulations $T,T'$. We will now be describing this map $\mathbb{I}^\omega_T$ constructed in \cite{AK}.

\subsection{Bonahon-Wong construction}
\label{subsec:BW}

A key of the construction of $\mathbb{I}^\omega_T(\ell)$ is for the case when $\ell$ consists of a single non-peripheral good loop with weight $1$. For this, Allegretti and Kim used Bonahon and Wong's map \cite{BW} from the `skein algebra of $S$' to $\mathcal{O}^\omega(\mathcal{Z})$.
\begin{definition}[from \cite{BW}]
\label{def:framed_link}
A \ul{\em link} in a $3$-manifold $M$ is an unoriented $1$-dimensional submanifold possibily with boundary and several connected components. A \ul{\em framing} on a link is a continuous choice of an element $v_p$ of the tangent space $T_pM$ to $M$ at each point $p$ on the link, such that $v_p$ is not in the tangent space to the link. A \ul{\em framed link} is a link together with a framing on it.

\vs

For a decorated surface $S$ and a ring $R$, the \ul{\em framed link algebra} $\mathcal{K}(S;R)$ over $R$ is the free $R$-module with a basis consisting of all isotopy classes of framed links $K$ in $S\times [0,1]$ satisfying the following conditions:
\begin{enumerate}
\item[\rm 1)] $\partial K = K \cap \partial(S\times [0,1])$ is a finite subset of $(\partial S) \times [0,1]$,

\item[\rm 2)] at each point $p$ of $K$, the framing $v_p \in T_p(S\times [0,1])$ at $p$ is \ul{\em upward vertical}, i.e. it is parallel to the $[0,1]$ factor and points toward $1 \in [0,1]$,

\item[\rm 3)] for every boundary arc $k$ of $S$, the points of $\partial K$ in $k \times [0,1]$ have different elevations, where the \ul{\em elevation} of a point of $S\times [0,1]$ is its $[0,1]$-coordinate,
\end{enumerate}
and the isotopy of framed links must respect all three conditions 1), 2), and 3).

\vs

The multiplication of $\mathcal{K}(S;R)$ is given by the `superposition operation'; for two basis elements $K, K'$, define $K K' := K''$, where $K''$ is the disjoint union of $K$ rescaled to live in $S\times [0,\frac{1}{2}]$ and $K'$ rescaled to live in $S\times [\frac{1}{2},1]$.
\end{definition}
From the requirement that isotopies should respect condition 3), we see that for each boundary arc $k$ of $S$, the ordering on the set $\partial K \cap (k\times [0,1])$ induced by their elevations is well-defined, i.e. preserved by isotopies.

\vs

When dealing with a framed link $K$ in $S\times [0,1]$ satisfying the above conditions 1), 2), and 3), Bonahon and Wong project $K$ down to a diagram on the surface $S$ in the following way. Choose an arbitrary orientation on each boundary arc $k$ of $S$. Through the projection $P: S\times [0,1] \to S$, the set $\partial K \cap (k \times [0,1])$ projects to $P(\partial K) \cap k$, and we may assume that, after an isotopy if necessary, this projection is injective on $\partial K \cap (k \times [0,1])$. The orientation on $k$ induces an ordering on the set $P(\partial K) \cap k$, hence on $\partial K \cap (k\times [0,1])$. But the set $\partial K \cap (k\times [0,1])$ also has an ordering induced by the elevation of its members; we may isotope $K$ such that these two orderings coincide. The projected diagram of $K$ is basically $P(K)$, with the following enhancement; we may assume by using isotopy that the map $K \to P(K)$ is at most $2$-to-$1$. Above a small neighborhood of a point of $P(K)$ with two inverse images in $K$, there are two little pieces of $K$. When we draw $P(K)$, for such `crossing' we indicate which segment has higher elevation, as in the left (i.e. $K_1$) of Fig.\ref{fig:Kauffman_triple}; the `broken' segment sits below the `unbroken' segment. This way we may identify a basis element of the framed link algebra $\mathcal{K}(S;R)$ by a projected diagram in $S$ with crossings. Meanwhile, there is a natural way of `resolving' the crossings, appearing in the theory of framed links.
\begin{definition}[from \cite{BW}]
A triple of basis elements $(K_1,K_0,K_\infty)$ of $\mathcal{K}(S;R)$ is a \ul{\em Kauffman triple} if they are identical except above a small disc in $S$, where their projected diagrams are as in Fig.\ref{fig:Kauffman_triple}. 

\vs

For a ring $R$ and $A$ a formal symbol or a complex number, the \ul{\em (Kauffman) skein algebra} $\mathcal{S}^A(S;R)$ is the quotient of the framed link algebra $\mathcal{K}(S;R[A,A^{-1}])$ by the two-sided ideal generated by $K_1 - A^{-1} K_0 - A K_\infty$, where $(K_1,K_0,K_\infty)$ runs over all Kauffman triples.

\vs

An element $[K] \in \mathcal{S}^A(S;R)$ represented by a single basis element $K$ of $\mathcal{K}(S;R[A,A^{-1}])$ is called a \ul{\em skein} in $S$. The relations
$$
[K_1] = A^{-1} [K_0] + A [K_\infty]
$$
are called the \ul{\em skein relations}.

\vs

Let $\mathcal{S}^A(S) := \mathcal{S}^A(S;\mathbb{C})$, and consider $\omega$ and $q$ as being elements of $\mathbb{C}^*$ with $q= \omega^4$.

\vs

A \ul{\em state} for a skein $[K] \in \mathcal{S}^A(S)$ is the choice of signs at each boundary point of $K$, i.e. a map $\partial K \to \{+,-\}$. A \ul{\em stated skein} is a skein together with a state. Let $\mathcal{S}^A_{\rm s}(S)$ be the algebra consisting of linear combinations of stated skeins.
\end{definition}

\begin{figure}
\hspace{0mm} \includegraphics[width=80mm]{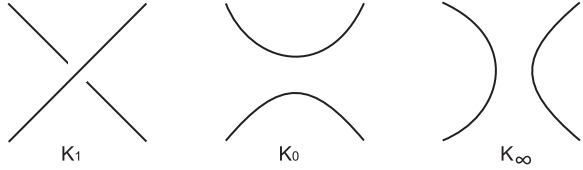}
\caption{Kauffman triple $(K_1,K_0,K_\infty)$ over a small disc}
\label{fig:Kauffman_triple}
\end{figure}

Multiplication in $\mathcal{S}^A_{\rm s}(S)$ follows the multiplication rule of $\mathcal{S}^A(S)$ in an obvious manner. The main result of \cite{BW} is the construction of an algebra homomorphism
$$
{\rm Tr}^\omega : S^{1/\omega^2}_{\rm s}(S) \to \mathcal{O}^\omega(\mathcal{Z}),
$$
called the \ul{\em quantum trace map}, which is a map from the stated skein algebra to the square-root version of the ring of regular functions on the quantum Teichm\"uller space. What is actually constructed is the map
$$
{\rm Tr}^\omega_T : S^{1/\omega^2}_{\rm s}(S) \to \mathcal{Z}^\omega_T  \cap \wh{\mathcal{Z}}^\omega_T
$$
for each ideal triangulation $T$ of $S$, which are mutually compatible in the sense that $\Theta^\omega_{TT'} \circ {\rm Tr}^\omega_T = {\rm Tr}^\omega_{T'}$ holds for each pair of ideal triangulation $T,T'$. The Bonahon-Wong quantum trace map ${\rm Tr}^\omega$ is in fact a unique map satisfying some natural properties including the gluing rule for gluing decorated surfaces along boundaries. These defining properties allow us to express the value of ${\rm Tr}^\omega$ in terms of the values of ${\rm Tr}^\omega$ for smaller surfaces obtained by cutting $S$ along ideal arcs of a triangulation of $S$. After cutting along all ideal arcs of a triangulation, each of the resulting smaller surfaces is a `triangle', i.e. a decorated surface of genus $0$ with one boundary component having three marked points on the boundary. In the end, ${\rm Tr}^\omega$ for each triangle is all we need to know.

\vs

Instead of writing down these defining properties of ${\rm Tr}^\omega$, let us just describe the state-sum formula for ${\rm Tr}^\omega$ which is a consequence of these properties. Let $T$ be an ideal triangulation of $S$, and let $[K,s] \in \mathcal{S}^{1/\omega^2}_{\rm s}(S)$ be a stated skein, represented by a framed link $K$ with a state $s : \partial K \to \{+,-\}$. Let $\wh{T}$ be the \ul{\em split ideal triangulation} of $S$ corresponding to $T$, which is obtained by replacing each constituent arc of $T$ by two parallel copies of it, forming a `biangle'. So the number of constituent arcs of $\wh{T}$ is twice that of $T$, and $\wh{T}$ divides $S$ into triangles and biangles. Let $B_1,\ldots,B_{|T|}$ be the biangles of $\wh{T}$ (where $|T|$ is the number of constituent arcs of $T$), and $t_1,\ldots,t_m$ be the ideal triangles of $\wh{T}$. For each constituent arc of $T$, choose an arbitrary orientation, inducing corresponding orientations on the constituent arcs of $\wh{T}$. 
\begin{lemma}[\cite{BW}]
\label{lem:good_position}
In the above situation, one can isotope $K$ such that $K$ is in a \ul{\em good position}, meaning that the following three conditions hold:
\begin{enumerate}
\item[\rm 1)] for each contituent arc $e$ of $\wh{T}$, the link $K$ meets $e\times [0,1]$ transversally,

\item[\rm 2)] for each constituent arc $e$ of $\wh{T}$, the ordering on the set $K \cap (e\times [0,1])$ induced by the elevation of its members coincides with that induced by the orientation on $e$ \footnote{This condition it not part of the definition of Bonahon-Wong's good position, but part of their diagram convention.},

\item[\rm 3)] for each ideal triangle $t_j$ of $\wh{T}$, the set $K \cap (t_j \times [0,1])$ has finitely many connected components, each connected component connects two distinct components of $(\partial t_j) \times [0,1]$ and has upward vertical framing, each connected component is at a constant elevation, and distinct connected components have distinct elevations.
\end{enumerate}
\end{lemma}
Assume now that $K$ is in a good position. For an ideal triangle $t_j$, let $k_1,k_2,\ldots,k_l$ be the connected components of $K\cap (t_j \times [0,1])$, arranged in the increasing order of elevations. Regard $t_j$ as a decorated surface, with the boundary arcs $e_{j1},e_{j2},e_{j3}$, i.e. the sides of $t_j$, appearing in this order clockwise. Suppose some $k_i$ connects $e_{j1}$ and $e_{j2}$. Let $s_i: \partial k_i \to \{+,-\}$ be a state for $k_i$, given by the two signs $\sigma_1,\sigma_2$ at $k_i \cap e_{j1}, k_i\cap e_{j2}$. Then ${\rm Tr}^\omega_{t_j}([k_i,s_i])$ is defined as the following element of the triangle square-root algebra $\mathcal{Z}^\omega_{t_j}$, which is generated by $\wh{Z}_{j1},\wh{Z}_{j2},\wh{Z}_{j3}$:
\begin{align}
\label{eq:BW_formula_triangle_factor_one_segment}
{\rm Tr}^\omega_{t_j}([k_i,s_i]) := \left\{ \begin{array}{ll} 0 & \mbox{if $\sigma_1 = -$, $\sigma_2 = +$,} \\ \omega^{- \sigma_1 \sigma_2} \wh{Z}_{j1}^{\sigma_1} \wh{Z}_{j2}^{\sigma_2} & \mbox{otherwise}, \end{array} \right.
\end{align}
where each sign $+,-$ appearing in the exponent is understood as the number $-1,1$, respectively. If $k_i$ connects $e_{j2}$ and $e_{j3}$, then replace each subscript $1,2$ in the above right hand side by $2,3$. If $k_i$ connected $e_{j3}$ and $e_{j1}$, then replace each subscript $1,2$ in the above by $3,1$. For a state $\sigma_j$ for the skein $[K\cap (t_j\times [0,1])]$ in $t_j$, note that the equality
\begin{align}
\label{eq:BW_formula_triangle_factor_detail}
{\rm Tr}^\omega_{t_j}([K\cap (t_j\times [0,1]), \sigma_j]) = {\rm Tr}^\omega_{t_j}([k_1,s_1]) \, {\rm Tr}^\omega_{t_j}([k_2,s_2]) \, \cdots \, {\rm Tr}^\omega_{t_j}([k_l,s_l]),
\end{align}
where $s_1,\ldots,s_l$ constitutes $\sigma_j$, follows from the requirement that ${\rm Tr}^\omega_{t_j}$ be an algebra homomorphism. The tensor product
\begin{align}
\label{eq:BW_formula_triangle_factor}
\bigotimes_{j=1}^m {\rm Tr}^\omega_{t_j}([K\cap (t_j\times [0,1]), \sigma_j])
\end{align}
is an element of $\bigotimes_{j=1}^m \mathcal{Z}^\omega_{t_j}$. From the `gluing' property of ${\rm Tr}^\omega$, for each state $s$ of $K$ we get
\begin{align}
\label{eq:BW_formula}
{\rm Tr}^\omega_T([K,s]) = \sum_{\sigma_j,\tau_i} \left( \prod_{i=1}^{|T|} {\rm Tr}^\omega_{B_i}([K\cap(B_i \times [0,1]),\tau_i]) \, \bigotimes_{j=1}^m {\rm Tr}^\omega_{t_j}([K\cap (t_j\times [0,1]), \sigma_j]) \right).
\end{align}
Here, $\tau_i$ is a state of $[K\cap (B_i \times [0,1])]$ in the biangle $B_i$, $\sigma_j$ is a state of $[K\cap (t_j \times [0,1])]$ in the triangle $t_j$, and the whole sum is taken over all $\sigma_j,\tau_i$ that are mutually compatible at every constituent arc of $\wh{T}$ and compatible with $s$ at $\partial K$. It remains to describe ${\rm Tr}^\omega_{B_i}([K\cap (B_i \times [0,1]),\tau_i])$, which is an element of $\mathbb{Z}[\omega,\omega^{-1}]$. Viewing the biangle $B_i$ as a decorated surface, $[K\cap (B_i \times [0,1])]$ is a skein in $B_i$, hence can be written as linear combination of skeins without crossings, with the help of skein relations. Now, suppose $[K']$ is a skein in a biangle $B$ without crossings, and let $\tau$ be a state for $[K']$. Each connected component of $[K']$ together with its state induced by $\tau$ is of one of the forms illustrated in Fig.\ref{fig:segments_in_a_biangle}; let each of $a^{\sigma_1}_{\sigma_2}$, $b^{\sigma_1}_{\sigma_2}$, and $c^{\sigma_1}_{\sigma_2}$ be the number of components of the type on the left, the middle, and the right of Fig.\ref{fig:segments_in_a_biangle}, respectively. Let $d$ be the number of closed components of $[K']$, i.e. the contractible circular components. Then ${\rm Tr}^\omega_B([K',\tau])$ is defined as:
\begin{align}
\label{eq:BW_biangle_formula}
& {\rm Tr}^\omega_B([K',\tau]) \\
\nonumber
& = \left\{
\begin{array}{ll}
0 & \mbox{if one of $a^+_-,a^-_+,b^+_+,b^-_-,c^+_+,c^-_-$ is nonzero}, \\
(-1)^{b^+_- + c^-_+} \omega^{-(5b^+_-+b^-_+) + (5c^-_+ + c^+_-)} (-\omega^4 - \omega^{-4})^d, &\mbox{otherwise}.
\end{array}
\right.
\end{align}
The value of a stated skein in a biangle $B$ under ${\rm Tr}^\omega_B$ can also be described without resolving the skein using skein relations, but instead first express it as gluing of simpler skeins living in parallel biangles; for easy cases, see \cite{BW}. For later reference, note that it is only these biangle values ${\rm Tr}^\omega_{B_i}([K\cap (B_i \times [0,1]),\tau_i])$ where a minus sign can appear in the right hand side of the sum in eq.\eqref{eq:BW_formula}, and hence the key point of the proof of positivity is to deal with these. We note that Gabella's construction \cite{Gabella} of an analog of quantum trace also has the part corresponding to these biangle values which is called the `R-matrix' there, but the positivity proof in \cite[\S6.4]{Gabella} does not provide an argument to fully deal with the possible minus signs appearing in these biangle values.

\vs

This finishes a description of the result of the Bonahon-Wong quantum trace map ${\rm Tr}^\omega_T([K,s])$. A priori, the expression in eq.\eqref{eq:BW_formula_triangle_factor}, hence each term in the right hand side of eq.\eqref{eq:BW_formula}, is an element of $\bigotimes_{j=1}^m \mathcal{Z}^\omega_{t_j}$. The mutual-compatibility condition for $\sigma_j,\tau_i$ and the condition for ${\rm Tr}^\omega_{B_i}([K\cap(B_i \times [0,1]), \tau_i])$ to be zero guarantee that each nonzero term in the right hand side of \eqref{eq:BW_formula} belongs to the subalgebra $\mathcal{Z}^\omega_T$, the Chekhov-Fock square-root algebra for $T$.

\begin{figure}
\hspace{00mm} \includegraphics[width=100mm]{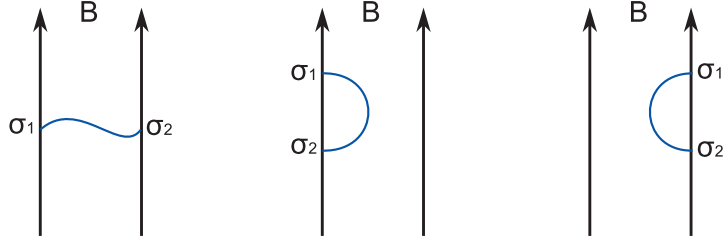}
\caption{segments in a biangle}
\label{fig:segments_in_a_biangle}
\end{figure}

\subsection{Allegretti-Kim construction}
\label{subsec:AK}

Now let's describe the Allegretti-Kim element $\mathbb{I}^\omega_T(\ell) \in \mathcal{Z}^\omega_T$, for an integral lamination $\ell \in \mathring{\mathcal{A}}_{\rm L}(S,\mathbb{Z})$. Represent $\ell$ by mutually non-homotopic constituent loops $\gamma_i$'s of $\ell$, with integer weight $k_i$ respectively. Let $k_i \ell_i$ be the integral lamination consisting of just one loop $\gamma_i$ with weight $k_i$. In particular, different $\gamma_i$'s are mutually non-intersecting. Write
\begin{align}
\label{eq:ell_as_sum}
\ell = \sum_i k_i \ell_i.
\end{align}
We define
\begin{align}
\label{eq:def_of_I_omega_as_product}
\mathbb{I}^\omega_T(\ell) := \prod_i \mathbb{I}^\omega_T(k_i \ell_i),
\end{align}
where the order of the product $\prod_i$ does not matter. We should now describe each factor $\mathbb{I}^\omega_T(k_i \ell_i)$. 

\vs

[Case 1] Suppose $\ell$ consists of just one non-peripheral good loop $\gamma$ with weight $1$. Lift $\gamma$ to a link in $S\times [0,1]$ at a constant elevation, and give it a constant upward vertical framing; we then obtain a framed link $\til{\gamma}$ in $S\times [0,1]$ satisfying all conditions 1), 2), and 3) of Def.\ref{def:framed_link}, thus yielding a skein $[\til{\gamma}]$ in $S$. Apply the Bonahon-Wong quantum trace ${\rm Tr}^\omega_T$ to this skein; no state is needed, as $\partial \til{\gamma}$ is empty. Define
\begin{align}
\label{eq:I_omega_Case_1}
\mathbb{I}^\omega_T(\ell) := {\rm Tr}^\omega_T([\til{\gamma},{\O}]) \in \mathcal{Z}^\omega_T.
\end{align}

\vs

[Case 2] Let $\ell$ be as above, and consider the integral lamination $k\ell$, consisting of one non-peripheral good loop $\gamma$ with weight $k\ge 1$. Define
\begin{align}
\label{eq:I_omega_Case_2}
\mathbb{I}^\omega_T(k\ell) := F_k(\mathbb{I}^\omega_T(\ell)),
\end{align}
where $F_k(x) \in \mathbb{Z}[x]$ is some version of the $k$-th Chebyshev polynomial, defined in our case by $F_0(x) =2$, $F_1(x)=x$ and the recursion relation
\begin{align}
\label{eq:F_k}
F_{k+1}(x) = F_k(x) \cdot x - F_{k-1}(x), \quad \forall k\ge 1.
\end{align}

\vs

[Case 3] Let $k\ell$ be an integral lamination consisting of just one peripheral good loop $\gamma$ with weight $k\in \mathbb{Z}$. As one travels once around along $\gamma$, let's say that one meets the ideal arcs $e_1,e_2,\ldots,e_r$ of $T$, in this order; in particular, $e_1,e_2,\ldots,e_r$ need not be mutually distinct. Define
\begin{align}
\label{eq:I_omega_Case_3}
\mathbb{I}^\omega_T(k\ell) := ( [\wh{Z}_{e_1} \wh{Z}_{e_2} \cdots \wh{Z}_{e_r} ] )^k,
\end{align}
where the bracket $[ \sim ]$ means the `Weyl ordering', i.e.
$$
\textstyle [\wh{Z}_{e_1} \wh{Z}_{e_2} \cdots \wh{Z}_{e_r}] := \omega^{(*)} \wh{Z}_{e_1} \wh{Z}_{e_2} \cdots \wh{Z}_{e_r}, \quad\mbox{where}\quad (*) = - \underset{1\le i<j\le r}{\sum} \varepsilon_{e_i e_j}.
$$

\vs

Note that each $k_i \ell_i$ appearing in the right hand side of eq.\eqref{eq:ell_as_sum} and \eqref{eq:def_of_I_omega_as_product} forms an integral lamination by itself, in an obvious manner; in particular, it consists of a single good loop with weight $k_i$. Hence $k_i \ell_i$ falls either to Case 2 or Case 3 above, for which we have defined the value $\mathbb{I}^\omega_T(k_i \ell_i)$. So, the formula \eqref{eq:def_of_I_omega_as_product} defines the value of $\mathbb{I}^\omega_T(\ell)$ for any integral lamination $\ell \in \mathring{\mathcal{A}}_{\rm L}(S,\mathbb{Z})$. A priori, $\mathbb{I}^\omega_T(\ell)$ is an element of the Chekhov-Fock square-root algebra $\mathcal{Z}^\omega_T$, i.e. a Laurent polynomial in the square-root generators $\wh{Z}_e$'s for $e\in T$ (in fact only the ideal arcs of $T$ appear) with coefficients in $\mathbb{Z}[\omega,\omega^{-1}]$, since each factor $\mathbb{I}^\omega_T(k_i\ell_i)$ is. It is proved in \cite{AK} that, in case $\ell$ belongs to $\mathring{\mathcal{A}}_{{\rm SL}_2,S}(\mathbb{Z}^t) \subset \mathring{\mathcal{A}}_{\rm L}(S,\mathbb{Z})$, i.e. is an even integral lamination, $\mathbb{I}^\omega_T(\ell)$ belongs to the Chekhov-Fock algebra $\mathcal{X}^q_T \subset \mathcal{Z}^\omega_T$, i.e. is a Laurent polynomial in the generators $\wh{X}_e = \wh{Z}_e^2$'s ($e\in T$) with coefficients in $\mathbb{Z}[q,q^{-1}]$, where $q = \omega^4$; we let
$$
\wh{\mathbb{I}}^q_T(\ell) := \mathbb{I}^\omega_T(\ell) \in \mathcal{X}^q_T, \qquad \forall \ell \in \mathring{\mathcal{A}}_{{\rm SL}_2,S}(\mathbb{Z}^t),
$$
which is really the sought-for duality map \eqref{eq:AK_map_for_T}, yielding \eqref{eq:AK_map}. The main result of the present paper is that, for $\ell \in \mathring{\mathcal{A}}_{\rm L}(S,\mathbb{Z})$, the coefficients of the Laurent polynomial expression of $\mathbb{I}^\omega_T(\ell)$ belong to $\mathbb{Z}_{\ge 0}[\omega,\omega^{-1}]$, i.e. the coefficients are Laurent polynomials in $\omega$ with {\em positive} integer coefficients. In particular, in case $\ell \in \mathring{\mathcal{A}}_{{\rm SL}_2,S}(\mathbb{Z}^t) \subset \mathring{\mathcal{A}}_{\rm L}(S,\mathbb{Z})$, the coefficients of the Laurent polynomial expression of $\wh{\mathbb{I}}^q_T(\ell) = \mathbb{I}^\omega_T(\ell)$ belongs to $\mathbb{Z}_{\ge 0}[q,q^{-1}]$, i.e. the coefficients are Laurent polynomials in $q$ with {\em positive} integer coefficients.

\subsection{Proof of the main result, using the ordering problem}

For convenience, let us introduce some new notations.
\begin{definition}[the positive semi-rings]
Denote by $\mathbb{Z}_{\ge 0}[\omega,\omega^{-1}]$ the semi-ring of all Laurent polynomials in $\omega$ with non-negative integer coefficients. Define $\mathbb{Z}_{\ge 0}[q,q^{-1}]$ similarly.

\vs

Let $(\mathcal{Z}^\omega_T)^+ \subset \mathcal{Z}^\omega_T$ be the semi-ring consisting of all elements of $\mathcal{Z}^\omega_T$ that can be written as Laurent polynomial in $\wh{Z}_e$'s ($e\in T$) with coefficients in $\mathbb{Z}_{\ge 0}[\omega,\omega^{-1}]$. 

\vs

Let $(\mathcal{X}^q_T)^+ \subset \mathcal{X}^q_T$ be the semi-ring of all elements of $\mathcal{X}^q_T$ that can be written as Laurent polynomial in $\wh{X}_e$'s ($e\in T$) with coefficients in $\mathbb{Z}_{\ge 0}[q,q^{-1}]$.
\end{definition}
A semi-ring means that it is closed under addition and multiplication, but not necessarily by subtraction. It is easy to verify that $\mathbb{Z}_{\ge 0}[\omega,\omega^{-1}]$, $\mathbb{Z}_{\ge 0}[q,q^{-1}]$, $(\mathcal{Z}^\omega_T)^+$, $(\mathcal{X}^q_T)^+$ are semi-rings containing $0$. In view of the inclusion $\mathcal{Z}^\omega_T \subset \mathcal{X}^q_T$, one observes
\begin{align}
\label{eq:Z_and_X_positive}
(\mathcal{X}^q_T)^+ = \mathcal{X}^q_T \cap (\mathcal{Z}^\omega_T)^+.
\end{align}

\vs

{\bf Equivalent form of the main theorem (Thm.\ref{thm:main}).} {\em For any even integral lamination $\ell \in \mathring{\mathcal{A}}_{{\rm SL}_2,S}(\mathbb{Z}^t)$, we have $\wh{\mathbb{I}}^q_T(\ell) \in (\mathcal{X}^q_T)^+$.}

\vs

We shall prove the following, which is stronger than the main theorem.
\begin{theorem}[square-root version of main theorem]
\label{thm:main2}
For each integral lamination $\ell \in \mathring{\mathcal{A}}_{\rm L}(S,\mathbb{Z})$, we have $\mathbb{I}^\omega_T(\ell) \in (\mathcal{Z}^\omega_T)^+$.
\end{theorem}
Indeed, Thm.\ref{thm:main2} implies Thm.\ref{thm:main}; for $\ell \in \mathring{\mathcal{A}}_{{\rm SL}_2,S}(\mathbb{Z}^t) \subset \mathring{\mathcal{A}}_{\rm S}(S,\mathbb{Z})$ we have $\wh{\mathbb{I}}^q_T(\ell) = \mathbb{I}^\omega_T(\ell) \in \mathcal{X}^q_T \subset \mathcal{Z}^\omega_T$, so Thm.4.9 yields $\wh{\mathbb{I}}^q_T(\ell) \in \mathcal{X}^q_T \cap (\mathcal{Z}^\omega_T)^+ = (\mathcal{X}^q_T)^+$ (see eq.\eqref{eq:Z_and_X_positive}). So it suffices to prove Thm.\ref{thm:main2}.

\vs

Using the ordering problem we solved (Thm.\ref{thm:ordering_problem}), we can first show the following, which is the most crucial step in our proof of Thm.\ref{thm:main2}:

\vs

{\bf Equivalent form of Thm.\ref{thm:Laurent_positivity_of_BW}} (Laurent positivity for a single non-peripheral loop with weight $1$)
{\em 
For an integral lamination $\ell \in \mathring{\mathcal{A}}_{\rm L}(S,\mathbb{Z})$ consisting of a single non-peripheral good loop $\gamma$ with weight $1$, the Laurent positivity holds, i.e. $\mathbb{I}^\omega_T(\ell)$ belongs to $(\mathcal{Z}^\omega_T)^+$.}

\vs

\ul{\it Proof.} This falls into the Case 1 of the previous subsection, hence $\mathbb{I}^\omega_T(\ell)$ is defined by eq.\eqref{eq:I_omega_Case_1}, i.e. by the Bonahon-Wong quantum trace ${\rm Tr}^\omega_T$ applied to the skein $[K]$ obtained by lifting $\gamma$ to a link $K = \til{\gamma}$ in $S\times [0,1]$ with constant elevation, with constant upward vertical framing. Since this $K$ has no boundary, we need no state $s$ for it, so let $s={\O}$. Then $\mathbb{I}^\omega_T(\ell) = {\rm Tr}^\omega_T([K,{\O}])$, by eq.\eqref{eq:I_omega_Case_1}. The quantum trace ${\rm Tr}^\omega_T$ is given by the sum formula \eqref{eq:BW_formula}, so it suffices to show that each term in the right hand side of \eqref{eq:BW_formula} belongs to $(\mathcal{Z}^\omega_T)^+$, for our $K$ and $s={\O}$. However, according to the recipe of Bonahon-Wong, we must first isotope $K$ into a good position, in the sense as in Lem.\ref{lem:good_position}. There are many ways of doing so, but we choose a specific one as follows. By Thm.\ref{thm:ordering_problem}, one can choose an ordering on the set of all loop segments of the loop $\gamma$ in each ideal triangle, so that these orderings are `compatible' at every ideal arc. That is, at each ideal arc that the loop $\gamma$ intersects, the two orderings on the set of junctures (i.e. intersection points) on this arc induced from the triangle-orderings on the two triangles having this arc as a side coincide with each other. 

\vs

Choose any orientation on each constituent arc of $T$. Let $\wh{T}$ be the split ideal triangulation associated to the triangulation $T$, and let the constituent arcs of $\wh{T}$ inherit the orientations from those on the constituent arcs of $T$. Let $t$ be an ideal triangle of $T$, and denote also by $t$ the corresponding triangle in $\wh{T}$. At the moment, the segments of $K$ over the triangle $t$ of $\wh{T}$ are all at some same constant elevation. Isotope $K$ so that each segment over $t$ is at some different constant elevation, such that the ordering on the set of all segments of $K$ over $t$ induced by their elevations coincide with the chosen triangle-ordering on $t$, i.e. with the ordering on the loop segments of $\gamma$ in $t$ obtained from Thm.\ref{thm:ordering_problem}; keep in mind that the loop segments of $\gamma$ in this triangle naturally correspond to the segments of $K$ over $t$. When isotoping  $K$ this way, we allow the segments of $K$ over the biangles adjacent to $t$ to deform accordingly; in particular, the segments over biangles need not be at constant elevations. When isotoping, we keep the upward vertical framing all the time. Apply such isotopy for each ideal triangle $t$ of $\wh{T}$. Finally, for each ideal arc $e$ of $\wh{T}$, `drag around' the points $K\cap (e\times [0,1])$ by an isotopy on $K$, keeping the elevations of these points, keeping the constant elevation of the segments of $K$ over the triangle having $e$ as a side, and such that the ordering on $K\cap (e\times [0,1])$ induced by the elevation coincides with that induced by the chosen orientation on $e$. Then $K$ satisfies all conditions of Lem.\ref{lem:good_position}, i.e. is in a good position. From the compatibility at each ideal arc of $T$ of the triangle-orderings for ideal triangles of $T$ we obtained from Thm.\ref{thm:ordering_problem}, we observe that the diagram of $K$ over each biangle consists of disjoint parallel lines, each line being of the form in the left of Fig.\ref{fig:segments_in_a_biangle}.

\vs

For each term in the right hand side of \eqref{eq:BW_formula}, consider the tensor product part, i.e. eq.\eqref{eq:BW_formula_triangle_factor}; each tensor factor ${\rm Tr}^\omega_{t_j}([K\cap (t_j \times [0,1]), \sigma])$ is as in eq.\eqref{eq:BW_formula_triangle_factor_detail}, i.e. is a product of factors of the form ${\rm Tr}^\omega_{t_j}([k_i,s_i])$. By the formula eq.\eqref{eq:BW_formula_triangle_factor_one_segment}, we observe that ${\rm Tr}^\omega_{t_j}([k_i,s_i])$ is either $0$ or a Laurent monomial in $\wh{Z}_{j1}, \wh{Z}_{j2},\wh{Z}_{j3}$, i.e. generators of the triangle square-root algebra $\mathcal{Z}^\omega_{t_j}$, with coefficient being some power of $\omega$. So, ${\rm Tr}^\omega_{t_j}([K\cap (t_j \times [0,1]), \sigma])$ is also either $0$ or a Laurent monomial in the generators of the triangle square-root algebra $\mathcal{Z}^\omega_{t_j}$ with coefficient being a power of $\omega$. In turn, $\otimes_{j=1}^m {\rm Tr}^\omega_{t_j}([K\cap (t_j \times [0,1]), \sigma])$ is either $0$ or a Laurent monomial in the generators of $\otimes_{j=1}^m \mathcal{Z}^\omega_{t_j}$, i.e. in $\wh{Z}_{j1}, \wh{Z}_{j2},\wh{Z}_{j3}$ with $j=1,2,\ldots,m$, with coefficients being a power of $\omega$.

\vs

Now let's look at the `coefficient' part $\prod_{i=1}^{|T|} {\rm Tr}^\omega_{B_i}([K\cap (B_i\times[0,1]),\tau_i])$ of a summand term in the right hand side of \eqref{eq:BW_formula}. Each ${\rm Tr}^\omega_{B_i}([K\cap(B_i\times[0,1]),\tau_i])$ is given by the formula  \eqref{eq:BW_biangle_formula}. As mentioned, in our case, the diagram of $K$ over $B_i$ consists of disjoint parallel lines, i.e. disjoint union of figures in the left of Fig.\ref{fig:segments_in_a_biangle}. So, by definition (see \S\ref{subsec:BW}, right above eq.\eqref{eq:BW_biangle_formula}), all numbers $b^{\sigma_1}_{\sigma_2}$, $c^{\sigma_1}_{\sigma_2}$, $d$ for ${\rm Tr}^\omega_{B_i}([K\cap(B_i\times[0,1]),\tau_i])$ are zero. Hence, in view of the formula eq.\eqref{eq:BW_biangle_formula}, ${\rm Tr}^\omega_{B_i}([K\cap(B_i\times[0,1]),\tau_i])$ is either $0$ or $1$. Hence each nonzero summand term of the right hand side of \eqref{eq:BW_formula} is a Laurent monomial in the generators of $\otimes_{j=1}^m \mathcal{Z}^\omega_{t_j}$, i.e. in $\wh{Z}_{j1}, \wh{Z}_{j2},\wh{Z}_{j3}$ with $j=1,2,\ldots,m$, with coefficients being a power of $\omega$. As mentioned at the end of \S\ref{subsec:BW}, it is a Laurent monomial in the generators of the Chekhov-Fock square-root algebra $\mathcal{Z}^\omega_T$, i.e. in $\wh{Z}_e$'s ($e\in T$), now with coefficient being a power of $\omega$. Thus, the whole sum in the right hand side of \eqref{eq:BW_formula} is a Laurent polynomial in $\wh{Z}_e$'s ($e\in T$) with coefficients in $\mathbb{Z}_{\ge 0}[\omega,\omega^{-1}]$. So we proved $\mathbb{I}^\omega_T(\ell) = {\rm Tr}^\omega_T([K,{\O}]) \in (\mathcal{Z}^\omega_T)^+$, as desired. \qed

\vs

What about an integral lamination $k\ell \in \mathring{\mathcal{A}}_{\rm L}(S,\mathbb{Z})$ consisting of a non-peripheral loop $\gamma$ with weight $k\ge 1$? This falls into the Case 2 of the previous subsection, so by eq.\eqref{eq:I_omega_Case_2} we have $\mathbb{I}^\omega_T(k\ell) = F_k(\mathbb{I}^\omega_T(\ell))$, where $F_k(x)\in \mathbb{Z}[x]$ is the $k$-th Chebyshev polynomial. Since $F_1(x)=x$, i.e. $F_1$ is the identity function, the case $k=1$ follows from Thm.\ref{thm:Laurent_positivity_of_BW}, as $\mathbb{I}^\omega_T(1\ell) = F_1(\mathbb{I}^\omega_T(\ell))=\mathbb{I}^\omega_T(\ell) \in (\mathcal{Z}^\omega_T)^+$. However, for $k\ge 2$, it is not obvious. For example, $F_2(x) = x^2 - 2$, $F_3(x) = x^3 - 3x$, etc, so not all coefficients of $F_k$ are positive. So the Laurent positivity of $\mathbb{I}^\omega_T(\ell)$ does not immediately imply that of $F_k(\mathbb{I}^\omega_T(\ell))$, and needs a proof.

\begin{proposition}[Laurent positivity for a single non-peripheral loop with weight $k$; \cite{AK1}]
\label{prop:positivity_of_non-peripheral_k}
Suppose $\ell \in \mathring{\mathcal{A}}_{\rm L}(S,\mathbb{Z})$ consists of a single non-peripheral good loop $\gamma$ with weight $1$. Denote by $k\ell$ the integer-weight lamination consisting of one loop $\gamma$ with weight $k\ge 1$.

\vs

Then $\mathbb{I}^\omega_T(\ell) \in (\mathcal{Z}^\omega_T)^+$ implies $\mathbb{I}^\omega_T(k\ell) \in (\mathcal{Z}^\omega_T)^+$.
\end{proposition}
Allegretti and Kim proved this by induction on $k$, using basic properties of $F_k$ and $\mathbb{I}^\omega_T(\ell)$. Their proof, which is quite elementary and spans a few pages, is contained in \S3.2 of the first arXiv version \cite{AK1} of their paper, but is omitted in the final published version \cite{AK}, because they did not prove the positivity for a single non-peripheral loop with weight $1$, i.e. $\mathbb{I}^\omega_T(\ell) \in (\mathcal{Z}^\omega_T)^+$, which is only done in the present paper!

\vs

From this Prop.\ref{prop:positivity_of_non-peripheral_k}, together with Thm.\ref{thm:Laurent_positivity_of_BW} which we proved using our topological result Thm.\ref{thm:ordering_problem}, we now know that $\mathbb{I}^\omega_T(\ell) \in (\mathcal{Z}^\omega_T)^+$ holds for any integer-weight lamination $\ell \in \mathring{\mathcal{A}}_{\rm L}(S,\mathbb{Z})$ consisting of a single non-peripheral good loop $\gamma$ with any positive integer weight. 

\vs

We can finally finish a proof of our stronger version main theorem (Thm.\ref{thm:main2}), which asserts $\mathbb{I}^\omega_T(\ell) \in (\mathcal{Z}^\omega_T)^+$ for {\em any} $\ell \in \mathring{\mathcal{A}}_{\rm L}(S,\mathbb{Z})$.

\vs

\ul{\it Proof of Thm.\ref{thm:main2}.} Let $\ell \in \mathring{\mathcal{A}}_{\rm L}(S,\mathbb{Z})$ be any integer-weight lamination. Recall \S\ref{subsec:AK} for the definition of $\mathbb{I}^\omega_T(\ell)$. Let $\gamma_i$'s be the mutually non-homotopic constituent curves of $\ell$, with weight $k_i$, so that eq.\eqref{eq:ell_as_sum} holds. Let $k_i \ell_i$ be the lamination consisting just of $\gamma_i$ with weight $k_i$. Then $\mathbb{I}^\omega_T(\ell)$ is defined as $\prod_i \mathbb{I}^\omega_T(k_i \ell_i)$ (eq.\eqref{eq:def_of_I_omega_as_product}). Since $(\mathcal{Z}^\omega_T)^+$ is closed under multiplication, in order to show $\mathbb{I}^\omega_T(\ell) \in (\mathcal{Z}^\omega_T)^+$ it suffices to show that each factor $\mathbb{I}^\omega_T(k_i\ell_i)$ belongs to $(\mathcal{Z}^\omega_T)^+$. In case $\gamma_i$ is a non-peripheral loop, we showed $\mathbb{I}^\omega_T(k_i\ell_i) \in (\mathcal{Z}^\omega_T)^+$, as mentioned above. In case $\gamma_i$ is a peripheral loop, $\mathbb{I}^\omega_T(k_i\ell_i)$ is defined by the formula eq.\eqref{eq:I_omega_Case_3}, hence is clearly in $(\mathcal{Z}^\omega_T)^+$. Done. \qed

\vs

Just to recall the readers, Thm.\ref{thm:main2} which we just proved implies the original main theorem Thm.\ref{thm:main}.

\section{Further research}
\label{sec:further_research}

The first future problem is to look for a proof shorter than the one in the present paper. The second is to look for a deeper meaning of the (quantum) Laurent positivity result we proved, such as a categorification. A priori, there is no reason to expect that Thm.\ref{thm:main} or Thm.\ref{thm:ordering_problem} should hold. The Allegretti-Kim elements are `universally Laurent' elements, hence are very interesting objects in the theory of quantum cluster algebras and quantum cluster varieties. But not so much has been discussed on the `universally positive Laurent' elements. This may have to do with physical quantization, namely with positive-definiteness of some self-adjoint operators, which is also related to the following topic.

\vs

The third is to try to build a relationship between the Allegretti-Kim quantum elements $\mathbb{I}^\omega(\ell)$ (or $\wh{\mathbb{I}}^q(\ell)$) for integral laminations and the quantum operators associated to closed curves, constructed by J. Teschner in a totally different manner \cite{T}. Teschner assigns some quantum expression to each closed curve using a recursive algorithm, where the counterpart of the `positivity' that we just proved for Allegretti-Kim elements is already built-in. Teschner's quantum expressions for loops are constructed using a slightly different framework of quantum Teichm\"uller theory, and it is not known whether they satisfy as many properties as satisfied by the Allegretti-Kim elements. If we believe that both Teschner's and Allegretti-Kim's expressions have enough naturality and canonicity, then it is natural to expect that they are related by some natural map bridging between the two frameworks of quantum Teichm\"uller theory.

\vs

Among the properties of Fock-Goncharov's classical functions $\mathbb{I}(\ell)$'s and Allegretti-Kim's quantum elements $\wh{\mathbb{I}}^q(\ell)$'s that we have not listed, an important one is about their structure coefficients. Since $\mathbb{I}(\ell)$'s ($\ell \in \mathcal{A}_{{\rm SL}_2,S}(\mathbb{Z}^t)$) form a $\mathbb{Q}$-basis of the ring $\mathcal{O}(\mathcal{X}_{{\rm PGL}_2,S})$ say in case when $S$ is a punctured surface, a product of two of them $\mathbb{I}(\ell) \mathbb{I}(\ell')$ can be written as a $\mathbb{Q}$-linear combination of $\mathbb{I}(\ell'')$'s for some finitely many $\ell'' \in \mathcal{A}_{{\rm SL}_2,S}(\mathbb{Z}^t)$. Surprisingly, all these coefficients of $\mathbb{I}(\ell'')$'s appearing in the combination, usually called the {\em structure coefficients},  are again {\em positive} integers \cite{Th}. Such phenomenon usually hints for a `monoidal categorification', often leading to a very fruitful mathematics. In the quantum version, it is still true that $\wh{\mathbb{I}}^q(\ell) \wh{\mathbb{I}}^q(\ell')$ is a finite linear combination of $\wh{\mathbb{I}}^q(\ell'')$ with coefficients in $\mathbb{Z}[q,q^{-1}]$. Now, we expect that the positivity of the structure coefficients also perseveres in the quantum version too:
\begin{conjecture}[the `structure-coefficient' positivity of Allegretti-Kim quantum elements]
\label{conj:structure-coefficient_positivity}
For any integral laminations $\ell,\ell'\in \mathcal{A}_{{\rm SL}_2,S}(\mathbb{Z}^t)$, one has
\begin{align}
\label{eq:second_positivity}
\wh{\mathbb{I}}^q(\ell) \wh{\mathbb{I}}^q(\ell') = \sum_{\ell'' \in \mathcal{A}_{{\rm SL}_2,S}(\mathbb{Z}^t)} c^q(\ell,\ell';\ell'') \, \wh{\mathbb{I}}^q(\ell''),
\end{align}
where only finitely many structure coefficients $c^q(\ell,\ell';\ell'')$ are non-zero, and they all lie in $\mathbb{Z}_{\ge 0}[q,q^{-1}]$, i.e. are Laurent polynomials in $q$ with \ul{\em positive} integral coefficients.
\end{conjecture}
Ideally, a slightly stronger version with $\ell,\ell',\ell'' \in \mathcal{A}_{\rm L}(S,\mathbb{Z})$ and $\wh{\mathbb{I}}^q$ replaced by $\mathbb{I}^\omega$ is expected to hold; the statement would be that the coefficients $c^\omega(\ell,\ell';\ell'')$ lie in $\mathbb{Z}_{\ge 0}[\omega,\omega^{-1}]$. Conjecture \ref{conj:structure-coefficient_positivity} or its stronger version then would hint to some more complicated version of monoidal categorification. We hope that, when investigating \eqref{eq:second_positivity}, it probably helps much if we know that each $\wh{\mathbb{I}}^q(\ell), \wh{\mathbb{I}}^q(\ell'), \wh{\mathbb{I}}^q(\ell'')$ are `Laurent positive' in the sense we proved in the present paper. We note that the above conjecture already appeared in several papers, including \cite{FG06} \cite{AK} \cite{Th} \cite{Le}. In particular, Le \cite{Le} investigated the role of the Chebyshev polynomials related to this structure coefficient positivity\footnote{pointed out to the third author by Dylan Allegretti.}.

\end{document}